\documentclass[a4paper]{amsart}
\usepackage{amssymb}
\usepackage{amscd}
\usepackage{amsthm}
\usepackage{amsmath}
\usepackage{mathtools}
\usepackage{latexsym}
\usepackage[all]{xy}
\usepackage[utf8]{inputenc}
\usepackage{enumitem}

\setlist[enumerate]{label={(\roman*)}}

\theoremstyle{plain}
\newtheorem{theorem}{Theorem}
\newtheorem{corollary}[theorem]{Corollary}
\newtheorem{lemma}[theorem]{Lemma}
\newtheorem{proposition}[theorem]{Proposition}

\theoremstyle{definition}
\newtheorem{definition}[theorem]{Definition}
\newtheorem{example}[theorem]{Example}
\newtheorem{construction}[theorem]{Construction}

\theoremstyle{remark}

\newtheorem{remark}{Remark}
\newtheorem{acknowledgment}[theorem]{Acknowledgment}
\numberwithin{theorem}{section}

\newcommand{\mSpec}[1]{\operatorname{mSpec}(#1)}

\newcommand{\Filt}[1]{\operatorname{Filt}(#1)}
\newcommand{\card}{\mbox{\rm{card\,}}}

\newcommand{\ssum}{\mbox{\rm{sum\,}}}
\newcommand{\Sum}{\mbox{\rm{Sum\,}}}
\newcommand{\add}{\mbox{\rm{add\,}}}
\newcommand{\Add}{\mbox{\rm{Add\,}}}

\newcommand{\Gen}{\mbox{\rm{Gen\,}}}

\newcommand{\End}{\mbox{\rm{End\,}}}

\newcommand{\Soc}[2]{\mbox{\rm{Soc}}_{#1}(#2)}

\newcommand{\im}{\mbox{\rm{Im\,}}}

\newcommand{\Hom}[3]{\operatorname{Hom}_{#1}(#2,#3)}
\newcommand{\Ext}[4]{\operatorname{Ext}^{#1}_{#2}(#3,#4)}
\newcommand{\Tor}[4]{\mbox{\rm{Tor}}_{#1}^{#2}(#3,#4)}
\newcommand{\Ctrtor}[4]{\operatorname{Ctrtor}_{#1}^{#2}(#3,#4)}

\newcommand{\rfmod}[1]{\mbox{\rm{mod}--}{#1}}
\newcommand{\rmod}[1]{\mbox{\rm{Mod}--}{#1}}
\newcommand{\lfmod}[1]{{#1}\mbox{--\rm{mod}}}
\newcommand{\lmod}[1]{{#1}\mbox{--\rm{Mod}}}
\newcommand{\ModR}{\text{Mod-}R}

\newcommand{\ldiscr}[1]{{#1}\mbox{--\rm{Discr}}}
\newcommand{\rcontra}[1]{{\mathrm{Contra\text{\rm--}}}{#1}}
\newcommand{\rfcontra}[1]{\mbox{\rm{contra}--}{#1}}

\newcommand{\pd}[1]{\mbox{\rm{proj.dim\,}}#1}

\newcommand{\Ker}[1]{\mbox{\rm{Ker}}(#1)}

\begin{document}

\title{Closure properties of $\protect\varinjlim\mathcal C$}

\author{Leonid Positselski}

\address[Leonid Positselski]{%
Institute of Mathematics, Czech Academy of Sciences \\
\v Zitn\'a~25, 115~67 Prague~1 \\ Czech Republic}
\email{positselski@math.cas.cz}

\author{Pavel P\v{r}\'{\i}hoda}
\address[Pavel P\v{r}\'{\i}hoda]{Charles University, Faculty of Mathematics
and Physics, Department of Algebra \\
Sokolovsk\'{a} 83, 186 75 Prague 8, Czech Republic}
\email{prihoda@karlin.mff.cuni.cz}

\author{Jan Trlifaj}
\address[Jan Trlifaj]{Charles University, Faculty of Mathematics
and Physics, Department of Algebra \\
Sokolovsk\'{a} 83, 186 75 Prague 8, Czech Republic}
\email{trlifaj@karlin.mff.cuni.cz}

	\begin{abstract} Let $\mathcal C$ be a class of modules and $\mathcal L = \varinjlim \mathcal C$ the class of all direct limits of modules from $\mathcal C$. The class $\mathcal L$ is well understood when $\mathcal C$ consists of finitely presented modules: $\mathcal L$ then enjoys various closure properties. Our first goal here is to study the closure properties of $\mathcal L$ in the general case when $\mathcal C \subseteq \rmod R$ is arbitrary. Then we concentrate on two important particular cases, when $\mathcal C = \add M$ and $\mathcal C = \Add M$, for an arbitrary module $M$. 
	
In the first case, we prove that $\varinjlim \add M = \{ N \in \rmod R \mid \exists F \in \mathcal F _S: N \cong F \otimes_S M \}$ where $S = \End M$, and $\mathcal F _S$ is the class of all flat right $S$-modules. In the second case, $\varinjlim \Add M = \{ \mathfrak F \odot _{\mathfrak S} M \mid \mathfrak F \in \mathcal F _{\mathfrak S} \}$ where $\mathfrak S$ is the endomorphism ring of $M$ endowed with the finite topology, $\mathcal F _{\mathfrak S}$ is the class of all right $\mathfrak S$-contramodules that are direct limits of direct systems of projective right $\mathfrak S$-contramodules, and $\mathfrak F \odot _{\mathfrak S} M$ is the contratensor product of the right $\mathfrak S$-contramodule $\mathfrak F$ with the discrete left $\mathfrak S$-module $M$.
	
For various classes of modules $\mathcal D$, we show that if $M \in \mathcal D$ then $\varinjlim \add M = \varinjlim \Add M$ (e.g., when $\mathcal D$ consists of pure projective modules), but the equality for an arbitrary module $M$ remains open. Finally, we deal with the question of whether $\varinjlim \Add M = \widetilde{\Add M}$ where $\widetilde{\Add M}$ is the class of all pure epimorphic images of direct sums of copies of a module $M$. We show that the answer is positive in several particular cases (e.g., when $M$ is a tilting module over a Dedekind domain), but it is negative in general.   
  \end{abstract}

\date{\today}

\thanks{Research supported by GA\v CR 20-13778S. The first-named author's research is also supported by RVO:~67985840.}

\maketitle

\tableofcontents

\section*{Introduction}

Direct limits provide one of the key constructions for forming large modules from families of small ones. In the case when the small modules are taken from a class of finitely presented modules, classic theorems of Lenzing et al.\ make it possible to describe completely the resulting class of large modules. However, if we start with a class, $\mathcal C$, consisting of arbitrary modules, then the structure of the class $\mathcal L = \varinjlim \mathcal C$ is much less clear: for example, $\mathcal L$ need not be closed under direct limits.   

Our first goal here is to investigate which closure properties of the class $\mathcal C$ carry over to $\mathcal L$. Then we will characterize the class $\mathcal L$ for two particular instances: when $\mathcal C$ is the class of all, and all finite, direct sums of copies of a single (infinitely generated) module $M$. The first characterization relies on the well-known equivalence between the category $\add M$ of all direct summands of finite direct sums of copies of $M$ and $(\rfmod S)_{\mathrm{proj}}$, the category of all finitely generated projective right $S$-modules, where $S$ is the endomorphism ring of $M$. The second characterization is based on a recently discovered equivalence \cite{PS} between the category $\Add M$ of all direct summands of arbitrary direct sums of copies of $M$ and $(\rcontra{\mathfrak S})_{\mathrm{proj}}$, the category of projective right contramodules over $\mathfrak S$, the endomorphism ring of $M$ endowed with the finite topology.

We will prove that in many cases, e.g., when $\mathcal C$ consists of small or pure projective modules, particular injective or Pr\" ufer modules, the classes $\varinjlim \add \mathcal C$ and $\varinjlim \Add \mathcal C$ coincide. However, whether this is true in general, remains an open problem. 

We will also characterize the class $\varinjlim\Add P$ when $P$ is a projective module in terms of its trace ideal. Another problem addressed here is the question of whether $\varinjlim \Add M = \widetilde{\Add M}$ where $\widetilde{\Add M}$ denotes the class of all pure epimorphic images of direct sums of copies of a module $M$. We will give a positive answer in several particular cases, e.g., when $M$ is an (infinitely generated) tilting module over a Dedekind domain. However, we will show that the answer is negative in general, even if the class $\Add M$ is closed under direct limits: we will construct an example of a countably generated flat module $M$ such that $\Add M = \varinjlim \Add M \subsetneq \widetilde{\Add M}$.            

 Let us say a few more words about the applications of contramodules
to the study of the class $\varinjlim\Add M$ and to
the $\varinjlim\add M$ versus $\varinjlim\Add M$ question.
 The notions of a flat module and a \emph{flat
contramodule}~\cite{PR,Pproperf,BPS} play a key role in the descriptions
of the classes $\varinjlim\add M$ and $\varinjlim\Add M$, respectively.
 The classical Govorov--Lazard theorem~\cite{Gov,La0} describes
the flat modules as the direct limits of projective modules, or even
more precisely, as the direct limits of finitely generated free modules.
 The analogous assertion is \emph{not} true for contramodules,
generally speaking, and we present a counterexample.

 Still it is not known whether every direct limit of projective
contramodules is a direct limit of finitely generated projective
(or finitely generated free) contramodules.
 When this holds for the topological endomorphism ring $\mathfrak S$
of a module $M$, it follows that $\varinjlim\add M=\varinjlim\Add M$.
 In particular, this observation is applicable to some Pr\"ufer-type 
modules $M$, or more generally, to modules $M$ whose topological
endomorphism ring $\mathfrak S$ admits a dense left noetherian
subring $S$ such that the induced topology on $S$ is a left Gabriel
topology with a countable base of ideals generated by central elements.
 It is important here that ideals generated by central elements in
noetherian rings have the \emph{Artin--Rees property}, which allows to
prove that the underlying $S$-modules of flat $\mathfrak S$-contramodules
are flat.

 We also prove that, for any module $M$, both the classes
$\varinjlim\add M$ and $\varinjlim\Add M$ are \emph{deconstructible}
(i.e.\ every module from the respective class is filtered by modules
of bounded size from the same class).
 The assumption that all flat $\mathfrak S$-contramodules are direct
limits of projective ones, for $\mathfrak S=\End M$, allows to improve
the cardinality estimate for deconstructibility of $\varinjlim\Add M$.
 In order to obtain the improved cardinality estimate, we study
homological properties of the class of all flat contramodules and
its natural subclass of so-called $1$-strictly flat contramodules.
 Under a mild assumption (that all flat contramodules are $1$-strictly
flat), we show that the class of all flat $\mathfrak S$-contramodules
is closed under (transfinite) extensions and kernels of epimorphisms,
and that it is quasi-deconstructible modulo the class of all so-called
contratensor-negligible contramodules.

\medskip

\section{Preliminaries}\label{prelim}

Let $R$ be a ring and let $\rmod R$ ($\rfmod R$) denote the class of all (all finitely presented) right $R$-modules. Let $\mathcal C$ be any class of modules closed under finite direct sums.

\medskip
The key subject of this paper is the class $\mathcal L = \varinjlim \mathcal C$ of all modules $M \in \rmod R$ for which there exists a direct system $\mathfrak D = (C_i, f_{ji} \mid i \leq j \in I )$ in $\rmod R$ with $C_i \in \mathcal C$ for all $i \in I$, such that $M$ is the direct limit of $\mathfrak D$. That is, $(M, f_i (i \in I))$ is the colimit of the diagram $\mathfrak D$ in $\rmod R$. We will use the notation of $M = \varinjlim C_i$ or $M = \varinjlim \mathfrak D$. 

That $M = \varinjlim C_i$ can equivalently be expressed as an internal property of the diagrams

\[
\xymatrix{& {M} \\ {C_{i}} \ar[rr]^{f_{ji}} \ar[ur]^{f_{i}} &  & {C_j} \ar[ul]_{f_{j}}}
\]

\noindent namely, as the conjunction of the following three conditions  

\begin{enumerate}  
\item[{(C1)}] $f_i = f_j f_{ji}$ for all $i \leq j \in I$, 
\item[{(C2)}] $M = \bigcup_{i \in I} \im f_i$, and 
\item[{(C3)}] $\Ker {f_i} \subseteq \bigcup_{i \leq j \in I} \Ker {f_{ji}}$ for all $i \in I$. 
\end{enumerate}

Also, $M = \varinjlim C_i$ is equivalent to the existence of a short exact sequence of the form 
$$(\ast) \quad 0 \to K \hookrightarrow \bigoplus_{i \in I} C_i \overset{\pi}\to M \to 0$$
where $\pi \restriction C_i = f_i$ for each $i \in I$, and $K = \Ker \pi = \langle x - f_{ji}(x) \mid x \in C_i \,\&\, i \leq j \in I \rangle$. This sequence is pure exact (in Lemma \ref{general} below, we will see that it is even locally split).

For more details and basic properties of direct limits, we refer to \cite[\S 2.1]{GT}. 

\begin{remark}\label{inverse} Condition (C3) has the easy corollary that if all the morphisms $f_{ji}$ in the direct system $\mathcal D$ are monomorphisms, then so are all the $f_i$ ($i \in I$). In Theorem \ref{describe} below, we will however prove that if $\mathcal C$ is closed under arbitrary direct sums, then we can always w.l.o.g.\ assume that all the morphisms $f_{ji}$ ($i \leq j \in I$) are split epimorphisms.   

It is worth noting that while the definition of a direct limit admits the equivalent internal formulation as above, this is not true of its category theoretic dual, that is, of the notion of an inverse limit of an inverse system of modules. 

The duals of conditions (C1) and (C2) do hold for inverse limits. The dual of condition (C3) holds when $I$ is countable and all the morphisms $f_{ij}$ in the inverse system are epimorphisms (in which case also all the morphisms $f_i$ are epimorphisms), but it fails in general. Using the existence of Aronzsajn trees, one can construct a well-ordered inverse system of modules $\mathfrak {I} = (C_\alpha, f_{\alpha \beta} \mid \alpha \leq \beta <  \aleph_1 )$ all of whose morphisms $f_{\alpha \beta}$ are non-zero epimorphisms, but the inverse limit  $M = \varprojlim \mathfrak {I}$ is $0$, whence $f_\alpha = 0$ for each $\alpha < \aleph_1$, cf.\ \cite{B} or \cite[6.39]{GT}. 

Moreover, the dual exact sequence to ($\ast$), expressing the inverse limit of an inverse system of modules as a submodule of the direct product of these modules, is not pure in general, cf.\ \cite[6.33]{GT}.  
\end{remark}

\medskip
 
For a class of modules $\mathcal D$, we will denote by $\Sum \mathcal D$ and $\ssum \mathcal D$ the class of all, and all finite, direct sums of copies of modules from $\mathcal D$, respectively. Further, $\Add \mathcal D$ and $\add \mathcal D$ will denote the class of all direct summands of modules in $\Sum \mathcal D$ and $\ssum \mathcal D$, respectively. If $\mathcal D$ consists of a single module $M$, we will write $\Sum M$ instead of $\Sum \{ M \}$, and similarly for $\ssum M$, $\Add M$, and $\add M$.  

For example, for $M = R$, $\Sum M$ and $\ssum M$ are the classes of all free, and finitely generated free, modules, and $\Add M$ and $\add M$ the classes of all projective, and finitely generated projective, modules, respectively. Note that this example shows that $\Sum \add M$ may be a proper subclass of $\Add M$ - this happens exactly in the case when there exists a (countably generated) projective module that does not decompose into a direct sum of finitely generated projective modules.

Of course, $\ssum M = \add M$ when the endomorphism ring of $M$ is local (cf.\ \cite[12.7]{AF}). However, even if $\ssum M \subsetneq \add M$, always $\varinjlim \add M = \varinjlim \ssum M$, and similarly for $\Add M$ and $\Sum M$. This simplifies the study of the direct limit closures in these cases. More in general, we have the following easy observation:

\begin{lemma}\label{summands} Let $R$ be a ring and $\mathcal E$ be any class of modules. Let $\mathcal E ^\prime$ denote the class of all direct summands of the modules in $\mathcal E$. Then $\varinjlim \mathcal E ^\prime = \varinjlim \mathcal E$.
\end{lemma}
\begin{proof} It suffices to prove that $\varinjlim \mathcal E ^\prime \subseteq \varinjlim \mathcal E$. Let $L \in \varinjlim \mathcal E ^\prime$, that is, there is a direct system of modules $\mathfrak D = ( E^\prime_i , f_{ji} \mid i \leq j \in I)$ with $E ^\prime_i \in \mathcal E ^\prime$, such that $\varinjlim \mathfrak D = (L, f_i (i \in I))$. For each $i \in I$, there exist modules $E_i \in \mathcal E$ and $E ^{\prime \prime}_i \in \mathcal E'$ such that  $E ^\prime_i \oplus E ^{\prime \prime}_i = E_i$.

If the poset $I$ has a maximal element $k$, then $L \cong E^\prime_k$, and $L$ is a countable direct limit of copies of $E_k$ (cf.\ Remark \ref{chain} below). If $I$ has no maximal element, we consider the direct system $\mathfrak C = ( E_i , g_{ji} \mid i \leq j \in I)$ with $g_{ji} \restriction E ^\prime_i = f_{ji}$ and $g_{ji} \restriction E ^{\prime \prime}_i = 0$. Then $\varinjlim \mathfrak C = (L, g_i (i \in I))$, where $g_i \restriction E ^\prime_i = f_i$ and $g_i \restriction E ^{\prime \prime}_i = 0$, whence $L \in \varinjlim \mathcal E$. 
\end{proof} 

Another easy, but important fact which holds for any class of modules $\mathcal E$, is that if $\mathcal L = \varinjlim \add \mathcal E$ is closed under direct limits, then, since $\Sum \mathcal E \subseteq \mathcal L$, also $\mathcal L =\varinjlim \Sum \mathcal E = \varinjlim \Add \mathcal E$, by Lemma \ref{summands}. Similarly, if $\mathcal L$ is closed under direct summands, then $\Add \mathcal E \subseteq \mathcal L$.  

\begin{remark}\label{chain} For any class of modules $\mathcal E$, we have the following implications: $\mathcal E$ is closed under arbitrary direct limits (i.e., $\mathcal E = \varinjlim \mathcal E$) implies that $\mathcal E$ is closed under countable direct limits, and that in turn implies that $\mathcal E$ is closed under direct summands. The latter implication holds because each direct summand $D$ of a module $E \in \mathcal E$ is a direct limit of a countable chain $E \overset{\pi}\to E \overset{\pi}\to ...$, where $\pi : E \to E$ is the identity on $D$ and zero on a (fixed) complement of $D$ in $E$.

These implications cannot be reversed in general: if $\mathcal E$ is the class of all projective modules over a non-right perfect ring $R$, so $R$ contains a strictly decreasing chain of principal left ideals $( Ra_i...a_0 \mid i < \omega )$, then by the classic Bass' Theorem P, if $M$ denotes the direct limit of the countable direct system $R \overset{f_0}\to R \overset{f_1}\to \dots$ where $f_i : R \to R$ is the left multiplication by $a_i$ for each $i < \omega$, then $M$ is not projective. Also, if $\mathcal E$ denotes the class of all countably presented modules over any ring, then $\mathcal E$ is closed under countable direct limits, but not under arbitrary ones (and even not under arbitrary direct sums). 

However, it is open whether if $\mathcal E = \varinjlim \add \mathcal C$ for a class of modules $\mathcal C$, and $\mathcal E$ is closed under direct summands, then $\mathcal E = \varinjlim \mathcal E$ (cf.\ Problem 3 in Section \ref{problems}). 
\end{remark} 

\medskip
   
A module $M$ is called \emph{self-small}, if for each (or equivalently, each countable) set $X$ and each $f \in \Hom RM{M^{(X)}}$, there exists a finite subset $F \subseteq X$ such that $\im f \subseteq M^{(F)}$. Moreover, $M$ is \emph{small}, if for each (or equivalently, each countable) system of modules $( N_\alpha \mid \alpha < \kappa )$ and each $f \in \Hom RM{\bigoplus_{\alpha < \kappa} N_\alpha}$, there exists a finite subset $F \subseteq \kappa$ such that $\im f \subseteq \bigoplus_{\alpha \in F} N_\alpha$. Note that the latter just says that the covariant $\Hom RM{-}$ functor commutes with arbitrary direct sums (so in the terminology of category theory, $M$ is a \emph{compact object} in $\rmod R$). 

For example, each finitely generated module over any ring is small, and each torsion-free module of finite rank over any commutative domain is self-small. However, if $M$ decomposes into an infinite direct sum of non-zero submodules, then $M$ is not self-small. Similarly, no countably, but not finitely, generated module is small. 

\medskip

For each $n \geq 0$, we will denote by $\mathcal P_n$, $\mathcal I _n$, and $\mathcal F _n$ the class of all modules of projective, injective, and flat dimension $\leq n$, respectively. 

\medskip

For a class of modules $\mathcal C$, we will denote by $\mathcal C ^\perp$ the right Ext-orthogonal class $\Ker{\Ext 1R{\mathcal C}{-}} = \{ M \in \rmod R \mid \Ext 1RCM = 0 \hbox{ for all } C \in \mathcal C \}$, and $\mathcal C ^{\perp_\infty} = \bigcap_{i \geq 1} \Ker{\Ext iR{\mathcal C}{-}}$. Similarly $^\perp \mathcal C = \Ker{\Ext 1R{-}{\mathcal C}}$ and $\mathcal C ^\intercal = \Ker{\Tor 1R{\mathcal C}{-}}$. For a class of left $R$-modules $\mathcal D$, we define $^\intercal \mathcal D = \Ker{\Tor 1R{-}{\mathcal D}}$. If $\mathcal C = \{ M \}$ for a module $M$, we write simply $M ^\perp$ in place of $\{ M \}^\perp$, and similarly for the other Ext- and Tor-orthogonal classes.  

A pair of classes of modules $\mathfrak C = (\mathcal A, \mathcal B)$ is a \emph{cotorsion pair} in case $\mathcal A = {}^\perp \mathcal B$ and $\mathcal B = \mathcal A ^\perp$. The cotorsion pair is called \emph{hereditary} if moreover $\Ext iRAB = 0$ for all $i > 1$, $A \in \mathcal A$, and $B \in \mathcal B$. The class $\Ker{\mathfrak C} = \mathcal A \cap \mathcal B$ is called the \emph{kernel} of $\mathfrak C$.   

\medskip

A module $T \in \rmod R$ is an (infinitely generated) \emph{tilting module} provided that $T$ has finite projective dimension, $\Ext iRT{T^{(X)}} = 0$ for all $i \geq 1$ and all sets $X$, and there is a finite exact sequence $0 \to R \to T_0 \to \dots \to T_k \to 0$ such that $T_i \in \Add T$ for each $i \leq k$. If $T$ is tilting, then there is the associated cotorsion pair $\mathfrak C = (\mathcal A, \mathcal B)$, such that $\mathcal B = T^{\perp_\infty}$. $\mathfrak C$ is called the \emph{tilting cotorsion pair}, and $\mathcal B$ the \emph{tilting class}, induced by $T$. Moreover, $\mathcal A = \Filt {\mathcal A ^{\leq \omega}}$, $\mathcal B = ({\mathcal A}^{< \omega})^\perp$, and $\Add T = \Ker{\mathfrak C} = \mathcal A \cap \mathcal B$. In particular, the class $\mathcal B$ is definable, and $\mathcal A \subseteq \mathcal P _n$ provided that $\pd T \leq n$. In the latter case, $T$ is called an \emph{$n$-tilting module}, and $\mathfrak C$ ($\mathcal B$) an \emph{$n$-tilting cotorsion pair} (\emph{$n$-tilting class}). For basic properties of tilting modules, we refer to \cite[Chap.\ 13]{GT}. 

\medskip

Let $\mathcal A$ be a class of modules. A homomorphism $f : A \to M$ is an \emph{$\mathcal A$-precover} of a module $M$ in case $A \in \mathcal A$, and each homomorphism from a module $A^\prime \in \mathcal A$ to $M$ factorizes through $f$. If $f$ is moreover right minimal, i.e., $f$ factorizes through itself only by an automorphism, then $f$ is an \emph{$\mathcal A$-cover} of $M$. If each module $M \in \rmod R$ has an $\mathcal A$-precover ($\mathcal A$-cover), then $\mathcal A$ is called a \emph{precovering} (\emph{covering}) class.

\medskip  

Let $\mathcal A$ be a class of modules and $M$ be a module. Then $M$ is \emph{$\mathcal A$-filtered} provided that there is a chain of submodules of $M$, $( M_\alpha \mid \alpha \leq \sigma )$, such that $M_0 = 0$, $M_{\alpha + 1}/M_\alpha$ is isomorphic to an element of $\mathcal A$ for each $\alpha < \sigma$, $M_\alpha = \bigcup_{\beta < \alpha} M_\beta$ for each limit ordinal $\alpha \leq \sigma$, and $M_\sigma = M$. The class of all $\mathcal A$-filtered modules is denoted by $\Filt {\mathcal A}$.    

Let $\kappa$ be an infinite cardinal. We will denote by $\mathcal A ^{< \kappa}$ and $\mathcal A ^{\leq \kappa}$ the class of all $< \kappa$-presented, and $\leq \kappa$-presented modules from $\mathcal A$. The class $\mathcal A$ is said to be \emph{$\kappa$-deconstructible} provided that $\mathcal A \subseteq \Filt {\mathcal A ^{< \kappa}}$. If moreover each module in $\mathcal A$ is isomorphic to a direct sum of $< \kappa$-presented modules from $\mathcal A$, then $\mathcal A$ is called \emph{$\kappa$-decomposable}. 

$\mathcal A$ is \emph{deconstructible} provided that $\mathcal A$ is $\kappa$-deconstructible for some infinite cardinal $\kappa$. Moreover, $\mathcal A$ is \emph{decomposable} provided that it is $\kappa$-decomposable for some infinite cardinal $\kappa$. For example, the class $\mathcal P _0$ of all projective modules is decomposable, as it is $\aleph_1$-decomposable, by a classic theorem of Kaplansky. However, most classes of modules encountered in homological algebra are not decomposable, but they often are deconstructible. For example, for each $n \geq 0$, the classes $\mathcal P_n$ and $\mathcal F _n$ are deconstructible over any ring $R$, cf.\ \cite[\S 8.1]{GT}. 
         
A class of modules $\mathcal A$ is \emph{closed under transfinite extensions} provided that $\mathcal A = \Filt {\mathcal A}$. In this case, $\mathcal A$ is closed under extensions and arbitrary direct sums. For example, for any class of modules $\mathcal B$, the class $^\perp \mathcal B$ is closed under transfinite extensions by the Eklof Lemma \cite[6.2]{GT}.  

Note that if $\mathcal S$ is any set of modules, then the class $\Filt {\mathcal S}$ is precovering, cf.\ \cite[7.21]{GT}. Hence, any deconstructible class of modules closed under transfinite extensions is precovering.  

\medskip

\section{Closure under direct sums and extensions}\label{one}

Let $\mathcal C$ be a class of modules closed under finite direct sums. The class $\mathcal L = \varinjlim \mathcal C$ is well-understood in the case when $\mathcal C$ consists of finitely presented modules: 

\begin{lemma}\label{lenzing} Let $R$ be a ring, $\mathcal C \subseteq \rfmod R$, and $\mathcal L = \varinjlim \mathcal C$. 
\begin{enumerate}
\item The class $\mathcal L$ is closed under arbitrary direct sums and direct limits, pure submodules and pure epimorphic images, and $\mathcal L \cap \rfmod R = \add \mathcal C$. Moreover, $\mathcal L$ is closed under pure extensions.
\item Assume moreover that $\mathcal C$ is closed under direct summands, extensions, $R \in \mathcal C$, and $\mathcal C$ consists of FP$_2$-modules. Then $\mathcal L = {}^\intercal (\mathcal C ^\intercal)$, whence $\mathcal L$ is a covering class of modules which is closed under transfinite extensions, and $\mathcal L$ is $\kappa^+$-deconstructible for $\kappa = \card R + \aleph_0$.
\end{enumerate}
\end{lemma}
\begin{proof} (i) Except for the last claim, these properties of $\mathcal L$ follow from the classic work of Lenzing \cite{Len}, see also \cite[2.13]{GT}.

For the last claim, let $(\ast \ast) \quad 0 \to X \to Z \overset{\rho}\to Y \to 0$ be a pure exact sequence with $X, Y \in \mathcal L$. 
Let $\mathfrak C = ( C_i , f_{ji} \mid i \leq j \in I)$ be a direct system with $C_i \in \mathcal C$ such that $\varinjlim \mathfrak C = (Y, f_i (i \in I))$. Taking pullbacks of $\rho$ and $f_i$ ($i \in I$), we obtain a direct system of short exact sequences $0 \to X \to Z_i \to C_i \to 0$ ($i \in I$) whose direct limit is the sequence $(\ast \ast)$. 

Since the pullback of a pure epimorphism is again a pure epimorphism, and $\mathcal C \subseteq \rfmod R$, we infer that for each  $i \in I$, the sequence $0 \to X \to Z_i \to C_i \to 0$ splits. So $Z_i \cong X \oplus C_i \in \mathcal L$ for all $i \in I$. Then $Z = \varinjlim Z_i \in \mathcal L$, too, because $\mathcal L$ is closed under direct limits when $\mathcal C \subseteq \rfmod R$. 

(ii) This was proved in \cite{AT}, see also \cite[6.19 and 8.40]{GT}.   
\end{proof}

\begin{remark}\label{genlenzing} The tools developed in the sequel will allow us to extend Lemma \ref{lenzing}(i) to the case when $\mathcal C$ is an arbitrary class consisting of pure projective modules -- see Corollary \ref{pureprojlim} below.
\end{remark}

For a general class of modules $\mathcal C$ closed under finite direct sums, $\mathcal L$ may fail some of the closure properties mentioned above. However, the closure under arbitrary direct sums always holds. In the particular case when $\mathcal C$ is closed under \emph{arbitrary} direct sums, this follows from the characterization of direct limits given in Theorem \ref{describe} below. However, \ref{describe} may fail in the general setting of classes closed only under finite direct sums -- see Example \ref{unitreg} below. Here we give a direct proof for the general setting employing the internal characterization of $\varinjlim$: 

\begin{proposition}\label{directs} The class $\mathcal L$ is closed under arbitrary direct sums.
\end{proposition}
\begin{proof} 
Let $(L_\alpha \mid \alpha < \kappa )$ be a sequence of modules from $\mathcal L$ and put $L = \bigoplus_{\alpha < \kappa} L_\alpha$. 

For each $\alpha < \kappa$, let $\mathfrak C _\alpha = (C_{\alpha, i}, f_{\alpha, j, i} \mid i \leq j \in I_\alpha )$ be a direct system of modules from $\mathcal C$ witnessing that $L_\alpha \in \mathcal L$, i.e., $(L_\alpha, f_{\alpha, i} (i \in I_\alpha))$ is the colimit of the diagram $\mathfrak C _\alpha$ in $\rmod R$ (where $f_{\alpha, i} : C_{\alpha, i} \to L_\alpha$). 

Let $F$ be a finite subset of $\kappa$, $F = \{ \alpha_1, \dots, \alpha_m \}$. Since $\mathcal C$ is closed under finite direct sums, 
the module $C_{F,\bar i} = \bigoplus_{k \leq m} C_{\alpha _k, i_k}$ belongs to $\mathcal C$ for each $m$-tuple of indices $\bar i = (i_1, \dots i_m ) \in I_{\alpha_1} \times \dots \times I_{\alpha_m}$. 

Let $G$ be a finite subset of $\kappa$ containing $F$, so $G = \{ \alpha_1, \dots, \alpha_n \}$ for some $n \geq m$ and let  $\bar j = (j_1, \dots j_n ) \in I_{\alpha_1} \times \dots \times I_{\alpha_n}$ be such that $\bar i \leq \bar j$, i.e., $i_k \leq j_k$ for all $k \leq m$. Define $h = f_{G,F,\bar j,\bar i} : C_{F,\bar i} \to C_{G,\bar j}$ by $h \restriction C_{\alpha _k, i_k} = f_{\alpha _k, j_k, i_k}$ for all $k \leq m$. 

Let $\mathfrak C = ( C_{F,\bar i}, f_{G,F, \bar j, \bar i} )$ where $F$ and $G$ run over all pairs of finite subsets of $\kappa$ such that $F \subseteq G$, and $\bar i$ and $\bar j$ over all $m$-tuples and $n$-tuples, with $m = \card F$, $n = \card G$ and $\bar i \leq \bar j$ as above. Then $\mathfrak C$ is a direct system of modules from $\mathcal C$. 

We will show that $L = \varinjlim \mathfrak C$ by verifying conditions (C1)-(C3) in the given setting. For $F$ and $\bar i$ as above, we define $h_{F, \bar i} : C_{F,\bar i} \to L$ by $h_{F, \bar i} = \bigoplus_{k \leq m} f_{\alpha _k, i_k}$. Note that $h_{F, \bar i} = h_{G, \bar j} f_{G,F,\bar j,\bar i}$ for $F \subseteq G$ and  $\bar i \leq \bar j$, because the equality holds when restricted to each $C_{\alpha _k, i_k}$ ($k \leq m$). Thus condition (C1) holds.

For conditions (C2) and (C3), we have to prove that $L = \bigcup \hbox{Im}(h_{F, \bar i})$, and that for each $x \in C_{F,\bar i}$, $h_{F, \bar i}(x) = 0$ implies the existence of $G \supseteq F$ and $\bar j \geq \bar i$ such that $f_{G,F,\bar j,\bar i}(x) = 0$.

The first claim is clear, since $\hbox{Im}(h_{F, \bar i}) = \sum_{k \leq m} \hbox{Im}(f_{\alpha _k, i_k})$ and $\bigcup_{i \in I_{\alpha}} \hbox{Im}(f_{\alpha,i}) = L_{\alpha}$ for each $\alpha < \kappa$.

The assumption of the second claim says that $\bigoplus_{k \leq m} f_{\alpha _k, i_k} (x) = 0$. So for each $k \leq m$, there exists $j_k \in I_{\alpha_k}$ such that the $i_k$-th component of $x$ is mapped to zero by $f_{\alpha _k, j_k, i_k}$. Take $G = F$ and $\bar j = (j_1,\dots,j_m)$. Then $f_{G,F,\bar j,\bar i} (x) = 0$, q.e.d.           
\end{proof}  

We will also make use of the following 

\begin{lemma}\label{holmjorg} $($\cite[Theorem 2.5]{HJ}$)$ Let $\mathcal D$ be a class of modules closed under pure epimorphic images. Then $\mathcal D$ is a covering class, iff $\mathcal D$ is closed under arbitrary direct sums.
\end{lemma}  

Proposition \ref{directs} and Lemma \ref{holmjorg} yield 

\begin{corollary}\label{holm} Assume that the class $\mathcal L$ is closed under pure epimorphic images. Then $\mathcal L$ is a covering class.
\end{corollary}
 
Here is another closure property that is passed from $\mathcal C$ to $\mathcal L$ in general:

\begin{lemma}\label{fact} Assume that $\mathcal C$ is closed under homomorphic images. Then $\mathcal L$ coincides with the class of all homomorphic images of arbitrary direct sums of modules from $\mathcal C$. In particular, $\mathcal L$ is closed under homomorphic images and direct limits, and it is a covering class; the $\mathcal L$-cover of a module $M$ is the embedding $T \hookrightarrow M$ where $T$ is the trace of $\mathcal C$ in $M$. 

Moreover, $\mathcal L$ consists of $\mathcal C$-filtered modules. If $\mathcal C$ is $\kappa$-deconstructible for some infinite cardinal $\kappa$ (e.g., if $\mathcal C$ has a representative set $S$ of objects up to isomorphism such that $\card S < \kappa$), then $\mathcal L$ is $\kappa$-deconstructible.   
\end{lemma}
\begin{proof} Assume that $L \in \mathcal L$,  i.e., $(L, f_i (i \in I))$ is the direct limit of a direct system consisting of modules from $\mathcal C$. Then $L$ is a (pure) epimorphic image of a direct sum of modules from $\mathcal C$, cf.\ \cite[2.9]{GT}. Conversely, let $f$ be an epimorphism $f : \bigoplus_{i \in I} C_i \to M$ with $C_i \in \mathcal C$ for each $i \in I$. If $I$ is finite, then $M \in \mathcal C$ by our assumption on $\mathcal C$. Otherwise consider the $\subseteq$-directed set $J$ of all finite subsets of $I$, and for each $S \in J$, let $D_S = f(\bigoplus_{i \in S} C_i)$. By our assumption on $\mathcal C$, $D_S \in \mathcal C$, and $M$ is the directed union of the direct systems of its submodules $( D_S \mid S \in J )$, so $M \in \mathcal L$. 

That $\mathcal L$ is a covering class now follows by Corollary \ref{holm}. Any homomorphism $f$ from a module $L \in \mathcal L$ into a module $M$ satisfies $\im f \in \mathcal L$, whence $\im f \subseteq T$ by the above. It follows that $T \hookrightarrow M$ is a $\mathcal L$-cover of $M$.        

Next we show that each module $L \in \mathcal L$ is $\mathcal C$-filtered. Indeed, a $\mathcal C$-filtration $(L _\alpha \mid \alpha \leq \sigma )$ of $L$, such that $L_\alpha \in \mathcal L$ for each $\alpha \leq \sigma$, is obtained as follows: $L_0 = 0$; if $L_\alpha \in \mathcal L$ is defined and $L_\alpha \neq L$, then using the fact (proved above) that $L$ is a directed union of modules from $\mathcal C$, we find a $C \in \mathcal C$ such that $C \subseteq L$, but $C \nsubseteq L_\alpha$. Let $L_{\alpha + 1} = L_\alpha + C \subseteq L$. Also $L_\alpha$ is a directed union of modules from $\mathcal C$, say $L_\alpha = \bigcup_{i \in I} C_i$, and $C_i + C \in \mathcal C$ for each $i \in I$ by our assumption on $\mathcal C$, so $L_{\alpha + 1}$ is the directed union of the modules $C_i + C \in \mathcal C$, whence $L_{\alpha + 1} \in \mathcal L$. Moreover, $L_{\alpha + 1}/L_\alpha \cong C/(C \cap L_\alpha) \in \mathcal C$. If $L_\beta \subsetneq L$ for all $\beta < \alpha$ and $\alpha$ is a limit ordinal, then we define $L_\alpha = \bigcup_{\beta < \alpha} L_\beta$. Then $L_\alpha \in \mathcal L$ since $\mathcal L$ is closed under direct limits. 

The assumption of the final claim says each module $C \in \mathcal C$ is filtered by $< \kappa$-presented modules from $\mathcal C$. By the above, the same holds for the modules in the class $\mathcal L$, so the class $\mathcal L$ is also $\kappa$-deconstructible. 
\end{proof}   

\begin{example}\label{1-tilt} (a) Let $R$ be a ring and $T$ a $1$-tilting module and $\mathcal C$ be the class of all homomorphic images of finite direct sums of copies of $T$. Then $\mathcal L = \Gen T$ is the tilting class induced by $T$, see e.g.\ \cite[14.2]{GT}. By Lemma \ref{fact}, $\mathcal L$ is closed under direct limits and it is deconstructible. Also, $\mathcal L$ is closed under transfinite extensions, and $\mathcal L$ is a covering class.
  
(b) Let $R$ be an integral domain with the quotient field $Q$. Let $\mathcal C$ be the class of all homomorphic images of finite direct sums of copies of $Q$. By Lemma \ref{fact}, $\mathcal L$ is the class of all h-divisible modules (= homomorphic images of arbitrary direct sums of copies of $Q$ = homomorphic images of injective modules), and by Lemma \ref{fact}, $\mathcal L$ is deconstructible. Since $\mathcal L$ is the class of all cosyzygies of all modules, $^\perp \mathcal L = \mathcal P _1$, whence $(^\perp \mathcal L)^\perp = \mathcal D$ is the 1-tilting class of all divisible modules (cf.\ \cite[9.1(a)]{GT}). 

So $\mathcal L$ is a $1$-tilting class (i.e., (b) is a particular instance of (a)), iff $\mathcal L = \mathcal D$. By \cite[\S VII.2, Theorem 2.8]{FS}, the latter happens, iff $R$ is a Matlis domain (i.e., $Q$ has projective dimension $1$). Notice that this is further equivalent to the class $\mathcal L$ being closed under extensions: indeed, the closure is clear when $\mathcal L = \mathcal D$. Conversely, if $Q$ has projective dimension $> 1$, then there exists a module $M$ such that $0 \neq \Ext 2RQM \cong \Ext 1RQ{E(M)/M}$, so there is a non-split short exact sequence $0 \to E(M)/M \to N \overset{\pi}\to Q \to 0$. Here, $E(M)/M$ and $Q$ are h-divisible, but $X$ is not: otherwise, there is an epimorphism $\rho : Q^{(X)} \to N$, whence $\pi \rho$ is a split $Q$-epimorphism, and $\pi$ splits, too, a contradiction.      
\end{example}

The proof of Proposition \ref{directs} is motivated by the simple fact that infinite direct sums are directed unions of their finite subsums, where all the maps involved are split monomorphisms. However, as shown in part (a) of the following example, the converse is not true in general: even if we assume that $\mathcal C$ is closed under extensions and direct summands, all the maps $f_{ji}$ in a direct system $\mathfrak C = (C_i, f_{ji} \mid i \leq j \in I)$ are split monomorphisms, and so are all the maps $f_i$ in the direct limit $(L, f_i (i \in I))$ of $\mathfrak C$, the module $L$ need not be a direct sum of the modules from $\mathcal C$. 

\begin{example}\label{injectives} (a) Let $\kappa$ be an infinite cardinal and $R$ a ring of cardinality $\leq \kappa$ which is not right noetherian. Let $\mathcal C$ be the class of all injective modules of cardinality $\leq 2^\kappa$. 

We claim that for each module $M$ of cardinality $\leq 2^\kappa$, the injective hull $E(M)$ of $M$ satisfies $E(M) \in \mathcal C$. To see this, let $D(M)$ denote the divisible hull of $M$ (viewed as an abelian group). Then $D(M)$ has cardinality $\leq 2^\kappa$, and we have the homomorphisms $M \cong \Hom RRM \subseteq \Hom {\mathbb Z}RM \subseteq \Hom{\mathbb Z}R{D(M)} = H$. Since $D(M)$ is an injective $\mathbb Z$-module and $R$ is a flat left $R$-module, the module $H$ is injective (see e.g.\ \cite[2.16(c)]{GT}). Moreover, $H$ has cardinality $\leq (2^\kappa)^\kappa = 2^\kappa$. So $H \in \mathcal C$, whence also $E(M) \in \mathcal C$, and the claim is proved. 

Let $L$ be any injective module. By the claim above, $L$ is the directed union of a direct system of split monomorphisms, $\mathfrak C = (C_i, f_{ji} \mid i \leq j \in I)$, where $\{ C_i \mid i \in I\} \subseteq \mathcal C$ is the set of all injective submodules of $L$ of cardinality $\leq 2^\kappa$. Note that all the maps $f_i$ ($i \in I$) in the colimit $(L, f_i (i \in I))$ of $\mathfrak C$ are split monomorphisms, too. 

Since $R$ is not right noetherian, the Faith-Walker theorem \cite[25.8]{AF} yields an injective module $L$ such that $L$ is not a direct sum of $\leq 2^\kappa$-generated (injective) modules. By the above, $L$ is a directed union of a direct system $\mathfrak C$ of split monomorphisms of modules from $\mathcal C$ such that also all the morphisms $f_i$ are split monomorphisms, but $L$ is not a direct sum of modules from $\mathcal C$.     

(b) We have just proved that the class $\mathcal L = \varinjlim \mathcal C$ contains all injective modules. By Proposition \ref{directs}, $\mathcal L$ contains all direct sums of injective modules (some of these are not injective, because $R$ is not right noetherian). However, in this generality, it is not exactly clear which modules the class $\mathcal L$ contains. There are two cases where we can give a complete answer: 

If $R$ is right hereditary, then the class $\mathcal C$ is closed under homomorphic images, so Lemma \ref{fact} applies, and $\mathcal L$ is the class of all homomorphic images of arbitrary direct sums of injective modules. 

If $R$ is right self-injective, then $\mathcal C$ contains all finitely generated projective modules, whence $\mathcal L$ contains all flat modules. So if $R$ is moreover von Neumann regular, then $\mathcal L = \rmod R$.     
\end{example} 
 
As suggested by Example \ref{1-tilt}(b), closure under extensions is a more subtle problem. In the general setting, we have

\begin{lemma}\label{exts} Assume that $\mathcal C$ is closed under extensions. Then the class $\mathcal L$ is closed under extensions of modules from $\mathcal C$. That is, if $X \in \mathcal C$, $Y \in \mathcal L$, and there is an exact sequence $$(\dagger) \quad 0 \to X \to Z \overset{\rho}\to Y \to 0,$$ 
then $Z \in \mathcal L$.
\end{lemma}
\begin{proof} The proof is similar to the one for Lemma \ref{lenzing}(i): Let $\mathfrak C = ( C_i , f_{ji} \mid i \leq j \in I)$ be a direct system with $C_i \in \mathcal C$ such that $\varinjlim \mathfrak C = (Y, f_i (i \in I))$. Taking pullbacks of $\rho$ and $f_i$ ($i \in I$), we obtain a direct system of short exact sequences $0 \to X \to Z_i \to C_i \to 0$ ($i \in I$) whose direct limit is the sequence $(\ast)$. Since $\mathcal C$ is closed under extensions, $Z_i \in \mathcal C$ for each $i \in I$, so $Z = \varinjlim Z_i \in \mathcal L$.   
\end{proof}

However, the version of Lemma \ref{exts} with swapped roles of $X$ and $Y$ fails in general. In particular, $\mathcal L$ need not be closed under extensions even if $\mathcal C$ is:

\begin{example}\label{complred} Let $R$ be a commutative semiartinian von Neumann regular ring of Loewy length $\alpha \geq 2$. Let $(\Soc {\beta}{R} \mid \beta \leq \alpha)$ be the socle sequence of $R$. Let $\mathcal C$ be the class of all finitely generated completely reducible modules. Then $\mathcal C$ is closed under finite direct sums, direct summands, and extensions (the latter holds because all simple modules are injective, so all extensions in $\mathcal C$ split, see e.g.\ \cite[p.216]{AF}). Moreover, $\mathcal L = \varinjlim \mathcal C$ is the class of all completely reducible modules; in particular, $\mathcal L$ is closed under direct limits. 

Consider the short exact sequence $$0 \to \Soc{1}{R} \to \Soc{2}{R} \overset{\pi}\to \Soc{2}{R}/\Soc{1}{R} \to 0.$$ Let $X$ be any non-zero finitely generated submodule of $\Soc{2}{R}/\Soc{1}{R}$, $Y = \Soc{1}{R}$, and $Z = \pi^{-1}(X) \subseteq \Soc{2}{R}$. Then we have the short exact sequence $$0 \to Y \to Z \to X \to 0,$$ where  $X \in \mathcal C$, $Y \in \mathcal L$, but $Z \notin \mathcal L$, because $\Soc {1}{Z} = \Soc{1}{R} = Y \subsetneq Z$. In particular, $\mathcal L$ is not closed under (pure) extensions.        
\end{example}

Notice that in Example \ref{complred}, the class $\mathcal C$ consists of finitely generated modules, but $\mathcal C \nsubseteq \rfmod R$. 

The pullback argument employed in the proofs of Lemmas \ref{lenzing}(i) and \ref{exts} gives yet another positive case: 

\begin{lemma}\label{exte} Assume that $\mathcal C \subseteq \rfmod R$, and $\Ext 1RCL = 0$ for all $C \in \mathcal C$ and $L \in \mathcal L$. Then $\mathcal L$ is closed under extensions. 
\end{lemma}

Further, we have

\begin{proposition}\label{ext2} Assume that $\mathcal C$ is closed under extensions and $\mathcal C$ consists of FP$_2$-modules. Then the class $\mathcal L$ is closed under extensions.
\end{proposition}
\begin{proof} Since $\mathcal L$ is closed under direct limits for $\mathcal C \subseteq \rfmod R$, by the proof of Lemma \ref{exts}, we only have to show that if $X \in \mathcal C$, $Y \in \mathcal L$, and there is an exact sequence 
$$(\ddagger) \quad 0 \to Y \to Z \to X \to 0,$$ then $Z \in \mathcal L$. 

By assumption, there exists a direct system $\mathfrak D = (Y_i, f_{ji} \mid i \leq j \in I )$ with $Y_i \in \mathcal C$ for all $i \in I$, such that $Y = \varinjlim Y_i$. Since $X$ is FP$_2$, \cite[6.6]{GT} yields that the canonical group homomorphism $\varinjlim \Ext 1RX{Y_i} \to \Ext 1RXY$ is an isomorphism. Hence $(\ddagger)$ is the direct limit of a direct system of short exact sequences $0 \to Y_i \to Z_i \to X \to 0$. By assumption $Z_i \in \mathcal C$ for each $i \in I$, whence $Z \in \mathcal L$. 
\end{proof}

\medskip

\section{Closure under direct limits and
the class $\protect\varinjlim\add M$}
    
For a class of modules $\mathcal A$, we will denote by $\widetilde{\mathcal A}$ the class of all pure epimorphic images of the modules from $\mathcal A$ (cf.\ \cite[8.37]{GT}). This class comes up naturally in our context in the case when $M$ is \emph{$\sum$-pure split}, i.e., each pure embedding $N \subseteq M^\prime$ with $M^\prime \in \Add M$ splits. Note that each $\sum$-pure injective module is $\sum$-pure split, cf.\ \cite[2.32]{GT}, and the converse is true e.g. when $R$ is left hereditary and $M$ is a tilting module by \cite[5.6]{AST}.   

First, we have the following observations:

\begin{lemma}\label{observe} Assume that $\mathcal A \subseteq \rmod R$ is closed under arbitrary direct sums. Then $\varinjlim \mathcal A \subseteq \widetilde{\mathcal A}$, and $\widetilde{\mathcal A}$ is a covering class closed under direct limits.
\end{lemma}
\begin{proof} This follows by Lemma \ref{holmjorg}.
\end{proof} 

Let us stress that the inclusion $\varinjlim \mathcal A \subseteq \widetilde{\mathcal A}$ is strict in general by Example \ref{Bergman} below (however, see Problem 3 in section \ref{problems}).

\begin{lemma}\label{sigma}
Let $M$ be a $\sum$-pure split module. Then $\Add M = \varinjlim \Add M = \widetilde{\Add M}$ is a covering class. Moreover, if $\varinjlim \add M$ is closed under direct summands, then also $\varinjlim \add M = \Add M$. 
\end{lemma}
\begin{proof} We always have $\Add M \subseteq \varinjlim \Add M \subseteq \widetilde{\Add M}$. By the assumption, pure epimorphic images of modules from $\Add M$ are their direct summands, whence $\widetilde{\Add M} \subseteq \Add M$. The covering property follows by Lemma \ref{observe}, and the final claim from the fact that $\Sum M \subseteq \varinjlim \add M$.
\end{proof}

Now, we arrive at the first main result of this paper characterizing the class $\varinjlim \add M$ for an arbitrary module $M$:

\begin{theorem}\label{limadd} Let $R$ be a ring, $M$ be a module and $S = \End {M_R}$. Then $\varinjlim \add M$ coincides with the class of all modules of the form $F \otimes_S M$ where $F$ is a flat right $S$-module. 
\end{theorem}
\begin{proof} First, notice that $\varinjlim \add M = \varinjlim \ssum M$ by Lemma \ref{summands}. Consider an arbitrary direct system of the form $\mathfrak D = ( M^{n_i} , g_{ji} \mid i \leq j \in I)$ where $n_i < \omega$ for each $i \in I$, and let $\varinjlim \mathfrak D = (L, g_i (i \in I))$.  

For all $i \leq j \in I$, $g_{ji}$ can be represented by an $n_j \times n_i$ matrix $H_{ji}$ with entries in $S$ as follows: for each $k < n_i$ and $l < n_j$, the element of $S$ occurring in the $l$th row and the $k$th column of the matrix $H_{ji}$ is the restriction of $g_{ji}$ to the $k$th copy of $M$ in $M^{n_i}$ composed with the canonical projection on to the $l$th copy of $M$ in $M^{n_j}$. 

Since $\mathfrak D$ is a direct system of modules, $\mathfrak E = ( S^{n_i} , h_{ji} \mid i \leq j \in I)$, where $h_{ji}$ is represented by the matrix $H_{ji}$ defined above for all $i \leq j \in I$, is a direct system of finitely generated free right $S$-modules. Let $(F, h_i (i \in I))$ be the direct limit of $\mathfrak E$ in $\rmod S$.   

Applying the functor $- \otimes_S M$ (which commutes with direct limits), we infer that $F \otimes_S M$ is the direct limit of the direct system $\mathfrak E \otimes_S M = ( S^{n_i} \otimes_S M, h_{ji} \otimes_S M \mid i \leq j \in I)$ in $\rmod R$. The latter system is isomorphic to the original direct system $\mathfrak D = ( M^{n_i}, g_{ji} \mid i \leq j \in I)$, since $M^{n_i} \cong S^{n_i} \otimes_S M$ and $g_{ji} \in \Hom R{M^{n_i}}{M^{n_j}}$ is the homomorphism corresponding to $h_{ji}$ in the isomorphism $\Hom S{S^{n_i}}{S^{n_j}} \cong \Hom R{M^{n_i}}{M^{n_j}}$ for all $i \leq j \in I$. Thus $F \otimes _S M \cong L$. 

Conversely, each flat right $S$-module $F$ is a direct limit of finitely generated free right $S$-modules, and tensoring by $- \otimes _S M$, we get that $F \otimes _S M$ is a direct limit of modules from $\ssum M$.      
\end{proof}

\begin{remark}\label{remk} If $\mathcal E$ is any class of modules closed under finite direct sums and direct limits, and Vop\v enka's principle (VP) holds, then there is a subset $\mathcal S \subseteq \mathcal E$ such that $\mathcal E = \varinjlim \mathcal S$. Let $M = \bigoplus_{S \in \mathcal S} S$. Then $\mathcal E = \varinjlim \ssum M = \varinjlim \Sum M$. In particular, VP implies that \emph{all} classes of modules closed under finite direct sums and direct limits are of the form $\varinjlim \add M$ for some module $M$, that is, of the form described in Theorem \ref{limadd}. For more details, see \cite[\S3]{El}.  
\end{remark}

Theorem \ref{limadd} has the following  

\begin{corollary}\label{deconstr} Let $R$ be a ring and $M$ be a module. Then the class $\varinjlim \add M$ is deconstructible.
\end{corollary}
\begin{proof}
Let $S = \End {M_R}$ and $\kappa = \card S + \aleph_0$. Then the class of all flat right $S$-modules $\mathcal F _0$ is $\kappa^+$-deconstructible (see e.g.\ \cite[6.23]{GT}), so each $F \in \mathcal F _0$ is the union of a continuous chain $( F _\alpha \mid \alpha \leq \sigma )$ of flat modules such that the consecutive factors $F_{\alpha + 1}/F_\alpha$ are flat and $\leq \kappa$-presented for all $\alpha < \sigma$. Since $\Tor 1S{F_{\alpha + 1}/F_\alpha}M = 0$ for each $\alpha < \sigma$, $F \otimes _S M$ is the union of a continuous chain $( F _\alpha \otimes _S M \mid \alpha \leq \sigma )$ of (right $R$-) modules such that the consecutive factors $(F_{\alpha + 1} \otimes _S M)/(F_\alpha \otimes _S M)$ are $\leq \lambda$-presented for each $\alpha < \sigma$, where $\lambda = \kappa . \tau$, and $\tau$ is the minimal cardinality of the union of a set of generators and a set of relations of the right $R$-module $M$. In view of Theorem \ref{limadd}, this implies that the class $\mathcal L = \varinjlim \add M$ is $\lambda^+$-deconstructible. 
\end{proof}

Later on, in Theorem \ref{lim-Add}, we will prove an analog of Theorem \ref{limadd} for the class $\varinjlim \Add M$ employing (some) flat $\mathfrak S$-contramodules and the contratensor product functor $- \odot _{\mathfrak S} M$, where $\mathfrak S$ is the endomorphism ring of $M$ endowed with the finite topology. For an analog of Corollary \ref{deconstr} for the class $\varinjlim \Add M$, see Corollary \ref{AddM-deconstr}, and Section \ref{deconstructibility}.

The following example goes back to \cite{AT} -- see also \cite[2.4]{GT}.  It is based on a construction, pioneered in \cite{GS}, of large $\aleph_1$-free modules over a discrete valuation domain (DVD) that possess only trivial endomorphisms, see \cite[20.19]{GT}. The main point of the example is that it presents a module $M$ such that the class $\mathcal L = \varinjlim \add M$ is not closed under direct summands, and hence $\mathcal L$ is not closed under countable direct limits, cf.\ Remark \ref{chain} (for another example of this phenomenon, see Example \ref{Bergman} below): 
  
\begin{example}\label{indec} Let $R$ be a countable DVD with the quotient field $Q$. By \cite[20.19]{GT}, for each infinite cardinal $\mu$ such that $\mu ^{\aleph_0} = \mu$ there exists an $\aleph_1$-free module $M$ of rank $\mu^+$ such that $S = \End {M_R} = R$, i.e., the only endomorphisms of $M$ are multiplications by elements of $R$. 

Let $\mathcal C = \add M$. By Theorem \ref{limadd}, $Q ^{(\mu^+)} \cong E(M) \in \varinjlim \mathcal C$, because $E(M) \cong Q \otimes _R M$ and $Q$ is a flat module. We will show that $Q \notin \varinjlim \mathcal C$; this will prove that $\varinjlim \mathcal C$ is not closed under direct summands. Indeed, if $0 \neq N \in \varinjlim \mathcal C$, then $N \cong F \otimes _R M$ for a non-zero flat (= torsionfree) module $F$ by Theorem \ref{limadd}. Since $R$ is a domain, $R \subseteq F$, whence $M \subseteq N$, and $N$ has rank $\geq \mu^+ > \aleph_0$. Thus, $\varinjlim \mathcal C$ contains no non-zero modules of countable rank; in particular, $Q \notin \varinjlim \mathcal C$.    

Next we show that $M$ is a self-small module: Let $X$ be a set and let $f \in \Hom RM{M^{(X)}}$. For each $x \in X$, denote by $\pi_x : M^{(X)} \to M$ the canonical projection on the $x$th component. Since $R \cong \End {M_R}$, for each $x \in X$, there exists $r_x \in R$ such that $\pi_x f \in \End {M_R}$ is the multiplication by $r_x$. In particular, if $0 \neq m \in M$, then the $x$th projection of $f(m)$ equals $m.r_x$. Since $M$ is torsion-free and $f(m) \in M^{(F)}$ for a finite subset $F \subseteq X$, necessarily $r_x = 0$ for all $x \in X \setminus F$, whence $\im f \subseteq M^{(F)}$, q.e.d. 

Since $M$ is self-small, $\mathcal L = \varinjlim \Add M$ (this will be proved for arbitrary self-small modules in Corollary \ref{selfsmall} below).               
\end{example}

\medskip

Note also that the class $\varinjlim \add M$ need not be closed under pure extensions, even though the class of all flat right $S$-modules, appearing in the characterization of $\varinjlim \add M$ in Theorem \ref{limadd}, is always closed under extensions. We have already seen this phenomenon in Example \ref{complred} above (where $\varinjlim \mathcal C = \varinjlim \add M$ was the class of all completely reducible modules, for $M$ = the direct sum of a representative set of all simple modules). Here is yet another example, over a DVD:   

\begin{example}\label{limext} Let $R = \mathbb Z _p$ be the localization of $\mathbb Z$ at a prime $p$, and $M = \mathbb J _p$ be the $p$-adic completion of $R$. Then $S = \End {M_R} \cong \mathbb J _p$ is the ring of all $p$-adic integers, so $\mathcal L$ is the class of all flat (= torsion-free) $\mathbb J_p$-modules, but viewed as $\mathbb Z _p$-modules.  

Notice that $\mathcal C = \add M = \{ \mathbb J _p ^n \mid n < \aleph_0 \}$ is closed under extensions in $\rmod {\mathbb Z _p}$, because $\Ext 1{\mathbb Z _p}{\mathbb J _p}{\mathbb J _p} = 0$ (as $\mathbb J _p$ is both a flat and a pure injective $\mathbb Z _p$-module).

We claim that $\Ext 1{\mathbb Z _p}{\mathbb J _p}{\mathbb J _p^{(\omega)}} \neq 0$. In order to verify the claim, consider the short exact sequence $0 \to \mathbb Z _p \to \mathbb J _p \to D \to 0$ in $\rmod {\mathbb Z _p}$ where $D$ is an uncountable direct sum of copies of $\mathbb Q$. Applying the functor $\Hom {\mathbb Z _p}{-}{\mathbb J _p^{(\omega)}}$, we obtain the long exact sequence 
$$0 = \Hom {\mathbb Z _p}{D}{\mathbb J _p^{(\omega)}} \to \Hom {\mathbb Z _p}{\mathbb J _p}{\mathbb J _p^{(\omega)}} \overset{\phi}\to 
\Hom {\mathbb Z _p}{\mathbb Z _p}{\mathbb J _p^{(\omega)}} \to$$  
$$\to \Ext 1{\mathbb Z _p}{D}{\mathbb J _p^{(\omega)}} \overset{\varphi}\to \Ext 1{\mathbb Z _p}{\mathbb J _p}{\mathbb J _p^{(\omega)}} \to \Ext 1{\mathbb Z _p}{\mathbb Z _p}{\mathbb J _p^{(\omega)}} = 0.$$
The restriction map $\phi$ is clearly surjective, whence $\varphi$ is an isomorphism. As $\mathbb J _p^{(\omega)}$ is not pure injective, and hence not cotorsion as a $\mathbb Z _p$-module, $\Ext 1{\mathbb Z _p}{\mathbb Q}{\mathbb J _p^{(\omega)}} \neq 0$, and the claim follows.

By the claim above, there is a non-split short exact sequence 
$$0 \to \mathbb J _p^{(\omega)} \overset{f}\to N \overset{g}\to \mathbb J _p \to 0$$ 
in $\rmod {\mathbb Z _p}$, whose outer terms belong to $\mathcal L$. It remains to prove that $N \notin \mathcal L$, i.e., the $\mathbb Z _p$-module structure on $N$ does not extend to a $\mathbb J _p$-module structure making $N$ a torsion-free $\mathbb J _p$-module. If so, then $N$ is a directed union of copies of free $\mathbb J _p$-modules of finite rank. Since $\mathbb J _p$ is a reduced $\mathbb Z _p$-module, 
$\Hom {\mathbb Z _p}{\mathbb J _p}{\mathbb J _p} = \Hom {\mathbb J _p}{\mathbb J _p}{\mathbb J _p}$, whence $g$ is a $\mathbb J _p$-homomorphism. Similarly, as $N$ is reduced, also $f$ is a $\mathbb J _p$-homomorphism. As $\mathbb J _p$ is a free $\mathbb J _p$-module, the short exact sequence above splits in $\rmod {\mathbb J _p}$, and hence in $\rmod {\mathbb Z _p}$, a contradiction.  

Let us finish by noting that the fact that $\mathbb J _p$ is a reduced $\mathbb Z _p$-module similarly implies that 
$\Hom {\mathbb Z _p}{\mathbb J _p}{\mathbb J _p ^{(X)}} = \Hom {\mathbb J _p}{\mathbb J _p}{\mathbb J _p ^{(X)}}$ for any set $X$, whence $\mathbb J _p$ is a self-small $\mathbb Z _p$-module, and $\varinjlim{\add \mathbb J _p} = \varinjlim{\Add \mathbb J _p}$ by Corollary \ref{selfsmall}.  
\end{example}

Next, we apply the results above to the particular setting of rings of quotients:

\begin{corollary}\label{Goldie} Let $R$ be a semiprime right Goldie ring and $Q$ be its classical right quotient ring. Then  $\rmod Q$ is a full subcategory of $\rmod R$, and as a right $R$-module, $Q$ satisfies $\varinjlim{\add Q} = \varinjlim{\Add Q} = \widetilde{\Add Q} = \rmod Q$.
\end{corollary}
\begin{proof} By the classic Goldie's Theorem, $Q$ is completely reducible, and $Q$ is the maximal right quotient ring of $R$ which is a perfect right localization of $R$ (see e.g.\ \cite[4.6.2]{L} and \cite[XII.2.6]{S}). Thus $\rmod Q$ is a full subcategory of $\rmod R$ by \cite[XI.1.2]{S}, and $\varinjlim{\add Q} = \rmod Q$ by Theorem \ref{limadd}. Moreover, $Q$ is a flat and divisible module, hence also $\widetilde{\Add Q} = \rmod Q$ by \cite[7.11 and 7.13]{GW}.
\end{proof}      

The last example in this section exhibits a case when $\rmod Q = \varinjlim{\add Q} = \varinjlim{\Add Q} \subsetneq \widetilde{\Add Q}$, but $\rmod Q$ is not closed under direct limits in $\rmod R$:  

\begin{example}\label{Bergman} Let $K$ be a field and let $R$ denote the $K$-subalgebra consisting of all eventually constant sequences in the $K$-algebra $Q = K^\omega$. Then $R$ is a hereditary von Neumann regular ring semiartinian ring whose each ideal is countably generated, $Q = E(R)$, and $Q$ is the maximal quotient ring of $R$. The Loewy length of $R$ is $2$, its socle sequence being $0 \subseteq K^{(\omega)} \subseteq R$. Moreover, $Q$ is a von Neumann regular right self-injective ring by \cite[1.24]{G} (see also \cite[\S\S 4.3-4.5]{L} and \cite[\S3]{T}). We claim that the $R$-module $Q$ satisfies $\rmod Q = \varinjlim{\add Q} = \varinjlim{\Add Q} \subsetneq \widetilde{\Add Q} = \Gen Q_R$.

Since $Q$ is von Neumann regular and $\End {Q_R} = Q$, Theorem \ref{limadd} gives $\varinjlim{\add Q} = \rmod Q$. Notice that here, $\rmod Q$ is \emph{not} a full subcategory of $\rmod R$, because the embedding $R \hookrightarrow Q$ is not a ring epimorphism (cf.\ \cite[XI.1.4]{S}). Since $Q/R$ is a singular module while $Q$ is non-singular, we have $\Hom R{Q/R}{Q^{(X)}} = 0$ for any set $X$, whence $\Hom RQ{Q^{(X)}} = \Hom QQ{Q^{(X)}}$. It follows that $Q$ is a self-small module, whence $\varinjlim{\add Q} = \varinjlim{\Add Q}$ by Corollary \ref{selfsmall}. That $\varinjlim {\Add Q} = \rmod Q$ can also be seen from the fact that $R$-homomorphisms between elements of $\Sum Q$ are $Q$-homomorphisms, so direct limits of elements of $\Sum Q$ are the same whether computed in $\rmod Q$ or $\rmod R$. 

Since $R$ is von Neumann regular, $\widetilde{\Add Q} = \Gen Q_R$. Note that $\Soc {Q}{Q} = \Soc {R}{Q} = K^{(\omega)}$. So the simple module $K \cong R/\Soc {R}{Q} \in \widetilde{\Add Q}$, because $R/\Soc {R}{Q}$ is a direct summand in the completely reducible module $Q/\Soc {Q}{Q} \cong K^{2^\omega}$. It is well-known that $Q$-submodules of $Q/\Soc {Q}{Q}$ correspond 1-1 to filters on $\omega$ containing the Fr\' echet filter. So $Q/\Soc {Q}{Q}$ contains no minimal, and hence no finite $K$-dimensional $Q$-submodules. Thus the one-$K$-dimensional module $K \cong R/\Soc {Q}{Q} \in \Gen Q_R \setminus \rmod Q$. It follows that $\rmod Q$ is not closed under direct summands in $\rmod R$, hence $\add Q \subsetneq \varinjlim{\add Q}$, and $\varinjlim{\add Q}$ is not closed under direct limits in $\rmod R$. 
\end{example}

It may even happen for an $R$-module $M$ that the class $\Add M$ is closed under direct limits, but it is not closed under pure epimorphic images in $\rmod R$. That is, $\Add M = \varinjlim{\Add M} \subsetneq \widetilde{\Add M}$. An example of a ring $R$ and a countably presented indecomposable flat $R$-module $M$ with this property will be constructed in Section~\ref{npureq}.         

\medskip

\section{Local splitting and a characterization of
$\protect\varinjlim$ for classes closed under direct sums}

We start by recalling the definition of a locally split monomorphism going back to Azumaya \cite{Az} (see also \cite{Z}):

\begin{definition}\label{recall} A monomorphism $\nu \in \Hom RXY$ is called \emph{locally split} in case for each finite subset $F \subseteq X$ there exists $\rho_F \in \Hom RYX$ such that $\rho_F \nu \restriction F = id_F$. A short exact sequence $0 \to X \overset{\nu}\to Y \to Z \to 0$ is said to be \emph{locally split} provided that $\nu$ is a locally split monomorphism. 
\end{definition}

It is easy to see that each locally split monomorphism is pure. The converse fails in general: 

\begin{example}\label{vNr} Let $R$ be a von Neumann regular ring which is not completely reducible. By \cite[Theorem]{O}, there exists a non-injective cyclic module $C$ (e.g., $C = R$ when $R$ is not right self-injective). Consider the short exact sequence $\mathcal E : 0 \to C \to E(C) \to E(C)/C \to 0$. Then $\mathcal E$ is pure, because $E(C)/C$ is a flat module (in fact, each module is flat as $R$ is von Neumann regular), but $\mathcal E$ is not locally split, because $C$ is finitely generated and $\mathcal E$ does not split. 
\end{example} 

Our interest in locally split short exact sequences comes from the fact that the short exact sequence ($\ast$) expressing a direct limit as a homomorphic image of a direct sum is always locally split. This was proved in \cite[2.1]{GG} for the particular case of linearly ordered systems of modules. Here we consider arbitrary direct systems: 

\begin{lemma}\label{general} Let $\mathfrak M = (M_i, f_{ji} \mid i \leq j \in I )$ be a direct system of modules, and $(M, f_i (i \in I))$ be its direct limit in $\rmod R$. Then the canonical presentation ($\ast$) of $M$, $0 \to K = \Ker \pi \overset{\sigma}\hookrightarrow \bigoplus_{i \in I} M_i \overset{\pi}\to M \to 0$, is locally split. 

In fact, $K$ is a directed union of a direct system $( K_F \mid F \in \mathcal F )$ of its submodules, where $\mathcal F$ is the set of all finite subsets of $I$ that contain a greatest element, and for each $F \in \mathcal F$ with the greatest element $j_F \in F$, $K_F \oplus M_{j_F} = \bigoplus_{i \in F} M_i$. 
\end{lemma}
\begin{proof} First, recall that $\pi \restriction M_i = f_i$ for each $i \in I$, and $K$ is generated by the set $G = \{ x - f_{ji}(x) \mid x \in M_i \,\&\, i \leq j \in I \}$. 

For each $F \in \mathcal F$, $K_F$ is defined as the submodule of $\bigoplus_{i \in F} M_i$ generated by the set $G_F = \{ x - f_{j i}(x) \mid i, j \in F, i \leq j, x \in M_i \}$. Clearly, if $F, F^\prime \in \mathcal F$ and $F \subseteq F^\prime$, then $K_F \subseteq K_{F^\prime}$, and $K$ is the directed union of the $K_F$ ($F \in \mathcal F$). 

We claim that $K_F \oplus M_{j_F} = \bigoplus_{i \in F} M_i$. Indeed, since $i \leq j_F$ for each $i \in F$, $K_F + M_{j_F}$ contains the module $M_i$ for each $i \in F$, whence $K_F + M_{j_F} = \bigoplus_{i \in F} M_i$.  

Let $y \in K_F \cap M_{j_F}$. Let $f \in \Hom R{\bigoplus_{i \in F} M_i}{M_{j_F}}$ be defined by $f \restriction M_i = f_{j_F,i}$ for each $i \in F$. This is possible since $i \leq j_F$ for each $i \in F$. On the one hand, $y \in M_{j_F}$ and $f_{j_F,j_F} = id_{M_{j_F}}$, so $f(y) = y$. On the other hand, $y \in K_F$, so $f(y) = 0$, because $f(g) = 0$ for each $g \in G_F$. Hence $K_F \cap M_{j_F} = 0$. 

Thus $K_F \oplus \bigoplus_{i \in (I \setminus F) \cup \{ j_F \}} M_i = \bigoplus_{i \in I} M_i$. Let $\rho _F \in \Hom R{\bigoplus_{i \in I} M_i}{K_F}$ denote the projection on to $K_F$ in this decomposition. Then $\rho _F \sigma \restriction K_F = id_{K_F}$. Since each finite subset of $K$ is contained in $K_F$ for some $F \in \mathcal F$, the short exact sequence ($\ast$) is locally split.  
\end{proof}

\begin{remark}\label{quasisplit} In the setting of Lemma \ref{general}, the local splitting of the monomorphism $\sigma$ can also be proved by showing that $\sigma$ is a direct limit of split monomorphisms with the same codomain (that is, $\sigma$ is a \emph{quasi-split monomorphism} in the sense of \cite[\S 4]{BPS}), see Proposition 4.1 and Lemmas 4.3 and 4.4 in the recent paper \cite{BPS}. In fact, \cite[Proposition 4.1]{BPS} also gives an alternative proof for Theorem \ref{describe} below.
\end{remark} 

\begin{definition}\label{locally} Let $\mathcal A$ be a class of modules. Denote by $\overline{\mathcal A}$ the class of all modules $M$ such that there is a short exact sequence $0 \to K \hookrightarrow A \to M \to 0$ where $A \in \mathcal A$, and $K$ is a directed union of a direct system consisting of direct summands of $A$ with complements in $\mathcal A$ (that is, $K$ is the directed union of a direct system $( K_F \mid F \in \mathcal F )$ such that for each $F \in \mathcal F$, $K_F$ is a direct summand of $A$, and $A/K_F \in \mathcal A$). Notice that each such short exact sequence is locally split, hence pure, so $\overline{\mathcal A} \subseteq \widetilde{\mathcal A}$.
\end{definition}

For classes of modules $\mathcal A$ closed under arbitrary direct sums, Lemma \ref{general} yields a surprising description of the modules in the class $\varinjlim \mathcal A$ as direct limits of direct systems consisting of split epimorphisms:  

\begin{theorem}\label{describe} Let $\mathcal A$ be a class of modules closed under arbitrary direct sums. Then $\varinjlim \mathcal A = \overline{\mathcal A}$. 

Moreover, $\varinjlim \mathcal A$ coincides with the class of all modules $M$ such that there exists a direct system $\mathfrak D = ( A_F, \pi_{F^\prime F} \mid F \subseteq F^\prime \in \mathcal F )$ such that $A_F \in \mathcal A$ for each $F \in \mathcal F$, $M = \varinjlim \mathfrak D$, and $\pi_{F^\prime F} : A_F \to A_{F^\prime}$ is a split epimorphism for all $F \subseteq F^\prime \in \mathcal F$.   
\end{theorem}

\begin{proof} By Lemma \ref{general}, if $M \in \varinjlim \mathcal A$, then the canonical presentation of $M$ has the form $0 \to K \hookrightarrow A \to M \to 0$ where $A \in \mathcal A$, and $K$ is the directed union of a direct system $( K_F \mid F \in \mathcal F )$ of direct summands of $A$ with complements in $\mathcal A$. So $M \in \overline{\mathcal A}$, and $M \cong A/K$ is the direct limit of the direct system $\mathfrak D = (A_F, \pi_{F^\prime F} \mid F \subseteq F^\prime \in \mathcal F )$, where $A_F = A/K_F \in \mathcal A$, and $\pi_{F^\prime F} \in \Hom R{A_F}{A_{F^\prime}}$ is the canonical projection modulo $K_{F^\prime}/K_F$ which is a split epimorphism for all $F \subseteq F^\prime \in \mathcal F$. 

Indeed, in the notation of the proof of Lemma \ref{general}, $K_F \oplus \bigoplus_{i \in (F^\prime \setminus F) \cup \{ j_F \}} M_i = \bigoplus_{i \in F^\prime} M_i$. Hence $K_{F^\prime} = K_F \oplus (K_{F^\prime} \cap \bigoplus_{i \in (F^\prime \setminus F) \cup \{ j_F \}} M_i)$. Also $A = \bigoplus_{i \in I} M_i = K_{F^\prime} \oplus \bigoplus_{i \in (I \setminus F^\prime) \cup \{ j_{F^\prime} \}} M_i$. Thus $A_F = A/K_F = K_{F^\prime}/K_F \oplus \bigoplus_{i \in (I \setminus F^\prime) \cup \{ j_{F^\prime} \}} M_i$, and $K_{F^\prime}/K_F \cong K_{F^\prime} \cap \bigoplus_{i \in (F^\prime \setminus F) \cup \{ j_F \}} M_i \in \add \mathcal A$.   

Conversely, if $M \in \overline{\mathcal A}$, then $M \cong A/\bigcup_{F \in \mathcal F} K_F = \varinjlim \mathfrak D$ for a direct system $\mathfrak D$ as above, whence $M \in \varinjlim \mathcal A$. 
\end{proof}

The next example shows that closure of $\mathcal A$ under arbitrary direct sums is essential for Theorem \ref{describe} to hold; the closure under finite direct sums is not sufficient:

\begin{example}\label{unitreg} Let $R$ be a unit regular ring which is not completely reducible (see \cite[Chap.\ 4]{G}). By \cite[Corollary 2.16]{G}, $R$ possesses an infinite set of non-zero pairwise orthogonal idempotents, $E = \{ e_i \mid i < \omega \}$. Let $I$ be the right ideal of $R$ generated by $E$, and $\mathcal A$ be the class of all finitely generated free modules. Then $\mathcal A$ is closed under finite direct sums, the sequence $0 \to I \to R \to R/I \to 0$ is quasi-split, $R/I \in \varinjlim \mathcal A = \rmod R$, but $R/I \notin \overline{\mathcal A}$. 

In order to prove the latter claim, assume that $R/I \in \overline{\mathcal A}$. Then there is a short exact sequence $0 \to K \hookrightarrow R^n \to R/I \to 0$ for some $0 < n < \omega$, where $K$ is the directed union of a direct system $( K_F \mid F \in \mathcal F )$ such that for each $F \in \mathcal F$, $K_F$ is a direct summand in $R^n$, and $R^n/K_F \cong R^{n_K}$ for some $n_K < \omega$. 

By \cite[Theorem 4.5]{G}, the unit regularity of $R$ implies (in fact, is equivalent to) cancellation for direct sums of finitely generated projective modules. Thus $n_K \leq n$ and $K_F$ is a free module of rank $n - n_K$. By Schanuel's Lemma, $K \oplus R \cong R^n \oplus I$, whence $K$ is a countably infinitely generated projective module. So $\mathcal F$ contains a cofinal chain $( F_i \mid i < \omega )$ which yields the chain of submodules $K_{F_0} \subseteq K_{F_1} \subseteq \dots \subseteq K_{F_i} \subseteq K_{F_{i+1}} \subseteq \dots \subseteq K = \bigcup_{i < \omega} K_{F_i}$. Since $R$ is von Neumann regular, $K_{F_i}$ is a direct summand in $K_{F_{i+1}}$ for each $i < \omega$. As the ranks of the free modules $K_{F_i}$ are bounded by $n$, the chain of submodules has to stabilize at some $m < \omega$. Then $K = K_{F_m}$ is finitely generated, a contradiction.    
\end{example}

\medskip

\section{$\protect\varinjlim\add\mathcal D$ versus
$\protect\varinjlim\Add\mathcal D$} \label{versus}

In this section, we consider the question of when $\varinjlim \add \mathcal D$ equals $\varinjlim \Add \mathcal D$ (or, equivalently, $\varinjlim \ssum \mathcal D$ equals $\varinjlim \Sum \mathcal D$ by Lemma \ref{summands}) for a class of modules $\mathcal D$.  

The equality is trivial when $\mathcal D$ is closed under arbitrary direct sums, and easy to prove in the case when $\mathcal D = \{ M \}$ for a self-small module $M$: one can refine the original direct system $\mathfrak D$ consisting of modules from $\Sum M$ into a direct system consisting of modules from $\ssum M$ making use of appropriate restrictions of the maps from $\mathfrak D$. A similar argument works for an arbitrary class $\mathcal D$ consisting of small modules. Lemma~\ref{small2} goes another step further in this direction.

\begin{lemma} \label{change}
  Let $\mathcal C$ be a class of modules. Consider a direct system $\mathfrak D = (M_i, f_{ji} \mid i \leq j \in I)$ with direct limit 
 $(M,f_i \mid i \in I)$. If $I$ does not have the largest element and $f_{ji}$ factors through an object of $\mathcal C$
 for every $i < j \in I$, then $M \in \varinjlim \mathcal C$.
\end{lemma}

\begin{proof}
Let $U = \{(i,j) \mid i < j \in I\}$. Note that $U \neq \emptyset$ since $I$ does not contain the largest element. 
Define a partial order on $U$ by $(i,j) \preceq (i',j')$ if either $(i,j) = (i',j')$ or $j \leq i'$. Note that $(U,\preceq)$ is 
directed. Indeed, if $(i,j),(i',j') \in U$, then there exists $i'' \in I$ such that $j, j'\leq i''$ and since $i''$ is not the 
largest element of $I$, there exists $j'' \in I$ such that $(i'',j'') \in U$. Then $(i,j),(i',j') \preceq (i'',j'')$.

By the assumption, for every $u = (i,j) \in U$ there exist $C_u \in \mathcal C$, $\alpha_u \in \Hom R{M_i}{C_u}$ 
and $\beta_u \in \Hom R{C_u}{M_j}$ such that $f_{ji} = \beta_u \alpha_u$. 

For $u = (i,j) \preceq v = (i',j') \in U$ define $g_{vu} = 1_{C_u}$ if $v = u$ and $g_{vu} = \alpha_v f_{i'j} \beta_u \in \Hom R{C_u}{C_v}$
if $u\prec v$. Note that if $u = (i,j) \prec v = (i',j') \prec w = (i'',j'')$, then 
$g_{wv}g_{vu} = \alpha_{w}f_{i''j'} \beta_v \alpha_v f_{i',j} \beta_u = 
\alpha_{w} f_{i''j'} f_{j'i'}f_{i'j} \beta_u  = g_{wu}$. 

Hence we can define the direct system $\mathfrak E = (C_u,g_{vu} \mid u \preceq v \in U)$. For $u = (i,j) \in U$ denote $r(u) = j$.
We claim that $\varinjlim \mathfrak E = (M, f_{r(u)}\beta_u \mid u \in U)$. Let us check conditions (C1) - (C3).

(C1) If $u = (i,j) \prec v = (i',j') \in U$ then $f_{j'}\beta_v g_{vu} = f_{j'}\beta_v\alpha_vf_{i'j} \beta_{u} = 
f_{j'}f_{j'j}\beta_u = f_{j} \beta_u$.

(C2) For $i \in I$ there exists $j \in I$ such that $u = (i,j) \in U$. Then $f_{j}\beta_u\alpha_u = f_i$ and 
$\im f_i \subseteq \im f_{r(u)}\beta_u$. In particular, $M = \cup_{u \in U} \im f_{r(u)}\beta_u$.

(C3) Let $x \in \Ker{f_{r(u)}\beta_u}$ for some $u = (i,j) \in U$. Then $\beta_u(x) \in \Ker{f_j}$, hence there exists 
$i' \geq j \in I$ such that $f_{i'j}\beta_u(x) = 0$. If $i' < j' \in I$ and $v = (i',j')$, then $g_{vu}(x) = 0$.
\end{proof}

\begin{lemma} \label{small2}
Let $\mathcal C$ be a class of small modules. If $\mathcal D \subseteq \Add \mathcal C$ then $\varinjlim \Add \mathcal D = 
\varinjlim \add \mathcal D$.
\end{lemma}

\begin{proof}
 By Lemma \ref{summands}, it is sufficient to prove $ \varinjlim \Sum \mathcal D \subseteq \varinjlim \ssum \mathcal D$.
So consider a direct system of the form $\mathfrak D = ( M_i , f_{ji} \mid i \leq j \in I)$
, where $M_i \in \Sum \mathcal D$ for every $i \in I$, and its direct limit $\varinjlim \mathfrak D = (N, f_i (i \in I))$.

If $I$ contains the largest element $i$, then $N \simeq M_i \in \Sum \mathcal D$ is contained in $\varinjlim \ssum \mathcal D$
by Proposition \ref{directs}. Therefore we may assume that $I$ has no largest element.

For every $i \in I$ there exists a module $M_i'$ such that $M_i \oplus M_i' \in \Sum \mathcal C$. Let $N_i := M_i \oplus M_i' = 
\bigoplus_{t \in X_i} C_t$, where $C_t \in \mathcal C$. For every $i \in I$ denote $\pi_i \in \Hom R{N_i}{M_i}$ and 
$\iota_i \in \Hom R{M_i}{N_i}$ the canonical projection and embedding related to the decomposition $N_i = M_i \oplus M_i'$. 
 
For $i < j \in I$ define $g_{ji} \in \Hom R {N_i}{N_j}$ by  
$g_{ji} \restriction M_i = \iota _j f_{ji}$ and $g_{ji}\restriction {M_i'} = 0$, that is, $g_{ji} = \iota_j f_{ji} \pi_i$.
Also define $g_{ii} = 1_{N_i}$ for every $i \in I$.
Note that the direct system $\mathfrak E = ( N_i, g_{ji} \mid i \leq j \in I)$ has limit 
$(N, g_i (i \in I))$, where $g_i\restriction M_i = f_i$ and $g_i \restriction M_i' = 0$. 

Let $P$ be the set of all pairs $p = (i,F)$ such that $i \in I$ and $F$ is a finite subset of $X_i$. For $p = (i,F) \in P$, we let 
$N_p = \oplus_{t \in F} C_t$. Note that $N_p$ is a direct summand of $N_i$, let $\iota_{p} \in \Hom R{N_p}{N_i}$ and 
$\pi_p \in \Hom R{N_i}{N_p}$ be the corresponding embedding and projection.
Also, if $q = (j,G) \in P$, then we define $p \preceq q$, iff either $p = q$ or $i < j$ in $I$ and $f_{ji}(N_p) \subseteq N_q$. Since every $N_p$ is small, for all $p_1 = (i_1,F_1) \in P$ and $p_2 = (i_2,F_2) \in P$, there exist $j \in I$ and a finite subset $G \subseteq X_j$ such that $q = (j,G) \in P$ satisfies $p_1 \preceq q$  and $p_2 \preceq q$. Thus $(P,\preceq)$ is a directed poset. 

For $p = (i,F) \preceq q = (j,G) \in P$, let $h_{qp} = g_{ji} \restriction N_p$. Then $\mathfrak F = ( N_p , h_{qp} \mid p \preceq q \in P)$ is a direct system of modules from $\ssum \mathcal C$. For each $p = (i,F) \in P$, we let $h_p = g_i \restriction N_p$. Then it is easy to see that $\varinjlim \mathfrak F = (N, h_p (p \in P))$ whence $N \in \varinjlim \mathcal \ssum \mathcal C$.

Note that if $p = (i,F) \prec q = (j,G) \in P$, then 
$h_{qp} = \pi_q g_{ji}\iota_p = \pi_q\iota_j f_{ji}\pi_i \iota_p$. In particular, $h_{qp}$ factors through $M_i \in \Sum \mathcal D$.
Let $M_i = \bigoplus_{t \in Y_i} D_t$, where $D_t \in \mathcal D$. Since $N_p$ is small, there exists $F_p \subseteq Y_i$
finite such that $\pi_i\iota_p (N_p) \subseteq \bigoplus_{t \in F_p} D_t$. 
It follows that $h_{qp}$ factors through a module in $\ssum \mathcal D$. 
Note that $P$ has no largest element, so 
Lemma \ref{change} implies  $N \in \varinjlim \ssum \mathcal D$. 
\end{proof}

\begin{corollary}\label{selfsmall} 
\begin{enumerate}
\item Let $M$ be a self-small module. Then $\varinjlim \add M = \varinjlim \Add M$.
\item Let $\mathcal D$ be a class consisting of small modules. Then $\varinjlim \add \mathcal D = \varinjlim \Add \mathcal D$. In particular, $\varinjlim \add M = \varinjlim \Add M$ whenever $M$ is an arbitrary direct sum of small modules. 
\item Let  $\mathcal D$ be a class  of pure projective modules, then $\varinjlim \Add \mathcal D = \varinjlim \add \mathcal D$. 
\end{enumerate}
\end{corollary}

\begin{proof}
Apply Lemma \ref{small2} for $\mathcal C = \{M\}$, $\mathcal C = \mathcal D$ and $\mathcal C = \rfmod R$, respectively.
\end{proof}

 The next proposition is a generalization of
Corollary~\ref{selfsmall} to $\eta$-self-small and $\eta$-small modules,
where $\eta$~is an infinite cardinal.
 Let us say that a module $M$ is \emph{$\eta$-self-small} if for each
set $X$ and each $f \in \Hom RM{M^{(X)}}$, there exists a subset
$F \subseteq X$ of the cardinality $\card F\le\eta$ such that
$\im f \subseteq M^{(F)}$.
 Similarly, $M$ is \emph{$\eta$-small} if for each system of
modules $( N_\alpha \mid \alpha < \kappa )$ and each $f \in \Hom RM{\bigoplus_{\alpha < \kappa} N_\alpha}$, there exists a subset
$F \subseteq \kappa$ of the cardinality $\card F\le\eta$ such that
$\im f \subseteq \bigoplus_{\alpha \in F} N_\alpha$. 

 In particular, any $\leq \eta$-generated module is $\eta$-small,
and any $\eta$-small module is $\eta$-self-small.

 Given a class of modules $\mathcal D$, let us denote by
$\mathcal D^{(\eta)}$ the class of all direct sums
$\bigoplus_{\alpha<\eta}D_\alpha$ of families of modules
$(D_\alpha\in\mathcal D\mid \alpha<\eta)$.

\begin{proposition}\label{eta-selfsmall} 
\begin{enumerate}
\item Let $M$ be an $\eta$-self-small module. Then
$\varinjlim \add M^{(\eta)} = \varinjlim \Add M^{(\eta)}=
\varinjlim \Add M$.
\item Let $\mathcal D$ be a class consisting of $\eta$-small modules.
Then $\varinjlim \add \mathcal D^{(\eta)} = \varinjlim
\Add\mathcal D^{(\eta)} = \varinjlim \Add \mathcal D$.
In particular, $\varinjlim \add M^{(\eta)} =
\varinjlim \Add M^{(\eta)} = \varinjlim \Add M$ whenever $M$ is
an arbitrary direct sum of $\eta$-small modules. 
\end{enumerate}
\end{proposition}

\begin{proof}
 The argument is analogous to the proof of
Lemma~\ref{small2}, with suitable modifications.
 Let us prove part~(ii); part~(i) is similar to~(ii) but simpler.
 For convenience of notation and without loss of generality,
assume that $0\in\mathcal D$.
 In view of Lemma~\ref{summands}, it suffices to prove that
$N\in\varinjlim\Sum\mathcal D$ implies
$N\in\varinjlim\mathcal D^{(\eta)}$.
 Assume that we are given a direct system $\mathfrak D=
(D_i , f_{ji} \mid i \leq j \in I)$ such that
$D_i=\bigoplus_{\alpha<\kappa_i}D_{\alpha,i}$ with
$D_{\alpha,i}\in\mathcal D$ for every $\alpha<\kappa_i$, $i\in I$,
and $\varinjlim \mathfrak D = (N, f_i (i \in I))$.

 Let $P$ be the set of all pairs $p=(i,F)$ such that $i\in I$ and
$F\subseteq\kappa_i$ is a subset of the cardinality $\card F\le\eta$.
 For every $p=(i,F)\in P$, put $E_p=\bigoplus_{\alpha\in F}
D_{\alpha,i}$.
 We define a partial order~$\preceq$ on the set $P$ as follows:
if $q=(j,G)\in P$, then we say that $p\preceq q$ iff $i\le j$ in $I$
and $f_{ji}(E_p)\subseteq E_q$.
 Since $D_{\alpha,i}$ is $\eta$-small for all $\alpha<\kappa_i$,
for every $p_1=(i_1,F_1)\in P$ and $p_2=(i_2,F_2)\in P$ there exists
$q=(j,G)\in P$ such that $p_1\preceq q$ and $p_2\preceq q$.
 So $(P,\preceq)$ is a directed poset.

 For every $p=(i,F)\preceq q=(j,G)\in P$ we put $g_{qp}=f_{ji}
\restriction E_p$.
 Then $\mathfrak E=(E_p,g_{qp}\mid p\preceq q\in P)$ is a direct
system of modules from $\mathcal D^{(\eta)}$.
 For each $p=(i,F)\in P$, put $g_p=f_i\restriction E_p$.
 Then $\varinjlim\mathfrak E=(N, g_p (p\in P))$, hence
$N\in\varinjlim\mathcal D^{(\eta)}$.
\end{proof}

 The following corollary is a version of
Corollary~\ref{deconstr} for $\varinjlim\Add M$.

\begin{corollary} \label{AddM-deconstr}
 Let $R$ be a ring and $M$ be a module.
 Then the class $\varinjlim\Add M$ is deconstructible.
\end{corollary}

\begin{proof}
 Let $\eta$~be the minimal infinite cardinal for which the $R$-module
$M$ is $\eta$-self-small (or, if $M$ is self-small, put $\eta=1$).
 Clearly, $\eta\le\tau$, where $\tau$~is the minimal cardinality of
a set of generators of the right $R$-module~$M$.
 Consider the $R$-module $N=M^{(\eta)}$;
then, by Proposition~\ref{eta-selfsmall}(i), we have
$\varinjlim\Add M=\varinjlim\add N$. 

At this point, Corollary~\ref{deconstr} already yields the deconstructibility of the class $\varinjlim\Add M$. 
In the remaining part of the proof, we will give a more precise deconstructibility bound: 
Let $S=\End M_R$ and $T=\End N_R$; put $\kappa=\card S+\aleph_0$
and $\rho=\card T+\aleph_0$.
 Then $\rho=\kappa^\eta$ (since the elements of $T$ can be represented
as $\eta\times\eta$ matrices with the entries in~$S$; and while there
is some convergence condition on the columns of such matrices,
arbitrary rows are allowed).
By (the proof of) Corollary~\ref{deconstr}, the class $\varinjlim\Add M$ is
$\lambda^+$-deconstructible, where $\lambda=\rho.\tau$.
\end{proof}

The equality $\varinjlim \add M = \varinjlim \Add M$ may hold even if $M$ is an indecomposable non self-small module, such as $M = \mathbb Z _{p^\infty}$ (the Pr\" ufer $p$-group, for a prime integer $p$) over $R = \mathbb Z$. Indeed, $\mathbb Z _{p^\infty} =  \bigcup_{n < \omega} \mathbb Z g_n$ where $p g_0 = 0$ and $p g_{n+1} = g_n$ for all $n < \omega$. Considering $f \in \Hom{\mathbb Z}{\mathbb Z _{p^\infty}}{\mathbb Z _{p^\infty}^{(\omega)}}$ defined by $f(g_n) = (g_n,g_{n-1},\dots,g_0,0,\dots )$ for each $n < \omega$, we see that $\mathbb Z _{p^\infty}$ is not self-small. That $\varinjlim \add M = \varinjlim \Add M$ is a consequence of the following more general fact:

\begin{lemma}\label{injectlim} Let $R$ be a right noetherian ring and $\mathcal C$ be any class of injective modules. Then $\mathcal L = \varinjlim \add \mathcal C = \varinjlim \Add \mathcal C = \Add \mathcal C$. Moreover, $\mathcal L$ is deconstructible and closed under direct limits, hence $\mathcal L$ is a covering class. 
\end{lemma}
\begin{proof} Let $\mathcal A = \Sum \mathcal C$ and $N \in \varinjlim \mathcal A$. By Theorem \ref{describe}, $N \cong A/K$ where $A \in \mathcal A$ and $K$ is the directed union of a direct system $(K_F \mid F \in \mathcal F )$ of direct summands of $A$ with complements in $\mathcal A$. Since $R$ is right noetherian, $K$ is injective, whence $N \in \Add \mathcal C$. Thus $\varinjlim \Add \mathcal C = \Add \mathcal C$.

Assume there exists $N \in \Add \mathcal C \setminus \varinjlim \ssum \mathcal C$. Since $R$ is right noetherian, $N$ is a direct sum of indecomposable injective modules. By Proposition \ref{directs}, we can w.l.o.g.\ assume that $N$ is indecomposable, whence $N = E(R/I)$ for a right ideal $I$ of $R$. Since $N \oplus K = A = \bigoplus_{j \in J} C_j$ for some $C_j \in \mathcal C$ ($j \in J$), there is a finite subset $G \subseteq J$ such that $R/I \subseteq \bigoplus_{j \in G} C_j$. Then $N = E(R/I)$ is isomorphic to a direct summand of $\bigoplus_{j \in G} C_j$, whence $N \in \varinjlim \ssum \mathcal C$, a contradiction.  

Since $R$ is right noetherian, all modules in $\mathcal L = \Add \mathcal C$ are direct sums of indecomposable direct summands of the modules from $\mathcal C$. Hence $\mathcal L$ is deconstructible. By the above, $\mathcal L = \varinjlim \mathcal L$.      
\end{proof}

\begin{remark}\label{smallinj} Note that by Example \ref{Bergman} above, Lemma \ref{injectlim} fails in general for non-right noetherian rings. However, even for certain non-right noetherian rings, $\varinjlim \add \mathcal C = \varinjlim \Add \mathcal C$ for any class of injective modules $\mathcal C$. This occurs when $R$ is a right purely infinite ring (that is, $R$ contains a free module of infinite rank), or $R$ is a simple von Neumann regular ring which is not noetherian. The point is that over such rings, \emph{all} injective modules are small (cf.\ \cite[Example 2.8]{TPad}), whence Corollary \ref{selfsmall}(ii) applies.   
\end{remark}

\medskip

 Let $R$ be a ring and $t\in R$ be a central element.
 Denote by $R[t^{-1}]$ the localization of $R$ at the central
multiplicative subset $\{1,t,t^2,t^3,\dotsc\}\subset R$ generated
by~$t$.
 Consider the localization map $R\to R[t^{-1}]$, and denote
simply by $R[t^{-1}]/R$ its cokernel.
 Then the right $R$-module $R[t^{-1}]/R$ is a generalization of
the Pr\"ufer $p$-group (which can be constructed as
$\mathbb Z[p^{-1}]/\mathbb Z$).
 The $R$-module $R[t^{-1}]/R$ is usually not injective, but one still
has $\varinjlim \add M=\varinjlim\Add M$ under the assumptions of
the next proposition.
 Notice that both Lemma~\ref{injectlim} and
Proposition~\ref{prufer-inverted-element} deal with right modules,
however the ring $R$ is assumed to be \emph{right} noetherian in
the former but \emph{left} noetherian in the latter.

\begin{proposition} \label{prufer-inverted-element}
 Let $R$ be a left noetherian ring and $t\in R$ be a central element.
 Consider the right $R$-module $M=R[t^{-1}]/R$.
 Then\/ $\varinjlim\add(M)=\varinjlim\Add(M)$.
\end{proposition}

 More generally, one can consider localizations by countable central
multiplicative subsets.
 Given a multiplicative subset $T$ consisting of central elements in $R$,
form the localization $T^{-1}R$ and denote by $T^{-1}R/R$ the cokernel
of the localization map $R\to T^{-1}R$.

\begin{proposition} \label{prufer-inverted-subset}
 Let $R$ be a left noetherian ring and $T\subset R$ be a countable
multiplicative subset consisting of (some) central elements in~$R$.
 Consider the right $R$-module $M=T^{-1}R/R$.
 Then\/ $\varinjlim\add(M)=\varinjlim\Add(M)$.
\end{proposition}

 The proofs of Propositions~\ref{prufer-inverted-element}
and~\ref{prufer-inverted-subset} use contramodule techniques.
 They will be given below in Section~\ref{genprufer} after
a preparation in Sections~\ref{contramodule-methods}--\ref{contra-versus}.

\medskip

\section{The case of projective modules}\label{projective}

In this section, we consider the particular case of projective modules. By Corollary~\ref{selfsmall}, 
$\varinjlim \add M = \varinjlim \Add M$ for a projective module $M$. In this particular case we can 
describe $\varinjlim \add M$ and show that this class is closed under direct summands and direct limits.

Let $I \subseteq R$ be an ideal, $F_1,F_2$ finitely generated free right $R$-modules, and $u \in {\rm Hom}_R(F_1,F_2)$.
We say that $u$ is {\em supported in} $I$ if $u(F_1) \subseteq F_2I$. If $F_1 = R^n, F_2 = R^m$ for some $n,m \in \mathbb{N}$, then 
$u$ is supported in $I$ if and only if $u$ is given by left multiplication of a matrix from ${\rm M}_{m,n}(I)$.

The following easy observation will be useful: if $F_1,F_2,F_3$ are finitely generated free right $R$-modules, $u \in \Hom R{F_1}{F_2}$, and $v \in \Hom R{F_2}{F_3}$, then 
$vu$ is supported in $I$ whenever either $u$ or $v$ is supported in $I$.

\begin{lemma} \label{extend}
Let $I$ be an ideal of $R$ and let $M \in \ModR$ satisfy $MI = M$. 
If $F$ is a finitely generated free right $R$-module and $f \in \Hom R{F}{M}$, then there exist
a finitely generated free right $R$-module $F'$, $g \in \Hom R{F'}{M}$ and $u \in \Hom R{F}{F'}$
supported in $I$ such that $f = gu$.
\end{lemma}

\begin{proof}
First let us assume $F= R$, $f \colon r \mapsto mr$ for some $m \in M$. Write $m = \sum_{j = 1}^{k} m_{j}i_j$, where 
$m_1,\dots,m_k \in M$ and $i_1,\dots,i_k \in I$. 
Define $F':=R^{k}$, $u \colon r \mapsto (i_1,i_2,\dots,i_k)r, r \in R$ and $g \colon (r_1,\dots,r_k) \mapsto 
\sum_{j = 1}^k m_jr_j, (r_1,\dots,r_k) \in F'$. Obviously $f = gu$ and $u$ is supported in $I$.

In general, consider $F = R^n$ and its canonical basis $e_1,\dots,e_n$. For 
$j = 1,\dots, n$ set $m_j := f(e_j)$ and define $f_j \in \Hom R{R}{M}$ by 
$f_j(1) = m_j$. Apply the previous paragraph to find $u_j  \in \Hom R {R}{R^{k_j}}$ 
supported in $I$ and $g_j \in \Hom R{R^{k_j}}{M}$ such that $f_j = g_j u_j$. 
Now it is enough to set $F' = R^{k_1} \oplus R^{k_2} \oplus \cdots \oplus R^{k_n}$, 
$u = u_1 \oplus u_2 \oplus \cdots \oplus u_n$ and $g = g_1 \oplus g_2 \oplus \cdots \oplus g_n$.
Obviously $u$ is supported in $I$ and $f = gu$ since $gu(e_j) = m_j$ for every $1 \leq j \leq n$.
\end{proof}

By Govorov-Lazard theorem, every flat module $M$ is a direct limit of a direct system of finitely generated free modules. 
If $M = MI$ for an ideal $I$, then the diagram can be chosen such that all its morphisms which do not have 
to be identity are supported in $I$.
  
\begin{proposition} \label{diagram}
Let $I$ be an ideal of $R$ and let $M$ be a flat module satisfying $MI = M$. Then there exists a direct system 
$\mathfrak D = (F_p,f_{qp} \mid p \leq q \in P)$ of finitely generated free modules having all morphisms 
$f_{qp}, p<q \in P$ supported in $I$ such that $\varinjlim \mathfrak D = (M,f_i (i \in P))$. 
\end{proposition}

\begin{proof}
Since $M$ is flat, there exists a direct system $\mathfrak E = (F_p,f_{qp} \mid p\leq q \in P
)$ such that $\varinjlim \mathfrak E = (M,f_p \mid p \in P)$, where $F_p$ are finitely generated free modules. 
Define a new partial order on $P$ by 
$p \leq' q$ if either $p = q$ or $p < q$ and $f_{qp}$ is supported in $I$.

Let us prove that for every $p \in P$ there exists $q \in P$ such that 
$p \leq' r$ for every $q \leq r \in P$. Apply Lemma \ref{extend} to $f_p$ to find 
a finitely generated free $R$-module $F$, $u \in \Hom R{F_p}{F}$ supported in $I$, and
$h \in \Hom R{F}{M}$ such that $f_p = hu$. Since $h$ is a homomorphism from a finitely presented 
module into $M$, it factors through some $f_{q_0}$, that is, $h = f_{q_0}v$ for some 
$v \in \Hom R{F}{F_{q_0}}$. Of course, we may assume $p \leq q_0$. 

Note that $vu,f_{q_0 p} \in \Hom R{F_p}{F_{q_0}}$ and $f_{q_0}(vu - f_{q_0 p}) = f_p - f_p = 0$.
Since $F_p$ is finitely generated, there exists $q_0 \leq q \in P$ such that $f_{q q_0}(vu - f_{q_0 p}) = 0$.
Therefore $f_{qp} = f_{q q_0}f_{q_0 p} = f_{q q_0}vu$ is supported in $I$. Consequently, $f_{r p} = 
f_{r q}f_{q p}$ is supported in $I$ for every $q \leq r \in P$. It follows that $(P, \leq')$ is a directed poset.

Consider the direct system $\mathfrak D = \{F_p, f_{qp} \mid p \leq' q \in P \}$. It is now easy to verify 
conditions (C1) - (C3) to get $\varinjlim \mathfrak D = (M,f_p  (p \in P))$.
\end{proof}

Let $M$ be a right $R$-module. Then $I := \sum_{f \in \Hom R{M}{R}} \im f$ is an ideal of 
$R$ called the {\em trace ideal} of $M$. If $M$ is a projective $R$-module, then $I = I^2$ is the smallest 
ideal of $R$ satisfying $MI = M$.  

\begin{lemma} \label{factor}
Let $P$ be a projective $R$-module with trace ideal $I$. If $F_1,F_2$ are finitely generated free right 
$R$-modules and $u \in \Hom R{F_1}{F_2}$ is supported in $I$, then there exist a countably generated 
module $P_u \in {\rm add}(P)$ and homomorphisms $\alpha \in \Hom R{F_1}{P_u}$, $\beta \in \Hom R{P_u}{F_2}$
such that $u = \beta\alpha$.
\end{lemma}

\begin{proof}
By \cite[Proposition 8.12]{AF}, $F_2I$ is generated by $P$. That is, there exist a cardinal $\kappa$
and an epimorphism $\beta' \in \Hom R{P^{(\kappa)}}{F_2I}$. Since $F_1$ is projective and $u$
is supported in $I$, there exists a homomorphism 
$\alpha' \colon F_1 \to P^{(\kappa)}$ such that $u = \beta' \alpha'$.
By the theorem of Kaplansky, $P = \oplus_{\lambda \in \Lambda} P_\lambda$, where $P_\lambda$ are countably generated modules. 
Since $F_1$ is finitely generated, $\alpha'$ factors through a finite subsum of $(\oplus_{\lambda \in \Lambda} P_{\lambda})^{(\kappa)}$. 
So there exist $P_u \in {\rm add}(P)$ countably generated and $\alpha\in \Hom R{F_1}{P_u}$, $\alpha'' \in \Hom R{P_u}{P^{(\kappa)}}$
such that $\alpha' = \alpha'' \alpha$. It remains to set $\beta:= \beta'\alpha''$.
\end{proof}

\begin{theorem} \label{main6}
Let $P$ be a projective module with trace ideal $I$. Then every flat module $M$ satisfying $MI = M$ is a direct 
limit of countably generated modules from $\add P$.
\end{theorem}

\begin{proof}
Since the statement is obviously true for $M = 0$, we may assume $M \neq 0$.
By Proposition~\ref{diagram}, there exists a direct system 
$\mathfrak D = (F_u,f_{vu} \mid u \leq v \in U)$ of finitely generated free modules and $f_{vu}$ supported in $I$
for every $u<v \in U$
such that $\varinjlim \mathfrak D = (M,f_u (u \in U))$. If $U$ contains the largest element $u$, then $M \simeq F_u$
and $I = R$, that is, $P$ is a generator. In this case $M$ is a finitely generated module in  $\add P$. 

If $U$ contains no largest element, we can apply Lemma~\ref{change} since, by Lemma \ref{factor}, 
$f_{vu}$ factors through a countably generated module from $\add P$ for every $u < v \in U$.
 \end{proof}

\begin{corollary}\label{project}
  Let $P$ be a projective module with trace ideal $I$. Then $\varinjlim \add P = \varinjlim \Add P = 
	\{M \in \ModR \mid M $ flat and $MI = M\}$ is closed under direct limits and direct summands.
\end{corollary}

\begin{proof}
The equality $\varinjlim \add P = \varinjlim \Add P$ was already proved in Corollary~\ref{selfsmall}. 
Theorem~\ref{main6} gives $\{M \in \ModR \mid M $ flat and $MI = M\} \subseteq \varinjlim \add P$.
On the other hand, since $PI = P$, every module $M \in \varinjlim \add P$ is flat and satisfies $MI = M$.
The last statement follows from the fact that the class of all flat modules is closed under 
direct summands and direct limits and the same is true for the class $\{M \in \ModR \mid MI = M\}$.
\end{proof}

The arguments used in the proof of Theorem~\ref{main6} can be used to extend Lenzing's results to classes of 
pure projective modules.

\begin{proposition} \label{pprojectives}
  Let $\mathcal C$ be a class of pure projective modules. Then $M \in \varinjlim \mathcal C$ if and only if 
	every $f \in \Hom R{X}{M}$ factors through a module of $\mathcal C$ whenever $X \in \rfmod R$.
\end{proposition}

\begin{proof}
  The direct implication holds for any class ${\mathcal C} \subseteq \rmod R$ (see for example \cite[Lemma~2.8]{GT}).
	
	Conversely, consider a direct system $\mathfrak D = (M_i, f_{ji} \mid i \leq j \in I)$ of finitely presented modules 
	and its colimit $\varinjlim \mathfrak D = (M,f_i(i \in I))$. Assume that for every $X \in \rfmod R$ and $f \in \Hom R{X}{M}$
	there exist $C \in \mathcal C$, $\alpha \in \Hom R{X}{C}$, $\beta \in \Hom R{C}{M}$ satisfying $f = \beta \alpha$.
	
	We claim that for every $i \in I$ there exists $i\leq j \in I$ such that $f_{ki}$ factors through an object of $\mathcal{C}$
	for every $j \leq k \in I$: Fix $i \in I$.
	By our assumption, $f_i = \beta \alpha$ for some $C \in \mathcal C$, $\alpha \in \Hom R{M_i}{C}$, $\beta \in \Hom R{C}{M}$.
	Since $C$ is pure projective, $C$ is a direct summand of $P = \oplus_{t \in T} F_t$, where $F_t \in \rfmod R$ for every $t\in T$.
	Obviously there are $\beta' \in \Hom R {C}{P}, h \in \Hom R {P}{M}$ such that $\beta = h\beta'$. Since $M_i$ is finitely generated, 
	there exists a finite set $T_0 \subseteq T$ such that $\im \beta'\alpha \subseteq Q = \oplus_{t \in T_0} F_t$. Let 
	$\pi_Q \in \Hom R{P}{Q}, \iota_Q \in \Hom R{Q}{P}$ satisfy $\pi_Q\iota_Q = 1_Q$. Then 
	$f_i = h\beta'\alpha = (h\iota_Q \pi_Q)(\beta'\alpha)$ is a factorization of $f_i$ through $Q$. Since $Q$ is finitely 
	presented, $h\iota_Q$ factors through some $f_{j_0}$, $j_0 \in  I$. Since $I$ is directed, we may assume  
	$j_0 \geq i$. Let $\gamma \in \Hom R{Q}{M_{j_0}}$ be such that $h\iota_Q = f_{j_0}\gamma$. 
	Therefore $f_i =f_{j_0} (\gamma \pi_Q\beta'\alpha)$. Let $\delta = \gamma\pi_Q\beta'\alpha \in \Hom R{M_i}{M_{j_0}}$ and note
	that $\delta$ factors through $C$. Since $f_{j_0}(\delta - f_{j_0i}) = f_i - f_i = 0$, by (C3), there exists $j_0 \leq j \in I$
	such that $f_{jj_0}(\delta - f_{j_0i}) = 0$. Finally, if $j \leq k \in I$, then $f_{ki} = f_{kj}f_{jj_0}f_{j_0i} = f_{kj_0}\delta$
	factors through $C$. This finishes the proof of the claim. 
	
	The rest of the proof is the same as in  Theorem~\ref{main6}: Define $(I,\leq')$ by $i \leq' j$ if either $i = j$ or $i < j$ and 
	$f_{ji}$ factors through an object of $\mathcal C$. The claim implies that $(I,\leq')$ is directed.
	It is easy to verify (C1)-(C3) to deduce that $\mathfrak E = (M_i,f_{ji} \mid i \leq' j \in I)$ satisfies
	$\varinjlim \mathfrak E = (M,f_i(i \in I))$. It remains to apply Lemma~\ref{change} to the direct system $\mathfrak E$ to get 
	$M \in \varinjlim \mathcal C$ (note that if $I$ contains the largest element $i$, then $M_i$ is a direct summand 
	of some $C \in \mathcal C$).
\end{proof}

The characterization of $\varinjlim \mathcal C$ obtained in Proposition~\ref{pprojectives} for $\mathcal C \subseteq \Add \rfmod R$ 
allows an extension of  Lemma~\ref{lenzing}(i) for $\varinjlim$ of a class of pure projective modules.

\begin{corollary}\label{pureprojlim}
Let $\mathcal C$ be a class of pure projective modules closed under finite direct sums. Then $\varinjlim \mathcal C$
is closed under direct limits, pure epimorphic images, pure submodules and pure extensions. 
\end{corollary}

\begin{proof}
Consider a pure exact sequence of modules $$0 \to K \stackrel{\mu}{\to} L \stackrel{\varrho}{\to} M \to 0\,.$$
Assume first that $L \in \varinjlim \mathcal{C}$. If $X \in \rfmod{R}$ and $f \in \Hom R{X}{M}$, there exists
$g \in \Hom R{X}{L}$ such that $f = \varrho g$. By Proposition~\ref{pprojectives}, $g$ factors through an object 
of $\mathcal C$ and so does $f$. Therefore $M \in \varinjlim \mathcal{C}$.

Consider now $X \in \rfmod{R}$ and $h \in \Hom R{X}{K}$. Since $L \in \varinjlim \mathcal C$, there exist 
$C \in \mathcal C$, $\alpha \in \Hom R{X}{C},\beta \in \Hom R{C}{L}$ such that $\beta\alpha = \mu h$. In the proof of 
Proposition~\ref{pprojectives} we showed that there are $Q \in \rfmod{R}$, $\beta' \in \Hom R {C}{Q}$ and $u \in \Hom R{Q}{L}$ 
such that $\mu h = u \beta'\alpha$. We claim that there exists $u' \in \Hom R{Q}{K}$ such that $h = u'\beta'\alpha$
which gives a factorization of $h$ through $\mathcal C$. 

Consider the following commutative diagram with exact rows
$$\begin{CD}
0@>>> X/\Ker {\beta'\alpha}@>{\nu}>> Q @>\pi>> Q/ \im \beta'\alpha @>>> 0\\
@. @V{h'}VV @V{u}VV @V\overline{u}VV@.\\
0@>>> K @>\mu>> L  @> \varrho>> M @>>> 0,\\
\end{CD}$$
where $\nu(x+ \Ker {\beta'\alpha}) = \beta'\alpha(x),h'(x+\Ker {\beta'\alpha}) = h(x)$ for every $x\in X$, 
$\pi$ is the canonical projection and $\overline{u}$ is induced by $u$. Since the bottom sequence is pure 
and $Q/\im \beta'\alpha$ is finitely presented, $\overline{u}$ factors through $\varrho$. By \cite[Lemma~8.4]{FS},
there exists $u'\in \Hom R{Q}{K}$ such that $u'\nu = h'$. Compose this equality with the canonical projection 
of $X$ onto $X/\Ker{\beta'\alpha}$ to obtain $h = u'\beta'\alpha$.

Conversely, assume $K,M \in \varinjlim \mathcal C$. Let $f \in \Hom R{X}{L}$ for some finitely presented module $X$.
Then there exist $C_1 \in \mathcal C$, $\alpha \in \Hom R{X}{C_1}$, $\beta \in \Hom R{C_1}{M}$ such that 
$\varrho f = \beta\alpha$. Since $C_1$ is pure projective, there exists $\gamma \in \Hom R{C_1}{L}$ such that 
$\beta = \varrho \gamma$. Note that $\varrho (f - \gamma\alpha) = 0$, therefore there exists $\delta \in \Hom R{X}{K}$ such
that $\mu \delta = f - \gamma\alpha $. Since $K \in \varinjlim \mathcal C$, there are $C_2 \in \mathcal C$, 
$\varepsilon \in \Hom R{X}{C_2}, \eta \in \Hom R{C_2}{K}$ such that $\delta = \eta \varepsilon$.  
Note that $\left(\begin{array}{c}\alpha \\ \varepsilon \end{array}\right)  \in \Hom R{X}{C_1 \oplus C_2}$ and 
$(\gamma,\mu\eta) \in \Hom R{C_1 \oplus C_2}{L}$ give a factorization of $f$ through $C_1 \oplus C_2 \in \mathcal C$.
\end{proof}

We will finish this section by considering the case of commutative rings and countably generated projective modules that coincide with their trace ideals. A basic example of this kind, over the ring $C \langle 0,1 \rangle$ of all continuous real functions on the closed unit interval $\langle 0,1 \rangle$, was constructed by Kaplansky - we will further discuss the particular case of $C \langle 0,1 \rangle$ in Example \ref{exKap} below.
         
\begin{example}\label{pureid} Let $R$ be a commutative ring. By \cite[2.12]{HP1}, countably generated pure ideals of $R$ coincide with the trace ideals of countably generated projective modules. Moreover, by \cite[Lemme 2]{La}, they also coincide with the ideals of $R$ generated by countable sets $\{ f_n \mid n < \omega \}$ of elements of $R$ such that $f_{n+1}f_n = f_n$ for each $n < \omega$. 

Let $I$ be such an ideal of $R$ and $\{ f_n \mid n < \omega \}$ be a generating set of $I$ as above. Notice that since $R/I$ is a countably presented flat module, $R/I$ has projective dimension $\leq 1$, whence $I$ is projective. Notice that $I$ is its own trace ideal, so by Corollary \ref{project}, $\varinjlim \add I = \varinjlim \Add I = 	\{M \in \ModR \mid M $ flat and $MI = M\}$.  

Let $S = \End {I_R}$. Since $R$ is commutative, the map $\mu : r \to -.r$ is a ring homomorphism from $R$ into $S$ which induces a functor from $\rmod S$ into $\rmod R$ by restriction of scalars. Also the ring $S$ is commutative, because for all $s, s^\prime \in S$ and $n < \omega$,
$$ss^\prime (f_n) = ss^\prime (f_n.f_{n+1}) = s(f_{n+1}.s^\prime(f_n)) = s(f_{n+1}).s^\prime(f_n) = f_{n+1}.s(f_{n}).s^\prime(f_{n}) =$$ 
$$ = f_{n+1}.s^\prime(f_{n}).s(f_{n}) = \dots = s^\prime(s(f_n)).$$ 
Moreover, the restriction of $\mu : r \to -.r$ to $I$ is monic, since $r = f_n.r^\prime$ and $f_{n+1}.r = 0$ imply $r = 0$. Also, $s(x.x^\prime) = x.s(x^\prime)$ for all $x, x^\prime \in I$ and $s \in S$, whence $\mu(I)$ is an ideal in $S$, and $\mu \restriction I$ is an $S$-module isomorphism of $I$ on to $\mu(I)$. As $\mu(I)$ is generated by the set $\{ \mu(f_n) \mid n < \omega \}$, we infer from \cite[Lemme 2]{La} (or \cite[2.12]{HP1}) that $\mu(I)$ is a pure ideal in $S$, so $I \cong \mu(I)$ is a projective $S$-module.  

As $S$ is commutative, $S = \End {I_R} = \End {I_S}$. It follows that $R$- and $S$-homomorphisms between arbitrary direct sums of copies of $I$ coincide. Thus the class $\varinjlim \add I$ is the same whether computed in $\rmod R$ or $\rmod S$. 

The former way yields $\varinjlim \add I = \{ F \otimes_S I \mid \, F \hbox{ a flat $S$-module }\}$ by Theorem \ref{limadd}. By the Flat Test Lemma \cite[19.17]{AF}, $F \otimes_S I \cong F.I$ as $S$-modules. Since $I^2 = I$, and $F \otimes_S I$ is a flat $S$-module whenever $F$ is such, we see that $\varinjlim \add I$ coincides with the class of all flat $S$-modules $F$ such that $F.I = F$. This is exactly what Corollary \ref{project} gives when $\varinjlim \add I$ is computed the latter way, i.e., in $\rmod S$.     
\end{example}

Example \ref{pureid} applies to all pure ideals $I$ in the ring of all continuous real functions on $\langle 0 , 1 \rangle$. In this particular case, we will determine the structure of the endomorphism rings of these ideals:

\begin{example}\label{exKap} Let $R = C \langle 0,1 \rangle$ be the ring of all continuous functions from the closed unit interval $\langle 0,1 \rangle$ into $\mathbb R$ with the ring operations defined pointwise (see \cite{deM}, \cite{ON}, or \cite[\S 9]{GPR}). Let $I$ be a pure ideal in $R$. By \cite[4.1(a)]{V}, $I$ is countably generated, whence $I$ fits the setting of Example \ref{pureid}. By Corollary \ref{project}, $\varinjlim \add I = \varinjlim \Add I$ is the class consisting of all flat modules $M$ such that $M.I = M$.   

For $f \in R$ we will denote by $z(f)$ the (closed) zero set of $f$, i.e., $z(f) = \{ x \in \langle 0 , 1 \rangle \mid f(x) = 0 \}$, and by $s(f) = \langle 0 , 1 \rangle \setminus z(f)$ the (open) support of $f$. Recall \cite[4.1(a)]{V} that pure ideals in $R$ correspond 1-1 to closed subsets of $\langle 0 , 1 \rangle$: a pure ideal $I$ defines the closed subset $X_I = \bigcap_{f \in I} z(f)$ while a closed subset $X$ defines the pure ideal $I_X$ consisting of all $f \in R$ such that $z(f)$ contains some open neighborhood of $X$. (The basic example of Kaplansky mentioned above is the particular case when $X = \{ 0 \}$). 

\smallskip

We will now determine the structure of the ring $S = \End {I_R}$ for an arbitrary pure ideal $I$ in $R$. By \ref{pureid}, $I$ is a projective module. 

The case when $I$ is finitely generated is trivial: $I$ is then a free module of rank $1$ by \cite[9.6(1)]{GPR}, so $S \cong R$.  

If $I$ is not finitely generated, then by (the proof of) \cite[9.6(2)]{GPR}, there is a countable set $J$ such that $I = \bigoplus_{j \in J} I_j$ where for each $j \in J$, $I_j$ is an indecomposable countably, but not finitely generated pure ideal in $R$, $X_{I_j} = \langle 0 , 1 \rangle \setminus O_j$, and $\{ O_j \mid j \in J \}$ is a set of pairwise disjoint open intervals in $\langle 0 , 1 \rangle$. 

By \cite[Lemme 2]{La}, for each $j \in J$, there is a subset $\{ f_{j,n} \mid n < \omega \}$ in $I_j$ such that $f_{j,n+1}f_{j,n} = f_{j,n}$ for each $n < \omega$, and $I_j$ is the union of the strictly increasing chain of ideals $( f_{j,n} R \mid n < \omega )$. Since $I_j = I_{X_{I_j}}$, $f_{j,n}$ vanishes at an open neighborhood of $\langle 0 , 1 \rangle \setminus O_j$. In particular, $s(f_{j,m}) \subseteq O_j$. Since $O_j \cap O_k = \emptyset$, we have $f_{j,m}.f_{k,n} = 0$ for all $m,n < \omega$ and all $j \neq k \in J$. It follows that $\Hom R{I_j}{I_k} = 0$ for all $j \neq k \in J$, whence $S \cong \prod_{j \in J} \End {I_j}$. 

\smallskip

It remains to compute $S = \End {I_R}$ when $I$ is a pure ideal in $R$ such that $X_I$ is the complement of a single open interval $\emptyset \neq O_I \subsetneq \langle 0 , 1 \rangle$. So either $O_I = (a,b)$, or $O_I = (a,1 \rangle$, or  $O_I = \langle 0, b )$, where $0 \leq a < b \leq 1$. In the first case, when $O_I = (a,b)$, we choose a strictly decreasing sequence $\bar{a} = ( a_n \mid n < \omega )$ and a strictly increasing sequence $\bar{b} = ( b_n \mid n < \omega )$ such that $a_0 < b_0$, $a =  \inf_{n < \omega} a_n$, $b = \sup_{n < \omega} b_n$. In the second case of $O_I = (a,1 \rangle$, we chose $\bar{a}$ as above, but let $\bar b$ be the constant sequence $b_n = 1$ ($n < \omega$). Symmetrically, in the third case of $O_I = \langle 0 , b )$, we chose $\bar{b}$ as above, but let $\bar a$ be the constant sequence $a_n = 0$ ($n < \omega$).

For each $n < \omega$, let $f_n \in R$ be such that $s(f_n) = (a_{n+1},b_{n+1})$ in the first case, and $s(f_n) = (a_{n+1},b_{n+1}\rangle$ and $s(f_n) = \langle a_{n+1},b_{n+1})$ in the second and third cases, and moreover $f_n \restriction \langle a_n,b_n \rangle = 1$ in all three cases. Then $f_{n+1} f_n = f_n$, so $I^\prime = \sum_{n < \omega} f_n R$ is a pure ideal in $R$ by \cite[Lemma 2]{La}. Since $O_I = \bigcup_{n < \omega} s(f_n)$, we infer that $X_I = \bigcap_{n < \omega} z(f_n) = X_{I^\prime}$. Thus $I^\prime = I$.    

Notice that $I$ is not self-small: indeed, for each $n < \omega$, let $0 \neq h_n \in I$ be such that $s(h_n) \subseteq ( a_{n+1},a_n )$ in the first and second cases and $s(h_n) \subseteq ( b_n, b_{n+1} )$ in the third case. Then $h_n.f_m = 0$ for $m < n < \omega$ and $h_n.f_m = h_n$ for $n < m < \omega$. So $z : I \to I^{(\omega)}$ defined by $z(f_n) = (f_n,h_0,\dots,h_{n-1},h_n.f_n,0,\dots )$ satisfies $f_n z(f_{n+1}) = z(f_n)$ for each $n < \omega$, whence $z$ defines an $R$-homomorphism proving that $I$ not self-small. 

\smallskip

Denote by $T$ the ring of all continuous functions from $O_I$ into $\mathbb R$. We will prove that $S \cong T$.

For each $g \in S$, we define $\varphi (g) \in T$ by $\varphi (g) \restriction s(f_n)  = g(f_{n+1}) \restriction s(f_n)$ for each $n < \omega$. This is possible since for all $n+1 < m < \omega$, $g(f_m) \restriction s(f_n) = g(f_{n+1}) \restriction s(f_n)$, because $f_{n+1} = f_{n+1} . f_m$, whence $g(f_{n+1}) = f_{n+1} . g(f_m)$, and $f_{n+1} \restriction s(f_n) = 1$. Conversely, for $t \in T$, we define $\psi(t) \in S$ by $\psi(t)(f_n) = f_n . t$ at $s(f_n) \subsetneq O_I$, and $\psi(t)(f_n) = 0$ at $z(f_n)$. This is correct since $\psi(t)(f_n) \in f_n R \subseteq I$, $\psi(t)(f_{n+1}).f_n = f_{n+1}.t.f_n = f_n . t = \psi(t) (f_n)$ at $s(f_n)$, and also $\psi(t) (f_{n+1}).f_n = 0 = \psi(t)(f_n)$ at $z(f_n)$. 

Notice that for all $g \in S$ and $n < \omega$, $g(f_n) = f_n . g(f_{n+1}) = f_n . \varphi(g) = \psi(\varphi(g))(f_n)$ at $s(f_n)$ while all these maps vanish at $z(f_n)$, so $g = \psi \varphi (g)$. Conversely, if $t \in T$, then for each $n < \omega$, $\varphi \psi (t) \restriction s(f_n) = \psi(t) (f_{n+1}) \restriction  s(f_n) = f_{n+1} . t \restriction  s(f_n) = t \restriction s(f_n)$, whence $\varphi \psi (t) = t$. 

It follows that $\varphi$ and $\psi$ are mutually inverse ring isomorphisms of $S$ and $T$. Let $\mu : R \to S$ be the canonical ring homomorphism $\mu : r \to -.r$, and $\nu : R \to T$ be the restriction ring homomorphism $r \to r \restriction O_I$. Then $\mu = \psi \nu$, and $\nu = \varphi \mu$, so the following diagram is commutative:
\[
\xymatrix{{R} \ar@{<->}[d]_{id} \ar[r]^{\mu} & {S} \ar@<0,5ex>[d]^{\varphi} \\ {R} \ar[r]^{\nu} & {T} \ar@<0,5ex>[u]^{\psi}} 
\]
\end{example}

\medskip

\section{The tilting case}\label{tilting}

We will now consider the particular case of (infinitely generated) tilting modules $T$ in more detail. If $T$ is $0$-tilting, i.e., $T$ is a projective generator, then $R$ is isomorphic to a direct summand of $T^n$ for some $n > 0$, whence $\varinjlim{\add T} = \varinjlim{\Add T} = \widetilde{\Add T} = \mathcal F _0$. However, the situation is much less clear already for infinitely generated $1$-tilting modules. In order to cover the case of arbitrary $n$-tilting modules, it will be convenient to deal with a slightly more general setting: 
  
Recall \cite[13.20]{GT} that tilting cotorsion pairs are characterized as the hereditary cotorsion pairs $\mathfrak C = (\mathcal A, \mathcal B)$ such that the class $\mathcal B$ is closed under direct limits, and $\mathcal A \subseteq \mathcal P _n$ for some $n < \omega$. The more general setting that we will be interested in here will neither require $\mathfrak C$ to be hereditary, nor $\mathcal A$ to consist of modules of bounded projective dimension. In particular, $\mathcal A$ will be allowed to contain modules of infinite projective dimension. 

\medskip
We fix our general notation for the rest of this section as follows: \emph{$\mathfrak C = (\mathcal A, \mathcal B)$ will denote a cotorsion pair in $\rmod R$ such that the class $\mathcal B$ is closed under direct limits.} By \cite[5.4]{AST}, there is a module $K \in \rmod R$ such that $\Add K = \Ker{\mathfrak C} = \mathcal A \cap \mathcal B$. Also, by \cite[6.1]{Sa}, $\mathcal B = ({\mathcal A}^{\leq \omega})^\perp$, and $\mathcal B$ is a definable class of modules, hence $\mathcal B = \widetilde{\mathcal B}$, cf.\ \cite[6.9]{GT}. Moreover, by \cite[5.3]{Sa} and \cite[3.3]{AST}, there is an \emph{elementary cogenerator} $C$ for $\mathcal B$, that is, a pure injective module $C \in \mathcal B$ that cogenerates $\rmod R$, such that each module from $\mathcal B$ is a pure submodule in a product of copies of $C$, and $\widetilde{\mathcal A} = {}^\perp C$. 

Two modules $M$ and $M^\prime$ will be called \emph{equivalent} in case $\Add M = \Add M^\prime$.

\medskip
If $\mathfrak C$ is a tilting cotorsion pair induced by a tilting module $T$, then we can just take $K = T$, and $\Ker{\mathfrak C} = \Add T$ completely determines $\mathfrak C$, as $\mathcal B = (\Ker{\mathfrak C})^{\perp_{\infty}}$. In particular, different tilting cotorsion pairs are induced by non-equivalent tilting modules. However, in our general setting, it may happen that $\mathfrak C \neq \mathfrak C ^\prime$, even if $\Ker{\mathfrak C} = \Ker{\mathfrak C ^\prime}$: 

\begin{example}\label{IwanagaG} Let $R$ be an Iwanaga-Gorenstein ring of infinite global dimension (e.g., a commutative noetherian local Gorenstein ring which is not regular). Let $\mathcal G \mathcal P$ denote the class of all Gorenstein projective modules, and $\mathcal I$ the class of all modules of finite injective dimension. Since $R$ is right noetherian, $\mathcal I$ is closed under direct limits, see \cite[6.7]{GT}. Consider the cotorsion pairs $\mathfrak C = (\mathcal G \mathcal P, \mathcal I)$ and $\mathfrak C ^\prime = (\mathcal P _0, \rmod R)$. Then $\mathfrak C \neq \mathfrak C ^\prime$, but $\Ker{ \mathfrak C} = \Ker{\mathfrak C ^\prime} = \mathcal P_0$, cf.\ \cite[8.13]{GT}.
\end{example}  

In the general notation above, we have

\begin{lemma}\label{aht+} $\varinjlim \Add K = \overline{\Add K} \subseteq \widetilde{\Add K} = \widetilde{\mathcal A} \cap \mathcal B$. 

Moreover, $\widetilde{\Add K}$ is a covering class closed under extensions. 
\end{lemma} 
\begin{proof} We have already noticed that the classes $\widetilde{\mathcal A} = {}^\perp C$ and $\mathcal B$ are closed under extensions and pure epimorphic images. In view of Lemma \ref{holmjorg} and Theorem \ref{describe}, it only remains to prove the inclusion $\widetilde{\mathcal A} \cap \mathcal B \subseteq \widetilde{\Add K}$.

Let $M \in \widetilde{\mathcal A} \cap \mathcal B$. Consider a special $\mathcal A$-precover $\rho$ of $M$, and the short exact sequence $0 \to B \to A \overset{\rho}\to M \to 0$. Since $M \in \mathcal B$, $A \in \mathcal A \cap \mathcal B = \Add K$. By the precovering property, the canonical presentation of $M$ as a pure epimorphic image of a module from $\mathcal A$ factorizes through $\rho$, whence $\rho$ is a pure epimorphism. Thus $M \in \widetilde{\Add K}$.    
\end{proof}

Of course, we always have $\varinjlim{\add K} \subseteq \varinjlim{\Add K} \subseteq \widetilde{\Add K}$. If these inclusions are equalities, then by the results above, the class $\mathcal L = \varinjlim{\add K}$ is a deconstructible class closed under direct limits and extensions, and $\mathcal L$ is covering. We will consider several instances when this occurs. The first one is an immediate corollary of Lemmas \ref{sigma} and \ref{selfsmall}:

\begin{corollary}\label{artinfg}
Let $K$ be a finitely generated $\sum$-pure split module (e.g., let $K$ be a finitely generated tilting module over an artin algebra). Then $\Add K = \varinjlim{\add K} = \varinjlim{\Add K} = \widetilde{\Add K}$. 
\end{corollary}

\medskip
We will now examine the problem of whether $\varinjlim \Add K = \widetilde{\Add K}$ in our general setting. Note that while we always have $\widetilde{\Add K} = \widetilde{\mathcal A} \cap \mathcal B$ by Lemma \ref{aht+}, and $\widetilde{\mathcal A} = \varinjlim (\mathcal A ^{< \omega})$ in case $K$ is tilting (see \cite[8.40]{GT}), the class $\mathcal B$, and hence also $\widetilde{\Add K}$, need not contain \emph{any} non-zero finitely generated modules, even if $K$ is tilting module of projective dimension $1$:

\begin{example}\label{i1} Let $R$ be a (commutative noetherian) regular local ring of Krull dimension $2$. Let $T$ be \emph{any} non-projective tilting module, and $(\mathcal A, \mathcal B)$ be the tilting cotorsion pair induced by $T$. Then $\mathcal B \subseteq \mathcal I _1$, so $\mathcal B \cap \rfmod R = \{ 0 \}$. Indeed, all non-zero finitely generated modules have injective dimension $2$. We refer to \cite[3.4]{PT} for more details. 

For an explicit instance of this phenomenon, let $\mathcal S$ be the set of all ideals of $R$ and $\mathcal A = \Filt {\mathcal S}$. Then $\mathcal A \subseteq \mathcal P_1$, and $(\mathcal A, \mathcal I_1)$ is a tilting cotorsion pair generated by a tilting module $T$ of projective dimension $1$. Since $R$ is a UFD, $\mathcal A ^{< \omega}$ is the class of all finitely generated torsion free modules, and $\varinjlim \mathcal A$ the class of all torsion-free modules. Hence $\widetilde{\Add T}$ is the class of all torsion free modules of injective dimension $\leq 1$, cf.\ \cite[5.4]{PT}.
\end{example}         

For countably presented modules in $\widetilde{\Add K}$, we have the following description:

\begin{lemma}\label{ctilde} Let $C \in ({\widetilde{\Add K}})^{\leq \omega}$. Then there exists a module $D \in \Add K$ such that $C \oplus D$ is a countable direct limit of modules from $\Add K$. 

In particular, if $\varinjlim \Add K$ is closed under direct summands then $C \in \varinjlim \Add K$.
\end{lemma}
\begin{proof}  
First, by \cite[3.4]{AST}, $C$ is a Bass module over ${\mathcal A}^{\leq \omega}$, that is, $C = \varinjlim_{i < \omega} A_i$ for a countable direct system $(A_i, f_{i+1,i} \mid i < \omega )$, such that $A_i \in \mathcal A ^{\leq \omega}$ for each $i < \omega$. 

Consider the exact sequence $0 \to A_0 \overset{\nu_0}\to B_0 \to A_0^\prime \to 0$, where $\nu_0$ is a special $\mathcal B$-preenvelope of $A_0$. Then $B_0 \in \mathcal B$ and $A_0^\prime \in \mathcal A$, whence $B_0 \in \mathcal A \cap \mathcal B = \Add K$. Possibly adding an element of $\Add K$, we can w.l.o.g.\ assume that $B_0 = K^{(\kappa_0)}$ for a cardinal $\kappa_0 > 0$.

Taking the pushout of $\nu_0$ and $f_{10}$, we obtain the following commutative diagram

$$\begin{CD}
0@>>> A_0@>{\nu_0}>> K^{(\kappa_0)}@>>> A_0^\prime @>>> 0\\
@. @V{f_{10}}VV @VVV @| @.\\
0@>>> {A_1} @>>> M_0 @>>> A_0^\prime @>>> 0.\\
\end{CD}$$

As above, the special $\mathcal B$-preenvelope of $M_0$ induces an exact sequence $0 \to M_0 \to B_1 \to A_0^{\prime\prime} \to 0$ with $B_1 \in \mathcal B$ and $A_0^{\prime\prime} \in \mathcal A$. This yields another commutative diagram,    

$$\begin{CD}
0@>>> A_0@>{\nu_0}>> K^{(\kappa_0)}@>>> A_0^\prime @>>> 0\\
@. @V{f_{10}}VV @V{g_{10}}VV @V{h_{10}}VV @.\\
0@>>> {A_1} @>{\nu_1}>> B_1 @>>> A_1^\prime @>>> 0.\\
\end{CD}$$

where $h_{10}$ is a monomorphism and $A_1^\prime \in \mathcal A$, because the cokernel of $h_{10}$ is isomorphic to $B_1/M_0 \cong A_0^{\prime\prime}$. Thus $B_1 \in \Add K$, and again, w.l.o.g., $B_1 = K^{(\kappa_1)}$ for a cardinal $\kappa_1 > 0$. Proceeding by induction, we obtain a direct system of short exact sequences 

$$\begin{CD}
@.     \dots  @.    \dots  @. \dots     @.\\
@.     @VVV       @VVV   @VVV     @.\\
0@>>> A_n@>{\nu_n}>> K^{(\kappa_n)}@>>> A_n^\prime @>>> 0\\
@. @V{f_{n+1,n}}VV @V{g_{n+1,n}}VV @V{h_{n+1,n}}VV @.\\
0@>>> {A_{n+1}} @>{\nu_{n+1}}>>  K^{(\kappa_{n+1})}@>>> A_{n+1}^\prime @>>> 0.\\
@.     @VVV       @VVV   @VVV     @.\\
@.     \dots   @.   \dots @. \dots     @.
\end{CD}$$

where $h_{n+1,n}$ ($n < \omega$) are monomorphisms with cokernels in $\mathcal A$. Its direct limit is the sequence

$$0 \to C \to \varinjlim_{i < \omega} K^{(\kappa_i)} \to D  \to 0$$

with $D$ countably $\mathcal A$-filtered, hence $D \in \mathcal A$. Since $C \in \mathcal B$, the latter sequence splits, whence $D \in \mathcal A \cap \widetilde{\Add K} = \Add K$ by Lemma \ref{aht+}.
\end{proof} 

Though $\widetilde{\Add K}$ need not contain any non-zero finitely generated modules, in some cases, $\widetilde{\Add K} = \varinjlim {(\widetilde{\Add K})^{\leq \omega}}$. Moreover, in those cases one can express the modules in ${\widetilde{\Add K}}$ in a stronger form, as $\aleph_1$-directed unions, or $\aleph_1$-direct limits, of modules from ${(\widetilde{\Add K})^{\leq \omega}}$ rather than the ordinary directed unions or direct limits.

\begin{definition}\label{aleph1-directed} 
Let $M$ be a module. 

Let $\mathcal M$ be a set of countably presented submodules of $M$ satisfying that $\mathcal M$ is closed under unions of countable chains, and each countable subset of $M$ is contained in a member of $\mathcal M$. Then $\mathcal M$ is called an \emph{$\aleph_1$-dense} system of submodules of $M$.

Let $\mathcal S$ be a set of submodules of $M$ such that each countable subset of $\mathcal S$ is contained in a member of $\mathcal S$ and $M = \bigcup \mathcal S$. Then $M$ is an \emph{$\aleph_1$-directed union} of the direct system $(\mathcal S, \subseteq)$. For example, if $\mathcal M$ is an $\aleph_1$-dense system of submodules of $M$, then $M$ is the $\aleph_1$-directed union of the direct system $(\mathcal M, \subseteq)$.

Let $(I,\leq)$ be an \emph{$\aleph_1$-directed poset}, i.e., a directed poset such that for each countable subset $C \subseteq I$ there exists $d \in I$ such that $c \leq d$ for all $c \in C$. A direct system $\mathcal D$ of modules from a class $\mathcal C$ is called \emph{$\aleph_1$-directed} provided that the underlying poset of the system, $(I,\leq)$, is $\aleph_1$-directed. Then $L = \varinjlim \mathcal D$ is called an \emph{$\aleph_1$-direct limit} of modules from $\mathcal C$.           
\end{definition}

\begin{lemma}\label{gtilde}  
\begin{enumerate}
\item Assume that the ring $R$ is countable. Then each module $M \in {\widetilde{\Add K}}$ is an $\aleph_1$-directed union of a direct system of its submodules, $\mathcal M$, such that $\mathcal M \subseteq (\widetilde{\Add K})^{\leq \omega}$.
\item Assume that $\mathcal A \subseteq \mathcal P _1$, and $K$ is a direct sum of countably generated modules. Then each module $M \in {\widetilde{\Add K}}$ is an $\aleph_1$-direct limit of modules from $(\widetilde{\Add K})^{\leq \omega}$.
\end{enumerate}
\end{lemma}  
\begin{proof}  
(i) Since $\widetilde{\mathcal A} = {}^\perp C$ and $C$ is pure injective, $\widetilde{\mathcal A}$ is closed under pure submodules by (the proof of) \cite[3.6]{AST}. So is the definable class $\mathcal B$. Since $R$ is countable, for each $M \in \widetilde{\Add K} = \widetilde{\mathcal A} \cap \mathcal B$, all the countable pure submodules of $M$ form an $\aleph_1$-dense system, $\mathcal M$, such that $\mathcal M \subseteq (\widetilde{\Add K})^{\leq \omega}$, whence $M$ is an $\aleph_1$-directed union of the modules from $\mathcal M$.   

(ii) Let $M \in \widetilde{\Add K}$, so there is a pure exact sequence $0 \to N \to K^{(X)} \to M \to 0$ for a set $X$. Since $K^{(X)} \in \mathcal A$, $K^{(X)}$ is a strict $\mathcal B$-stationary module \cite[4.2]{Sa}, and so is $N$ by \cite[4.4]{Sa}. In particular, $N$ is strict $C$-stationary, where $C$ is an elementary cogenerator for $\mathcal B$, so $N$ possesses an $\aleph_1$-dense system, $\mathcal D$, consisting of strict $C$-stationary submodules of $N$, such that $\Hom RNC \to \Hom RUC$ is surjective whenever $U$ is a directed union of modules from $\mathcal D$, see \cite[5.4]{Sa}.   

By our assumption on $K$, the module $K^{(X)}$ is a direct sum of countably generated modules, say $K^{(X)} = \bigoplus_{i \in I} K_i$. The rest of the proof proceeds similarly as the proof of \cite[3.3]{AST}: 

Consider the poset $J$ consisting of all the pairs $(D,Y)$ where $D \in \mathcal D$, and $Y$ is a countable subset of $I$ such that $D \subseteq \bigoplus_{i \in Y} K_i$. Then $J$ with componentwise inclusions is an $\aleph_1$-directed poset, and $M$ is an $\aleph_1$-direct limit of the system $\mathcal M = \{ (\bigoplus_{i \in Y} K_i)/D \mid (D,Y) \in J \}$ consisting of countably presented modules. 

Notice that each $h \in \Hom RDC$ extends to some $h^\prime \in \Hom RNC$, and since $M \in {}^\perp C$, also to $K^{(X)}$, and hence to $\bigoplus_{i \in Y} K_i$. Thus $\mathcal M \subseteq {}^\perp C = \widetilde{\mathcal A}$. Since $\mathcal A \subseteq \mathcal P _1$, $\mathcal B$ is closed under homomorphic images, whence $\mathcal M \subseteq \mathcal B$. We conclude that $\mathcal M \subseteq ({\widetilde{\Add K})^{\leq \omega}}$. 
\end{proof}

\begin{remark}\label{vNreg} If we leave the setting of cotorsion pairs $\mathfrak C = (\mathcal A, \mathcal B)$ with the class $\mathcal B$  closed under direct limits, then $\Ker{\mathfrak C}$ need not equal $\Add K$ for any module $K$. Moreover, the class $\mathcal B$ (and hence $\Ker{\mathfrak C}$) need not contain any non-zero countably generated modules. So there is no analog of Lemma \ref{gtilde} in general. 

For example,  if $R$ is a simple von Neumann regular ring which is not artinian, the dimension of $R$ over its center is countable (e.g., $R$ is countable), and $\mathfrak C = (\rmod R,\mathcal I _0)$, then there are no non-zero countably generated modules in the class $\mathcal I _0$, cf.\ \cite[3.3]{Tams}. 
\end{remark}

Lemmas \ref{ctilde} and \ref{gtilde} yield

\begin{corollary}\label{tilde} Assume that the class $\varinjlim \Add K$ is closed under direct limits, and either $R$ is countable, or $K$ is a direct sum of countably generated modules and $\mathcal A \subseteq \mathcal P_1$. Then $\varinjlim{\Add K} = \widetilde{\Add K}$. 
\end{corollary}

\medskip  
We finish this section by considering in more detail the case of modules over Dedekind domains. Since Dedekind domains are hereditary, all the cotorsion pairs $\mathfrak C = (\mathcal A, \mathcal B)$, such that the class $\mathcal B$ is closed under direct limits, are tilting. Moreover, all non-projective tilting modules are infinitely generated, and the only $\sum$-pure split tilting modules are the injective ones:  

Let $R$ be a Dedekind domain. Let $Q$ be the quotient field of $R$, so $Q/R \cong \bigoplus_{p \in \mSpec R} E(R/p)$ (cf.\cite[IV.3]{FS}). We will make use of the classification of tilting modules and classes known for this case: up to equivalence, tilting modules $T$ correspond 1-1 to subsets $P$ of $\mSpec R$ as follows. 

For $P \subseteq \mSpec R$, let $R_P$ be the (unique) module such that $R \subseteq R_P \subseteq Q$ and $R_P/R \cong \bigoplus_{p \in P} E(R/p)$. In particular, if $P = \mSpec R \setminus \{ q \}$ for a maximal ideal $q$, then $R_P = R_{(q)}$ is the localization of $R$ at $q$. For an arbitrary subset $P$ of $\mSpec R$, we have $R_P = \bigcap_{q \in \mSpec R \setminus P} R_{(q)}$; in particular, $R_P$ is a overring of $R$, and hence a Dedekind domain (cf.\ \cite[1.2]{C}).  

By \cite[14.30]{GT}, the tilting module corresponding to $P \subseteq \mSpec R$ is $T_P = R_P \oplus \bigoplus_{p \in P} E(R/p)$, the corresponding tilting class being $\mathcal B _P = \{ M \in \rmod R \mid M.p = M \hbox{ for all } p \in P \}$, the class of all $P$-divisible modules. (Notice that $T_P$ is countably generated, iff the set $P$ is countable.) The cotorsion pair generated by $T_P$ is $(\mathcal A_P, \mathcal B _P)$ where $\mathcal A _P = \Filt {\mathcal S _P}$, and $\mathcal S _P = \{ R/p \mid p \in P \} \cup \{ I \mid I \subseteq R \}$. Also, $\varinjlim \mathcal A _P = \varinjlim (\mathcal A _P ^{< \omega} ) = {}^\intercal (\mathcal S _P ^\intercal)$. 
    
\begin{theorem}\label{dedekind} Let $R$ be a Dedekind domain and $T \in \rmod R$ be a tilting module. Let $P$ be the subset of $\mSpec R$ such that $T$ is equivalent to $T_P$. Then $\varinjlim{\add T} = \varinjlim{\Add T} = \widetilde{\Add T} = \mathcal C _P$, where $\mathcal C _P$ is the class of all modules $M$ whose torsion part $T(M)$ is isomorphic to a direct sum of copies of $E(R/p)$ for $p \in P$, $M = T(M) \oplus N$, and $N$ is a torsion-free (= flat) $R_P$-module.     
\end{theorem}
\begin{proof} First, let $M \in \mathcal C _P$, so $M = T(M) \oplus N$ as above. Since $E(R/p) \in \add T$ for each $p \in P$, $T(M) \in \varinjlim \add T$ by Lemma \ref{directs}. Moreover, $N$ is a direct limit of a direct system of finitely generated free $R_P$-modules, whence also $N \in \varinjlim \add T$. This proves that $\mathcal C _P \subseteq \varinjlim \add T$.

It remains to prove that $\widetilde{\Add T} \subseteq \mathcal C _P$. Let $M \in \widetilde{\Add T} = (\varinjlim \mathcal A _P) \cap \mathcal B _P$ (by Lemma \ref{aht+}, as $\varinjlim (\mathcal A _P ^{< \omega} ) = {}^\intercal (\mathcal S _P ^\intercal)$). Let $T(M)$ be the torsion part of $M$. Then $T(M) = \bigoplus_{p \in \mSpec R} T_p$ where $T_p$ denotes the $p$-torsion part of $T(M)$ for each $p \in \mSpec R$ (see e.g.\ \cite[IV.3]{FS}). 

Since $T(M) \subseteq_* M$, we infer that $T_p \in \widetilde{\Add T}$ for all $p \in \mSpec R$, because both $\varinjlim \mathcal A _P$ and $\mathcal B _P$ are closed under pure submodules. If $p \in P$, then this means that $T_p$ is $p$-divisible, and hence divisible (= injective). So $T_p$ is isomorphic to a direct sum of copies of $E(R/p)$. If $q \in \mSpec R \setminus P$ and $0 \neq T_q \in \varinjlim \mathcal A _P = {}^\intercal (\mathcal S _P ^\intercal)$, then since the latter class is closed under submodules, $R/q \in \mathcal A _P ^{< \omega}$, in contradiction with $q \notin P$. Thus $T_q = 0$ for all $q \notin P$, and $T(M)$ is isomorphic to a direct sum of copies of $E(R/p)$ for $p \in P$. Being injective, $T(M)$ splits in $M$, so $M = T(M) \oplus N$ where $N$ is torsion-free and $p$-divisible for each $p \in P$.  

Consider the short exact sequence $0 \to R \to R_P \to \bigoplus_{p \in P} E(R/p) \to 0$. Applying the exact functor $- \otimes_R N$, we get 
$0 \to N \to R_P \otimes _R N \to {\bigoplus_{p \in P} E(R/p)} \otimes _R N = 0$. The latter tensor product is zero because $N$ is $p$-divisible for each $p \in P$. Thus $N$ is a torsion-free $R_P$-module. 
\end{proof}

Notice that if $P = \mSpec R$, then $T_P = Q \oplus Q/R$ is $\sum$-injective, and $\mathcal C_P = \Add T_P = \mathcal I_0$. If $P \neq \mSpec R$, then the tilting module $T_P$ is not $\sum$-pure split, because there is a non-split pure exact sequence of the form $0 \to X \to R_P^{(\kappa)} \to Q \to 0$ for some infinite cardinal $\kappa$. In particular, $Q \in \widetilde{\Add T _P} \setminus \Add T_P$.  

\medskip

\section{A covering class of modules not closed under pure quotients}\label{npureq}

If $\mathcal A$ is a class of modules closed under direct sums, but $\varinjlim \mathcal A$ is not closed under direct limits, then $\varinjlim \mathcal A \subsetneq \widetilde{\mathcal A}$ (see Examples \ref{indec} and \ref{Bergman} above). The aim of this section is to construct a ring $R$ and an $R$-module $F$ such that the class $\mathcal A = \Add F$ is closed under direct limits, i.e., $\Add F=\varinjlim\Add F$, but it is still \emph{not} closed under pure epimorphic images in $\rmod{R}$. Then it is clear that $\Add F$ is a precovering class closed under direct limits, hence a covering class in $\rmod{R}$.
 
 In fact, for any given field~$k$, we will construct an associative, unital
$k$-algebra $R$ and an $R$-module $F_R$ with the following properties:
\begin{itemize}
\item $F$ is a Bass flat $R$-module;
\item $\End(F_R)=k$;
\item $P$ is a projective $R$-module and a pure submodule in~$F$;
\item $F/P$ is also a Bass flat $R$-module;
\item $\End(P_R)=k=\End((F/P)_R)$;
\item $\Sum F_R=\Add F_R=\varinjlim\Add F_R$;
\item but $F/P\notin\Add F_R$.
\end{itemize}
 Here by a \emph{Bass flat $R$-module} we mean a countable direct limit of
copies of the free $R$-module~$R$.

 The idea of the construction can be explained as follows.
 The $R$-module $F$ is produced as the direct limit of a countable chain
$F=\varinjlim_\omega(P_0\to P_1\to P_2\to\dotsb)$ where $P_i$ are cyclic
projective $R$-modules.
 The compositions $P_0\to P_1\to\dotsb\to P_i$ are split monomorphisms for
all $i>0$, but the splittings for $P_0\to P_i$ and $P_0\to P_{i+1}$ do
\emph{not} agree.
 In other words, in the diagram
\begin{equation} \label{direct-system-with-splittings}
\begin{gathered}
\xymatrix{
 P_0\ar[r] & \ \dotsb\ \ar[r] & P_i \ar@/^1pc/[ll]
 \ar[r] & P_{i+1} \ar@/^2.5pc/[lll] \ar[r] & \ \dotsb
}
\vspace*{2.5pc}
\end{gathered}
\end{equation}
the curvilinear triangle is not commutative.
 So $P=P_0$ is a pure submodule, but not a split submodule in
$\varinjlim_{i\in\omega}P_i=F$.

\begin{construction} \label{the-quiver}
 Consider the following infinite quiver with relations.  The vertices
of the quiver are indexed by~$\omega$.
 For every $n\in\omega$, $n\ge1$, there is an arrow $\phi_n$ going from
the vertex~$n-1$ to the vertex~$n$.
 For every $n\in\omega$, $n\ge1$, there is also an arrow $\pi_n$ going
from the vertex~$n$ to the vertex~$0$.
 The relations $\pi_n\phi_n\phi_{n-1}\dotsm\phi_1=\mathrm{id}$ for all
$n\ge1$ are imposed in the quiver.

 A \emph{quiver representation} $V$ is a collection of $k$-vector spaces
$V_i$, $i\in\omega$, endowed with linear maps
$\phi_n\colon V_{n-1}\to V_n$ and $\pi_n\colon V_n\to V_0$ such that
$\pi_n\phi_n\dotsm\phi_1=\mathrm{id}_{V_0}$ for all $n\ge1$.
\end{construction}

 A \emph{$k$-algebra with many objects} $A=(A_{i,j})_{i,j\in\omega}$ is
a collection of $k$-vector spaces $A_{i,j}$, identity elements
$e_i\in A_{i,i}$ given for all $i\in\omega$, and $k$-linear multiplication
maps $A_{i,j}\otimes_k A_{j,l}\to A_{i,l}$ defined for all $i$, $j$,
$l\in\omega$, satisfying the associativity and unitality equations.
 Here we presume the objects to be indexed by~$\omega$, as this will be
the case in the situation we are interested in.
 A $k$-algebra with many objects is the same thing as a small $k$-linear
category $A$ with the set of objects~$\omega$ and the vector spaces of
morphisms $\Hom Aji=A_{i,j}$.
 To a $k$-algebra with many objects one can assign a nonunital associative
$k$-algebra $A=\bigoplus_{i,j\in\omega}A_{i,j}$, which is a ring with enough
idempotents.

 Let $A$ be a $k$-algebra with many objects.
 A \emph{left $A$-module} $V$ is a collection of vector spaces $V_i$,
$i\in\omega$, and $k$-linear multiplication maps $A_{i,j}\otimes_k V_j
\to V_i$ satisfying the associativity and unitality equations with
the multiplication maps in $A$ and the identity elements $e_i\in A_{i,i}$.
 Similarly, a \emph{right $A$-module} $N$ is a collection of vector spaces
$N_i$, $i\in\omega$, and multiplication maps $N_i\otimes_k A_{i,j}
\to N_j$ satisfying the similar equations.
 A left $A$-module is the same thing as a covariant $k$-linear functor
from the category $A$ to the category of $k$-vector spaces; a right
$A$-module is a similar contravariant $k$-linear functor.

 To any right $A$-module $N$ and any left $A$-module $V$, one can assign
their \emph{tensor product}, which can be simply defined as the tensor
product of the right module $\bigoplus_{i\in\omega}N_i$ and the left
module $\bigoplus_{j\in\omega}V_j$ over the nonunital ring
$\bigoplus_{i,j\in\omega}A_{i,j}$.
 Explicitly, the tensor product $N\otimes_AV$ is the cokernel of
the difference of the natural pair of maps
$$
 \bigoplus_{i,j\in\omega} N_i\otimes_k A_{i,j}\otimes_k V_j
 \,\rightrightarrows\, \bigoplus_{i\in\omega} N_i\otimes_k V_i.
$$

 The category of right $A$-modules $\rmod A$ can be recovered as
the category of colimit-preserving covariant $k$-linear functors from
the category of left $A$-modules $\lmod A$ to the category of $k$-vector
spaces.
 To every right $A$-module $N$, the tensor product functor
$T_N\colon\lmod A\to\rmod k$ defined by the rule $T_N(V)=N\otimes_AV$ 
is assigned.
 For the specific algebra with many objects $A$ which we are interested in,
based on the following Construction~\ref{path-algebra}, we will use this
point of view on left and right $A$-modules: the left $A$-modules are
our quiver representations, and the right $A$-modules are functors from
quiver representations to vector spaces.

\begin{construction} \label{path-algebra}
 The \emph{path algebra} $A$ of the quiver from
Construction~\ref{the-quiver}, taken modulo the relations imposed in
the quiver, is a $k$-algebra with many objects such that
a left $A$-module is the same thing as a quiver representation~$V$.
 Explicitly, $A_{i,j}=\Hom Aji$ is the vector space of linear combinations
of paths from the vertex~$j$ to the vertex~$i$ up to the relations imposed
in the quiver.
 The vector space $A_{i,j}$ can be computed as $A_{i,j}=A'_{i,j}\oplus
A''_{i,j}$, where
$$
 A'_{i,j}=
 \begin{cases}
  k.\phi_i\phi_{i-1}\dotsm \phi_{j+1} & \text{if $i\ge j$}, \\
  0, & \text{if $i<j$},
 \end{cases}
$$
and
$$
 A''_{i,j}=
 \begin{cases}
 \bigoplus_{m\ge j} k.\phi_i\dotsm\phi_2\phi_1\pi_m\phi_m\phi_{m-1}
 \dotsm\phi_{j+1} & \text{if $j>0$}, \\
 0, & \text{if $j=0$}.
 \end{cases}
$$
So the first summand $A'_{i,j}$ is either a one-dimensional vector space,
or zero.
The second summand $A''_{i,j}$ is either a vector space of countable
dimension, or zero.
\end{construction}

 Let $Q_i$, $i\in\omega$, denote the free right $A$-module with
the generator sitting at the vertex~$i$.
 Explicitly, $Q_i=((Q_i)_j)_{j\in\omega}$, where $(Q_i)_j=A_{i,j}$
and the right action map $(Q_i)_j\otimes_k A_{j,l}\to (Q_i)_l$ is equal
to the multiplication map $A_{i,j}\otimes_k A_{j,l}\to A_{i,l}$.

 For every $n\in\omega$, $n\ge1$, the element $\phi_n\in A_{n,n-1}$
induces a homomorphism of right $A$-modules $Q_{n-1}\to Q_n$.
 Let $G\in\rmod A$ be the direct limit of this sequence of morphisms,
$$
 G=\varinjlim(Q_0\to Q_1\to Q_2\to\dotsb).
$$
 We also put $Q=Q_0$.
 Then, for every $i\in\omega$, the functor $T_{Q_i}=Q_i\otimes_A{-}$
takes a quiver representation $V$ to the vector space $V_i$; in
particular, $T_Q(V)=V_0$.
 The functor $T_G$ takes a quiver representation $V$ to the vector space
$$
 T_G(V)=\varinjlim(V_0\overset{\phi_1}\longrightarrow V_1
 \overset{\phi_2}\longrightarrow V_2\longrightarrow\dotsb).
$$
 There is a natural right $A$-module morphism $Q\to G$ which, being
tensored with a left $A$-module (quiver representation) $V$, induces
the canonical map of vector spaces
$$
 V_0\longrightarrow\varinjlim(V_0\overset{\phi_1}\longrightarrow V_1
 \overset{\phi_2}\longrightarrow V_2\longrightarrow\dotsb).
$$

 By \emph{Bass flat right $A$-modules} we mean the countable direct
limits of $\omega$-indexed sequences of right $A$-modules such that each
module in the sequence is isomorphic to the free right $A$-module with
one generator $Q_i$ for some $i\in\omega$.

\begin{lemma}
 The natural map $Q\to G$ is a pure monomorphism of right $A$-modules.
 Both $G$ and $G/Q$ are Bass flat right $A$-modules.
\end{lemma}

\begin{proof}
 To prove the first assertion, we have to show that the induced map
$Q\otimes_AV\to G\otimes_AV$ is injective for any left $A$-module~$V$.
 Viewing $V$ as a quiver representation, the map in question becomes
the natural map $V_0\to\varinjlim(V_0\overset{\phi_1}\longrightarrow V_1
\overset{\phi_2}\longrightarrow V_2\longrightarrow\dotsb)$.
 Now, for every $n\ge1$, the map $\phi_n\dotsm\phi_1\colon V_0\to V_n$
is a monomorphism of $k$-vector spaces, since (by the definition of
a quiver representation) it is retracted by the map $\pi_n\colon
V_n\to V_0$.
 Hence the map to the direct limit $V_0\to\varinjlim_{n\in\omega}V_n$
is also a monomorphism.
 
 $G$ is a Bass flat right $A$-module by definition, and it remains to
explain why $G/Q$ is a Bass flat right $A$-module.
 We have $G/Q=\varinjlim(0\to Q_1/Q\to Q_2/Q\to\dotsb)$.
 The point is that the composition $\phi_n\dotsm\phi_1\colon Q\to Q_n$
is a split monomorphism of free right $A$-modules with one generator
(the morphism $\pi_n\colon Q_n\to Q$ providing the related retraction) for
every $n\ge1$.
 It remains to use (the construction from the proof of)
Lemma~\ref{summands} in order to present $G/Q$ as a direct limit of
a certain sequence of morphisms $Q_1\to Q_2\to Q_3\to\dotsb$.
\end{proof}

\begin{proposition} \label{endomorphisms-computed}
 The endomorphism rings of the $A$-modules $G$, $Q$, and $G/Q$ are
\begin{enumerate}
\item $\End(G_A)=k$;
\item $\End(Q_A)=k$;
\item $\End((G/Q)_A)=k$.
\end{enumerate}
\end{proposition}

\begin{proof}
(i) We have
\begin{multline*}
 \Hom A{\varinjlim_{j\in\omega}Q_j}{\varinjlim_{i\in\omega}Q_i}=
 \varprojlim_{j\in\omega}\Hom A{Q_j}{\varinjlim_{i\in\omega}Q_i}=
 \varprojlim_{j\in\omega}\varinjlim_{i\in\omega}\Hom A{Q_j}{Q_i} \\
 = \varprojlim_{j\ge1}(k.\tau_j\oplus\bigoplus_{m\ge j}
 k.\tau_0\pi_m\phi_m\dotsm\phi_{j+1})=\varprojlim_{j\ge1}(k\oplus
 \bigoplus_{m\ge j}k)=k\oplus 0=k,
\end{multline*}
where $\tau_j\colon Q_j\to\varinjlim_{i\in\omega}Q_i$ is
the canonical morphism.
 Here the $j$-indexed projective system of infinite-dimensional vector
spaces and injective transition maps $\bigoplus_{m\ge j} k$ is
a projective system of subspaces of the vector space~$k^{(\omega)}$.
 The projective limit, i.e.\ the intersection of this family of
vector subspaces in $k^{(\omega)}$, is zero.

(ii) Holds because $A_{0,0}=k$.

(iii) We have
\begin{multline*}
 \Hom A{\varinjlim_{j\ge1}Q_j/Q}{\varinjlim_{i\ge1}Q_i/Q}=
 \varprojlim_{j\ge1}\Hom A{Q_j/Q}{\varinjlim_{i\ge1}Q_i/Q}= \\
 \varprojlim_{j\ge1}\varinjlim_{i\ge1}\Hom A{Q_j/Q}{Q_i/Q} 
 = \varprojlim_{j\ge1}(k.\bar\tau_j)=k,
\end{multline*}
where $\bar\tau_j\colon Q_j/Q\to\varinjlim_{i\in\omega}Q_i/Q$ is
the canonical morphism.
 This computation uses the fact that $Q_j/Q$ is a finitely presented
(in fact, a finitely generated projective) object in $\rmod A$.
\end{proof}

\begin{lemma} \label{G/Q-not-in-SumG}
 The $A$-module $G/Q$ does not belong to the class\/ $\Sum G$.
\end{lemma}

\begin{proof}
 Consider the following two quiver representations $V'$ and $V''$:
$$
 V'=(k\overset=\to k\overset=\to k\to\dotsb)
 \quad\text{and}\quad
 V''=(0\to k\overset=\to k\overset=\to k\to\dotsb).
$$
 Then we have $T_G(V')=k$ and $T_{G/Q}(V')=k/k=0$, while on the other
hand $T_G(V'')=k$ and $T_{G/Q}(V'')=k/0=k$.

 Suppose that there is an isomorphism $G/Q\cong G^{(\lambda)}$ for
some cardinal~$\lambda$.
 Then $0=T_{G/Q}(V')\cong T_G(V')^{(\lambda)}=k^{(\lambda)}$, hence
$\lambda=0$ is the zero cardinal and $G/Q=0$.
 However, $T_{G/Q}(V'')\ne0$.
 The contradiction proves that $G/Q\notin\Sum G$.
\end{proof}

\begin{construction}
 Let $A=(A_{i,j})_{i,j\in\omega}$ be an arbitrary $k$-algebra with
$\omega$~objects.
 We construct a unital $k$-algebra $R$ by adjoining a unit formally to
the nonunital $k$-algebra $\bigoplus_{i,j\in\omega}A_{i,j}$: so
$R=k.1\oplus\bigoplus_{i,j\in\omega}A_{i,j}$.
 Define the functor
$$
 \Theta\colon\rmod A\to\rmod R
$$
by the rule $\Theta(N_A)=\bigoplus_{i\in\omega}N_i=M_R$.
 Then $\Theta\colon\rmod A\to\rmod R$ is an exact, colimit-preserving,
fully faithful functor.
 The identity elements $e_i\in A_{i,i}$ form a family of orthogonal
idempotents $e_i\in R$, and the components of the $A$-module $N$ can
be recovered from the $R$-module $M$ by the rule $N_i=Me_i$.
 The essential image of the functor $\Theta$ consists of all the $R$-modules
$M$ such that $M=\sum_{i\in\omega}Me_i$.
\end{construction}

\begin{lemma}
\begin{enumerate}
\item The functor $\Theta$ takes (finitely generated) projective
$A$-modules to (finitely generated) projective $R$-modules.
\item The functor $\Theta$ takes (Bass) flat $A$-modules to (Bass)
flat $R$-modules.
\end{enumerate}
\end{lemma}

\begin{proof}
 In part~(i), it suffices to check that $\Theta$ takes the free $A$-module
with one generator $Q_i$, $i\in\omega$, to a finitely generated projective
$R$-module.
 Indeed, one has $\Theta(Q_i)=e_iR_R$.
 In part~(ii), the assertion concerning Bass flat modules now follows
immediately from the fact that the functor $\Theta$ preserves direct limits.
 The assertion about arbitrary flat modules can be deduced from their
Govorov--Lazard characterization as direct limits of (finitely generated)
projective ones, which is provable for rings with many objects similarly
to the classical case.
\end{proof}

 Finally, in order to produce the desired pair of $R$-modules, we put
$F=\Theta(G)$ and $P=\Theta(Q)\subset F$.
 To justify the notation in
the diagram~\eqref{direct-system-with-splittings}, we also put
$P_i=\Theta(Q_i)$ for all $i\in\omega$.
 All the properties of the modules $F$ and $P$ itemized in the beginning
of this section follow from the lemmas and proposition above, with
the exception of the last two ones, which still need to be checked.
 The following lemma does the job.

\begin{lemma}
 Let $R$ be a ring and $M_R$ be a module.
 Suppose that $D=\End(M_R)$ is a division ring.
 Then\/ $\Sum M_R=\Add M_R=\varinjlim\Add M_R$.
\end{lemma}

\begin{proof}
 By Lemma~\ref{summands}, we have $\varinjlim\Sum M_R=\varinjlim\Add M_R$,
so it suffices to check that $\Sum M_R=\varinjlim\Sum M_R$.
 Here one simply observes that any direct system in $\Sum M_R$ comes from
a direct system of $D$-vector spaces via the tensor product functor
${-}\otimes_D M$.
 The tensor product functor commutes with the direct sums and direct limits,
and it remains to recall that $\Sum D_D=\rmod D=\varinjlim\Sum D_D$.
\end{proof}

 We have proved that $\Sum F_R=\Add F_R=\varinjlim\Add F_R$; and the fact
that $F/P\notin\Add F_R$ now follows from Lemma~\ref{G/Q-not-in-SumG}
(as the functor $\Theta$ is fully faithful and preserves direct sums).

\begin{remark}
 Our example of a $k$-algebra $R$ and an $R$-module $F$ can be thought of
as universal in the following sense.
 Let $R'$ be a $k$-algebra with a family of orthogonal idempotents
$e'_i\in R'$, $i\in\omega$.
 Put $P_i'=e'_iR'\in\rmod{R'}$.
 Suppose that we are given an $\omega$-indexed direct system of $R'$-modules
$$
 P_0'\overset{\phi_1'}\longrightarrow P_1'\overset{\phi_2'}\longrightarrow
 P_2'\longrightarrow\dotsb
$$
together with retractions $\pi'_n\colon P_n'\to P_0'$ for the compositions
$\phi'_n\dotsm\phi'_1\colon P_0'\to P_n'$.

 Then there exists a unique $k$-algebra morphism $R\to R'$ such that
the direct system of $R'$-modules $P_0'\overset{\phi_1'}\longrightarrow
P_1'\overset{\phi_2'}\longrightarrow P_2'\longrightarrow\dotsb$ with
the splittings $\pi'_n\colon P_n'\to P_0'$ is obtained by applying
the functor ${-}\otimes_RR'$ to the direct system of $R$-modules
$P_0\overset{\phi_1}\longrightarrow P_1\overset{\phi_2}\longrightarrow 
P_2\longrightarrow\dotsb$ with the splittings $\pi_n\colon P_n\to P_0$.

 Indeed, the $k$-algebra $R$ is generated by the elements $e_i$, $\phi_n$,
and~$\pi_n\in R$.
 The $R'$-module maps $\phi'_n\colon P_{n-1}'\to P_n'$ and $\pi'_n\colon
P_n'\to P_0'$ give rise to naturally defined elements $\phi'_n\in R'$ and
$\pi'_n\in R'$.
 The $k$-algebra homomorphism $R\to R'$ takes the elements $e_i$ to $e_i'$,
the elements $\phi_n$ to $\phi_n'$, and the elements $\pi_n$ to $\pi_n'$.

 So, one can say that our aim in the constructions of this section was to
force the existence of the diagram~\eqref{direct-system-with-splittings},
which guaranteed that $P=P_0$ is a projective pure submodule in a Bass flat
module $F=\varinjlim_{i\in\omega}P_i$.
 We did so in a universal way.
 However, we also had to compute the endomorphism ring $\End(F_R)$ and
make sure that $\End(F_R)=k$.
 To simplify this computation, we assumed additionally that
$P_i=e_iR$ for a family of orthogonal idempotents $e_i\in R$, and we built
this additional assumption into our universal construction.
 Then Proposition~\ref{endomorphisms-computed}(i) became our main
computation, which showed that our universal construction does indeed
produce a module with the desired properties.
\end{remark}

\medskip

\section{Contramodule methods}
\label{contramodule-methods}

 The aim of this section is to formulate and prove
Theorem~\ref{lim-Add}, which is a version of Theorem~\ref{limadd} for
$\varinjlim\Add M$.
 Various applications of this theorem will be presented in
the subsequent sections.

 The exposition in this section, as well as in
Sections~\ref{contra-versus}--\ref{deconstructibility},
is based on the theory of contramodules over topological rings.
 We recall the basic concepts of this theory and the main results which
we need, and then proceed to the desired applications.
 Our main reference for contramodules is~\cite[Sections~6--7]{PS};
see also~\cite[Section~1.2]{Pweak}, \cite[Sections~1.1--1.2 and~5]{PR},
\cite[Section~2]{Pcoun}, \cite[Section~2]{Pproperf},
\cite[Sections~1--2]{BP}.

 The main difference between these papers and the exposition below is
that in the cited papers the general convention is to consider \emph{left}
contramodules, while in this section we work with \emph{right}
contramodules.
 The aim of this change of notation is to comply with the commonly
accepted convention in the ring and module theory, where scalars act
on the right and (endo)morphisms act on the left.

 A topological ring $S$ is said to be \emph{left linear} if
open left ideals form a base of neighborhoods of zero in $S$.
 The \emph{completion} $\mathfrak S$ of a left linear topological ring
$S$ is defined as the projective limit $\varprojlim_{I\subset S}S/I$,
where $I$ ranges over the open left ideals of $S$.
 A topological ring $S$ is said to be \emph{complete} if
the completion map $S\to\mathfrak S$ is surjective and \emph{separated}
if this map is injective.
 The completion $\mathfrak S=\varprojlim_{I\subset S}S/I$ is endowed
with a topological ring structure as explained
in~\cite[Section~2.2]{Pcoun}, making $\mathfrak S$ a complete,
separated left linear topological ring.
 The completion map $S\to\mathfrak S$ is a continuous ring homomorphism.

 A left module $M$ over a topological ring $S$ is said to be
\emph{discrete} if, for every element $b\in M$, the annihilator of $b$
is an open left ideal in $S$.
 The discrete left $S$-modules form a hereditary pretorsion
class $\ldiscr S$ in the category of all left $S$-modules $\lmod S$
(see Section~\ref{gabrieltopol} below for a discussion of pretorsion
and torsion classes).
 The module structure of any discrete left $S$-module can be uniquely
extended to a structure of discrete module over the completion
$\mathfrak S$ of $S$.

 Let $\mathfrak S$ be a complete, separated left linear topological ring.
 For every set $X$ and abelian group $A$, we will denote by $[X]A=A^{(X)}$
the direct sum of $X$ copies of~$A$.
 The elements of the group $[X]A$ are interpreted as finite linear
combinations $\sum_{x\in X}xa_x$ of elements of $X$ with the coefficients
in $A$ (so $a_x=0$ for all but a finite subset of the indices $x\in X$).
 Furthermore, we denote by $[[X]]\mathfrak S$ the projective limit
$\varprojlim_{\mathfrak I\subset\mathfrak S}[X](\mathfrak S/\mathfrak I)$,
where $\mathfrak I$ ranges over the open left ideals of $\mathfrak S$.
 The elements of the abelian group $[[X]]\mathfrak S$ are interpreted as
infinite linear combinations $\sum_{x\in X}xs_x$ of elements of $X$ with
the families of coefficients $(s_x\in\mathfrak S\mid x\in X)$ converging
to zero in the topology of $\mathfrak S$.
 Here the convergence means that, for every open left ideal $\mathfrak I
\subset\mathfrak S$, the set $\{x\in X\mid s_x\notin\mathfrak I\}$
is finite.

 The map assigning to every set $X$ the underlying set of the group
$[[X]]\mathfrak S$ is a covariant endofunctor on the category of sets,
and in fact, this functor is a \emph{monad} on the category of
sets~\cite[Section~6]{PS}, \cite[Section~5]{PR}, \cite[Section~2.7]{Pcoun}.
 This means that for every map of sets $f\colon X\to Y$ there is
the induced map $[[f]]\mathfrak S\colon [[X]]\mathfrak S\to
[[Y]]\mathfrak S$, and moreover, for every set $X$ there are natural maps
$\epsilon_X\colon X\to[[X]]\mathfrak S$ and $\phi_X\colon [[[[X]]
\mathfrak S]]\mathfrak S\to [[X]]\mathfrak S$ satisfying the associativity
and unitality equations of a monad.
 Here the monad unit $\epsilon_X$ is the ``point measure'' map defined
in terms of the unit element of the ring $\mathfrak S$, while
the monad multiplication $\phi_X$ is the ``opening of parentheses'' map
defined in terms of the multiplication and addition in the ring
$\mathfrak S$ and assigning a formal linear combination to a formal linear
combination of formal linear combinations.
 Infinite sums, which have to be computed in the constructions of
the maps $[[f]]\mathfrak S$ and $\phi_X$ (as the linear combinations
are infinite), are understood as the limits of finite partial sums in
the topology of $\mathfrak S$.
 The assumptions imposed above on the topology of $\mathfrak S$
are designed to guarantee the convergence.

 We are interested in modules over this monad (usually called ``algebras
over the monad'', but we prefer to call them modules because our context
is additive).
 Modules over the monad $X\mapsto [[X]]\mathfrak S$ on the category of
sets are called \emph{right\/ $\mathfrak S$-contramodules}.
 Explicitly, a right $\mathfrak S$-contramodule $\mathfrak C$ is a set
endowed with a \emph{right contraaction} map $\pi_{\mathfrak C}\colon
[[\mathfrak C]]\mathfrak S\to\mathfrak C$ satisfying the following
contraassociativity and contraunitality equations.
 The two compositions
\[
\xymatrix{
 {[[[[\mathfrak C]]\mathfrak S]]\mathfrak S}
 \ar@<0,5ex>[rr]^-{[[\pi_{\mathfrak C}]]\mathfrak S}
 \ar@<-0,5ex>[rr]_-{\phi_{\mathfrak C}}
 && [[\mathfrak C]]\mathfrak S
 \ar[r]^-{\pi_{\mathfrak C}}
 & {\mathfrak C}
}
\]
must be equal to each other,
$\pi_{\mathfrak C}\circ[[\pi_{\mathfrak C}]]\mathfrak S=
\pi_{\mathfrak C}\circ\phi_{\mathfrak C}$;
and the composition
\[
\xymatrix{
 {\mathfrak C} \ar[r]^-{\epsilon_{\mathfrak C}}
 & {[[\mathfrak C]]\mathfrak S} \ar[r]^-{\pi_{\mathfrak C}}
 & {\mathfrak C}
}
\]
must be equal to the identity map,
$\pi_{\mathfrak C}\circ\epsilon_{\mathfrak C}=\mathrm{id}_{\mathfrak C}$.

 In particular, given an associative ring $S$, one can endow $S$ with
the discrete topology (which is always complete, separated, and left
linear).
 Then the map $X\mapsto [X]S$ is a monad on the category of sets, and
modules over this monad are the same thing as right
$S$-modules~\cite[Section~6.1]{PS}.
 This is a fancy category-theoretic way to define modules over a ring
in terms of the forgetful functor assigning to the module its underlying
set and the related monad on the category of sets.

 For any set $X$, there is the obvious inclusion map $[X]\mathfrak S
\hookrightarrow [[X]]\mathfrak S$.
 Given a right $\mathfrak S$-contramodule $\mathfrak C$ with
the contraaction map $\pi_{\mathfrak C}\colon [[\mathfrak C]]\mathfrak S
\to \mathfrak C$, one can consider the composition $[\mathfrak C]
\mathfrak S\hookrightarrow [[\mathfrak C]]\mathfrak S
\xrightarrow{\pi_{\mathfrak C}}\mathfrak C$.
 The resulting map $[\mathfrak C]\mathfrak S\to\mathfrak C$ endows
$\mathfrak C$ with a right $\mathfrak S$-module structure.
 Thus the underlying $\mathfrak S$-module structure of
an $\mathfrak S$-contramodule is constructed.

 The category of right $\mathfrak S$-contramodules is denoted by
$\rcontra{\mathfrak S}$.
 So we have the forgetful functor $\rcontra{\mathfrak S}\to
\rmod{\mathfrak S}$.
 The category $\rcontra{\mathfrak S}$ is abelian and locally
$\mu^+$-presentable (in the sense of~\cite[Definition~1.17]{AR}),
where $\mu$ is the cardinality of a base of neighborhoods of zero
in $\mathfrak S$.
 The forgetful functor $\rcontra{\mathfrak S}\to\rmod{\mathfrak S}$
is exact and faithful, and preserves all products (but not coproducts).
 In particular, the forgetful functor, generally speaking, does
\emph{not} preserve direct limits (though it preserves all projective
limits and all $\mu^+$-direct limits, where $\mu$~is as above).
 Here the assertions involving~$\mu$ are based on the observation
that any zero-convergent family of nonzero elements in $\mathfrak S$
has cardinality~$\le\mu$ (if $\mu$~is infinite).

 The abelian category $\rcontra{\mathfrak S}$ has enough projective
objects (but, generally speaking, no nonzero injectives).
 The projective objects in $\rcontra{\mathfrak S}$ are described as
follows.
 For any set $X$, put $\mathfrak P=[[X]]\mathfrak S$ and
$\pi_{\mathfrak P}=\phi_X$.
 This defines an $\mathfrak S$-contramodule structure on $\mathfrak P$;
the resulting $\mathfrak S$-contramodule $\mathfrak P=[[X]]\mathfrak S$ is
called the \emph{free right\/ $\mathfrak S$-contramodule spanned by $X$}.
 For any right $\mathfrak S$-contramodule $\mathfrak C$, the group of
all $\mathfrak S$-contramodule morphisms $[[X]]\mathfrak S\to\mathfrak C$
is naturally isomorphic to the group of all maps of sets $X\to\mathfrak C$.
 The projective $\mathfrak S$-contramodules are precisely the direct
summands of the free ones.

 Let $\mathfrak C$ be a right $\mathfrak S$-contramodule and
$M$ be a discrete left $\mathfrak S$-module.
 The \emph{contratensor product} $\mathfrak C\odot_{\mathfrak S}M$
of $\mathfrak C$ and $M$ over $\mathfrak S$ \,\cite[Section~7.2]{PS},
\cite[Definition~5.4]{PR}, \cite[Section~2.8]{Pcoun} is an abelian group
constructed as the cokernel of the difference of two natural maps of
abelian groups
\[
\xymatrix{
 {[[\mathfrak C]]\mathfrak S\otimes_{\mathbb Z}M}
 \ar@<0,5ex>[rr]^-{t_{\mathfrak S}}
 \ar@<-0,5ex>[rr]_-{\pi_{\mathfrak C}\otimes_{\mathbb Z}M}
 && {\mathfrak C\otimes_{\mathbb Z}M},
}
\]
where $t_{\mathfrak S}\colon
[[\mathfrak C]]\mathfrak S\otimes_{\mathbb Z}M
\to \mathfrak C\otimes_{\mathbb Z}M$ is the map defined by the formula
\[
\xymatrix{
 {(\sum_{c\in\mathfrak C}cs_c)\otimes_{\mathbb Z} b}
 \ar@{|->}[r]^-{t_{\mathfrak S}}
 & {\sum_{c\in\mathfrak C}(c\otimes_{\mathbb Z}s_cb)}.
}
\]
 Here $(s_c\in\mathfrak S\mid c\in\mathfrak C)$ is a family of elements
converging to zero in the topology of $\mathfrak S$, and $b\in M$
is an element.
 The sum $\sum_{c\in\mathfrak C}cs_c$ denotes a formal linear
combination belonging to $[[\mathfrak C]]\mathfrak S$, while
the sum in the right-hand side is actually a finite sum of elements
of the tensor product $\mathfrak C\otimes_{\mathbb Z}M$.
 This sum is finite because one has $s_cb=0$ for all but a finite
subset of elements $c\in\mathfrak C$, since the annihilator of $b$
is an open left ideal in $\mathfrak S$, so all but a finite subset
of elements $s_c$ belong to this ideal.

 For any discrete left $\mathfrak S$-module $M$ and any right
$\mathfrak S$-contramodule $\mathfrak C$, there is a natural surjective
homomorphism of abelian groups
\begin{equation} \label{tensor-contratensor-surjection}
 \mathfrak C\otimes_{\mathfrak S}M\twoheadrightarrow
 \mathfrak C\odot_{\mathfrak S}M.
\end{equation}
 For any discrete left $\mathfrak S$-module $M$ and any set $X$, there
is a natural isomorphism of abelian groups
\begin{equation} \label{contratensor-with-free-contramodule}
 [[X]]\mathfrak S\odot_{\mathfrak S}M\cong [X]M=M^{(X)}.
\end{equation}
 The functor of contratensor product $-\odot_{\mathfrak S}-$ is right
exact and preserves coproducts in both its arguments.

 We refer to~\cite[Section~6]{PR} or~\cite[Section~5]{Pcoun} for
the definitions and discussion of \emph{complete} and \emph{separated}
contramodules.

 Let $M$ be a right $R$-module.
 Then the endomorphism ring $\mathfrak S=\End{M_R}$ carries a natural
complete, separated, left linear topology, called the \emph{finite
topology} (see~\cite[Section~7.1]{PS} and the references therein).
 By definition, a base of neighborhoods of zero in the finite topology on
$\mathfrak S$ is formed by the annihilators of finitely generated
$R$-submodules of $M$.
 It follows immediately from the definitions that $M$ is a discrete left
$\mathfrak S$-module.

 For any abelian category $\mathsf B$, let us denote by
$\mathsf B_{\mathrm{proj}}$ the class of all projective objects in
$\mathsf B$.
 We will consider $\mathsf B_{\mathrm{proj}}$ as a full subcategory in
$\mathsf B$; then $\mathsf B_{\mathrm{proj}}$ becomes an (additive)
category.
 Similarly, any class of right $R$-modules can be considered as a full
subcategory in $\rmod R$.
 In particular, viewing $\Add M$ as a full subcategory in $\rmod R$ makes
$\Add M$ an additive category.

 The following result was obtained in~\cite[Theorem~7.1 and
Proposition~7.3]{PS} in connection with applications to
infinitely-generated tilting theory.
 This result is called ``generalized tilting theory''
in~\cite[Section~2]{BP}.

\begin{theorem} \label{gentilt}
 Let $M$ be a right $R$-module and $\mathfrak S=\End{M_R}$ be its
endomorphism ring, endowed with the finite topology.
 Then there is a natural equivalence of additive categories
\begin{equation} \label{gentilt-equivalence}
\xymatrix{
 {\Psi_M\colon\Add M} \ar@{=}[r]
 & {(\rcontra{\mathfrak S})_{\mathrm{proj}}\,:\!\Phi_M}.
}
\end{equation}
 The functors $\Psi_M$ and $\Phi_M$ can be extended naturally to
a pair of adjoint functors between the whole abelian categories of
right $R$-modules and right $\mathfrak S$-contramodules:
\[
\xymatrix{
 {\Psi_M\colon\rmod R} \ar@<0,5ex>[r]
 & \ar@<0,5ex>[l] {\rcontra{\mathfrak S}\,:\!\Phi_M}.
}
\]
 Here the right adjoint functor $\Psi_M$ takes any $R$-module $N$ to
the abelian group $\Hom RMN$, which has a natural right
$\mathfrak S$-contramodule structure induced by the discrete left
$\mathfrak S$-module structure on $M$.
 The left adjoint functor $\Phi_M$ takes any $\mathfrak S$-contramodule
$\mathfrak C$ to the $R$-module $\mathfrak C\odot_{\mathfrak S}M$. \qed
\end{theorem}

 Now we can use the contramodule theory in order to obtain
the following variation on the theme of Theorem~\ref{limadd}.
 Given a cocomplete category $\mathsf B$ and a direct system
$(B_i,g_{ji}\mid i\le j\in I)$ in $\mathsf B$, we denote by
$\varinjlim^{\mathsf B}B_i$ the direct limit of the given direct
system \emph{computed in the category $\mathsf B$}.
 Given a class of objects $\mathsf C$ in $\mathsf B$, let us denote by
$\varinjlim^{\mathsf B}{\mathsf C}$ the class of all objects in
$\mathsf B$ which can be obtained as the direct limits of direct
systems of objects from $\mathsf C$ indexed by directed posets.
 Here, once again, the direct limit is presumed to be taken in
the category $\mathsf B$.

\begin{theorem} \label{lim-Add}
 Let $R$ be a ring, $M_R$ be a module and\/ $\mathfrak S=\End{M_R}$ be
its endomorphism ring, endowed with the finite topology.
 Then\/ $\varinjlim\Add M$ coincides with the class of all $R$-modules of
the form\/ $\mathfrak F\odot_{\mathfrak S}M$, where\/ $\mathfrak F\in
\varinjlim^{\rcontra{\mathfrak S}}
(\rcontra{\mathfrak S})_{\mathrm{proj}}$.
\end{theorem}

\begin{proof}
 Notice that the functor $\Phi_M=-\odot_{\mathfrak S}M$ preserves direct
limits (since it is a left adjoint).
 More explicitly, this means that $\Phi_M$ takes direct limits
computed in the category $\rcontra{\mathfrak S}$ to the conventional
direct limits of modules.
 Hence it follows immediately from Theorem~\ref{gentilt} that any module
of the form $\mathfrak F\odot_{\mathfrak S}M$ belongs to
$\varinjlim \Add M$.
 Conversely, let $(D_i,f_{ji}\mid i\le j\in I)$ be a direct system in
$\rmod R$ such that $D_i\in\Add M$ for all $i\in I$.
 Then, since \eqref{gentilt-equivalence}~is an equivalence of categories,
there exists a direct system $(\mathfrak P_i,g_{ji}\mid i\le j\in I)$
in $\rcontra{\mathfrak S}$ such that $\mathfrak P_i\in
(\rcontra{\mathfrak S})_{\mathrm{proj}}$ and there are isomorphisms
$D_i\cong\Phi_M(\mathfrak P_i)$ given for all $i\in I$, identifying
the morphisms $f_{ji}\colon D_i\to D_j$ with the morphisms
$\Phi_M(g_{ji})\colon \Phi_M(\mathfrak P_i)\to\Phi_M(\mathfrak P_j)$.
 Put $\mathfrak F=\varinjlim^{\rcontra{\mathfrak S}}{\mathfrak P_i}$;
then $\varinjlim D_i=\Phi_M(\mathfrak F)$.
\end{proof}

 It may be tempting to use the term ``flat contramodules'' for
$\mathfrak S$-con\-tra\-mod\-ules from the class
$\varinjlim^{\rcontra{\mathfrak S}}
(\rcontra{\mathfrak S})_{\mathrm{proj}}$.
 However, this term is already busy as a name for a (generally speaking)
wider class of contramodules~\cite{PR,Pproperf,BPS}.

\begin{definition} \label{flat-contramodules-defn}
 A right $\mathfrak S$-contramodule $\mathfrak F$ is called \emph{flat}
if the functor $\mathfrak F\odot_{\mathfrak S}\nobreak{-}\,\colon
\allowbreak\ldiscr{\mathfrak S}\to\rmod{\mathbb Z}$ is exact.
 All projective contramodules are flat
by~\eqref{contratensor-with-free-contramodule}.
 Since flatness of contramodules is preserved by direct limits, it follows
that all contramodules from $\varinjlim^{\rcontra{\mathfrak S}}
(\rcontra{\mathfrak S})_{\mathrm{proj}}$ are
flat~\cite[Lemmas~5.6 and~6.9]{PR}.
\end{definition}

\begin{remark} \label{flat-contramodules-rem}
 The theory of flat contramodules over complete, separated topological
rings \emph{with a countable base of neighborhoods of zero} is
rather well developed~\cite[Sections~5--7]{PR}.
 Still it is not known whether every flat $\mathfrak S$-contramodule
belongs to $\varinjlim^{\rcontra{\mathfrak S}}
(\rcontra{\mathfrak S})_{\mathrm{proj}}$ for topological rings
$\mathfrak S$ with a countable base of neighborhoods of zero.
 For topological rings without a countable base, this need \emph{not}
be true in general, as we will see in
Example~\ref{bad-flat}.
 But overall there are more questions than answers in the theory of
flat contramodules over such general topological rings at present
(see~\cite[Section~3]{Pproperf} and~\cite[Sections~2--3 and~7]{BPS},
or Section~\ref{deconstructibility} below).
 We will continue this discussion
in Remark~\ref{govorov-lazard-for-contra}.
\end{remark}

\medskip

\section{Contramodules in
the $\protect\varinjlim\add M$ versus
$\protect\varinjlim\Add M$ problem}
\label{contra-versus}

 In this section we deduce a corollary of Theorem~\ref{lim-Add} which will be
used in the proofs of Propositions~\ref{prufer-inverted-element}
and~\ref{prufer-inverted-subset} given in the next section.
 We also prove a couple of other (related) corollaries, one of which
will be used in order to obtain a more generally formulated application
to the $\varinjlim\add M$ versus $\varinjlim\Add M$ problem
in Section~\ref{gabrieltopol}.

 We start with a discussion of finitely presented contramodules.
 Let $\mathfrak S$ be a complete, separated, left linear topological ring.
 If $X$ is a finite set, then we will say that the free right
$\mathfrak S$-contramodule $[[X]]\mathfrak S$ is a \emph{finitely
generated free\/ $\mathfrak S$-contramodule}.
 The direct summands of finitely generated free contramodules are
called \emph{finitely generated projective}.
 If $f\colon\mathfrak C\to \mathfrak D$ is a morphism of finitely
generated free $\mathfrak S$-contramodules, then the cokernel of~$f$
is said to be a \emph{finitely presented} $\mathfrak S$-contramodule.
 We will denote the full subcategory of finitely presented contramodules
by $\rfcontra{\mathfrak S}\subset\rcontra{\mathfrak S}$ and
the full subcategory of finitely generated projective contramodules
by $(\rfcontra{\mathfrak S})_{\mathrm{proj}}\subset
(\rcontra{\mathfrak S})_{\mathrm{proj}}$.

 By the definition, one has $[[X]]\mathfrak S=[X]\mathfrak S$ for
a finite set~$X$.
 In other words, the forgetful functor $\rcontra{\mathfrak S}\to
\rmod{\mathfrak S}$ takes finitely generated free
$\mathfrak S$-contramodules to finitely generated free
$\mathfrak S$-modules.
 Hence it also takes finitely generated projective
$\mathfrak S$-contramodules to finitely generated projective
$\mathfrak S$-modules, and finitely presented
$\mathfrak S$-contramodules to finitely presented $\mathfrak S$-modules.

 Moreover, for a finitely generated free $\mathfrak S$-contramodule
$\mathfrak P$ and any $\mathfrak S$-con\-tra\-mod\-ule $\mathfrak C$,
the forgetful functor induces an isomorphism between the group of
all morphisms $\mathfrak P\to\mathfrak C$ in $\rcontra{\mathfrak S}$
and the group of all morphisms $\mathfrak P\to\mathfrak C$ in
$\rmod{\mathfrak S}$.
 Consequently, the same holds when $\mathfrak P$ is a finitely presented
$\mathfrak S$-contramodule.
 It follows that the forgetful functor $\rcontra{\mathfrak S}\to
\rmod{\mathfrak S}$ restricts to an equivalence between
the full subcategories of finitely presented $\mathfrak S$-contramodules
and finitely presented $\mathfrak S$-modules, $\rfcontra{\mathfrak S}
\simeq\rfmod{\mathfrak S}$.

 The latter equivalence, in turn, restricts to an equivalence between
the full subcategories of finitely generated projective
$\mathfrak S$-contramodules and finitely generated projective
$\mathfrak S$-modules.
 Denoting by $(\rfmod S)_{\mathrm{proj}}$ the category of finitely
generated projective right modules over an arbitrary ring $S$, we have
$(\rfcontra{\mathfrak S})_{\mathrm{proj}}\simeq
(\rfmod{\mathfrak S})_{\mathrm{proj}}$ for the topological
ring~$\mathfrak S$.

\begin{remark}
 The reader should be warned that the notion of a finitely presented
contramodule, as defined above, has nothing to do with
the category-theoretic concept of a finitely presentable object.
 In fact, the free $\mathfrak S$-contramodule with one generator
$\mathfrak S=[[\{0\}]]\mathfrak S$ is usually \emph{not} finitely
presentable as an object of the category $\rcontra{\mathfrak S}$,
because the functor $\Hom{\rcontra{\mathfrak S}}{\mathfrak S}{-}$,
which is isomorphic to the forgetful functor from
$\rcontra{\mathfrak S}$ to the category of sets, does not preserve
direct limits.
\end{remark}

 The next proposition is essentially Theorem~\ref{limadd} translated
into the contramodule language for the convenience of comparison with
Theorem~\ref{lim-Add}.

\begin{proposition} \label{lim-add-contra}
 Let $R$ be a ring, $M_R$ be a module and\/ $\mathfrak S=\End{M_R}$ be
its endomorphism ring, endowed with the finite topology.
 Then\/ $\varinjlim\add M$ coincides with the class of all $R$-modules of
the form\/ $\mathfrak F\odot_{\mathfrak S}M$, where\/ $\mathfrak F\in
\varinjlim^{\rcontra{\mathfrak S}}
(\rfcontra{\mathfrak S})_{\mathrm{proj}}$.
\end{proposition}

\begin{proof}
 There are natural equivalences of additive categories
\begin{equation} \label{small-gentilt-equivalence}
\xymatrix{
 {\add M} \ar@{=}[r]
 & {(\rfmod{\mathfrak S})_{\mathrm{proj}}} \ar@{=}[r]
 & {(\rfcontra{\mathfrak S})_{\mathrm{proj}}},
}
\end{equation}
where $\add M$ is viewed as a full subcategory in $\rmod R$ or in
$\Add M$.
 Here the functor $(\rfmod{\mathfrak S})_{\mathrm{proj}}\to
\add M$ takes a finitely generated projective right $\mathfrak S$-module
$P$ to the right $R$-module $P\otimes_{\mathfrak S}M$.
 The inverse functor $\add M\to(\rfmod{\mathfrak S})_{\mathrm{proj}}$
takes an $R$-module $N\in\add M$ to the finitely generated projective
$\mathfrak S$-module $\Hom RMN$ (see~\cite{Dr}).
 The equivalence $(\rfmod{\mathfrak S})_{\mathrm{proj}}\simeq
(\rfcontra{\mathfrak S})_{\mathrm{proj}}$ was explained in
the discussion above.
 The equivalence $\add M\simeq(\rfcontra{\mathfrak S})_{\mathrm{proj}}$
is obtained by restricting the equivalence of
categories~\eqref{gentilt-equivalence} from Theorem~\ref{gentilt}; so
it is given by the functors $\Psi_M$ and $\Phi_M$.
 The rest of the proof is very similar to the proof of
Theorem~\ref{lim-Add} and based on the fact that the functor
$\Phi_M=-\odot_{\mathfrak S}M$ preserves direct limits.
\end{proof}

\begin{remark} \label{govorov-lazard-for-contra}
 The classical Govorov--Lazard description of flat modules~\cite{Gov,La0}
can be thought of as the conjunction of two assertions: for any ring $S$,
\begin{enumerate}[label={(GL\arabic*)},
labelindent=\parindent,itemindent=*]
\item any flat $S$-module is a direct limit of projective $S$-modules,
and
\item any direct limit of projective $S$-modules is a direct limit of
finitely generated projective $S$-modules.
\end{enumerate}
 Let us consider the two analogous properties for contramodules over
a topological ring~$\mathfrak S$:
\begin{enumerate}[label={(C-GL\arabic*)},
labelindent=\parindent,itemindent=*]
\item all flat $\mathfrak S$-contramodules can be obtained as
direct limits of projective $\mathfrak S$-contramodules, that is
$(\rcontra{\mathfrak S})_{\mathrm{flat}}=
\varinjlim^{\rcontra{\mathfrak S}}
(\rcontra{\mathfrak S})_{\mathrm{proj}}$;
\item all direct limits of projective $\mathfrak S$-contramodules
can be obtained as direct limits of finitely generated projective
$\mathfrak S$-contramodules, that is
$\varinjlim^{\rcontra{\mathfrak S}}
(\rcontra{\mathfrak S})_{\mathrm{proj}}=
\varinjlim^{\rcontra{\mathfrak S}}
(\rfcontra{\mathfrak S})_{\mathrm{proj}}$.
\end{enumerate}
 Here $(\rcontra{\mathfrak S})_{\mathrm{flat}}\subseteq
\rcontra{\mathfrak S}$ denotes the class of all flat right
$\mathfrak S$-con\-tra\-mod\-ules, as defined in
Definition~\ref{flat-contramodules-defn}.
 When one of the conditions (C-GL1) or (C-GL2) holds,
it has consequences for module theory, as we will see
in Corollary~\ref{C-GL2-implies-cor} (for (CGL2)) and
in Section~\ref{deconstructibility} (for (C-GL1)).
\end{remark}

\begin{example} \label{bad-flat}
 The following counterexample, developing the idea
of~\cite[Remark~6.3]{Pcoun}, shows that the condition (C-GL1) need not hold
for a complete, separated, left linear topological ring $\mathfrak S$
in general.

 Let $\mathfrak S$ be the ring of (commutative) polynomials in
an uncountable set of variables $x_i$ over a field~$k$, and let
$T\subset\mathfrak S$ be the multiplicative subset generated by
the elements~$x_i$.
 We endow $\mathfrak S$ with the \emph{$T$-topology}, in which
the ideals $\mathfrak S t$, \,$t\in T$, form a base of
neighborhoods of zero.
 By~\cite[Proposition~1.16]{GT}, \,$\mathfrak S$ is a complete,
separated topological ring.
 It is easy to see that no infinite family of nonzero elements
in $\mathfrak S$ converges to zero in the $T$-topology; so
$\mathfrak S$-contramodules are the same thing as the usual
$\mathfrak S$-modules (formally speaking, the forgetful functor
$\rcontra{\mathfrak S}\to\rmod{\mathfrak S}$ is an equivalence
of categories).

 An $\mathfrak S$-module is discrete if and only if each element in it
is annihilated by some element from~$T$.
 The natural morphism from the tensor to the contratensor product
$\mathfrak C\otimes_{\mathfrak S}N\to\mathfrak C\odot_{\mathfrak S}N$
is an isomorphism for any $\mathfrak S$-contramodule $\mathfrak C$
and any discrete $\mathfrak S$-module $N$; so the contratensor product
over $\mathfrak S$ agrees with the tensor product.
 It follows easily that an $\mathfrak S$-contramodule $\mathfrak C$
is flat (in the sense of the definition in
Definition~\ref{flat-contramodules-defn}) if and only if
the $\mathfrak S/\mathfrak S t$-module $\mathfrak C/\mathfrak C t$
is flat for every $t\in T$.

 On the other hand, the projective $\mathfrak S$-contramodules
are the same thing as the projective $\mathfrak S$-modules, and
the class of all direct limits of projective
$\mathfrak S$-contramodules, coinciding with the class of all
direct limits of finitely generated projective
$\mathfrak S$-contramodules, is simply the class of all flat
$\mathfrak S$-modules.
 So condition (C-GL2) holds for the topological ring $\mathfrak S$,
but condition (C-GL1) does \emph{not} hold.
 For example, the cokernel $T^{-1}\mathfrak S/\mathfrak S$ of
the localization map $\mathfrak S\to T^{-1}\mathfrak S$ is a flat
$\mathfrak S$-contramodule, but not a flat $\mathfrak S$-module.
 The discussion of flat contramodules over this topological ring
$\mathfrak S$ will be continued in Example~\ref{bad-flat-continued}.
\end{example}

 We are not aware of any counterexamples to (C-GL2), however.

\begin{corollary} \label{C-GL2-implies-cor}
 Let $R$ be a ring, $M$ be a module and\/ $\mathfrak S=\End{M_R}$ be
its endomorphism ring, endowed with the finite topology.
 Assume that condition (C-GL2) holds for
right\/ $\mathfrak S$-contramodules.
 Then\/ $\varinjlim\add M_R=\varinjlim\Add M_R$. 
\end{corollary} 

\begin{proof}
 Compare Theorem~\ref{lim-Add} with Proposition~\ref{lim-add-contra}.
\end{proof}

 Let $S$ be a ring, $\mathfrak S$ be a complete, separated left linear
topological ring, and $\sigma\colon S\to\mathfrak S$ be
a ring homomorphism.
 Then the composition of forgetful functors $\rcontra{\mathfrak S}
\to\rmod{\mathfrak S}\to\rmod S$ defines an exact, faithful forgetful
functor $\rcontra{\mathfrak S}\to\rmod S$.
 The functor $\rcontra{\mathfrak S}\to\rmod S$ has a left adjoint
functor $\Delta_\sigma\colon\rmod S\to\rcontra{\mathfrak S}$, which
can be constructed as follows.

 First of all, as any left adjoint functor between abelian categories,
$\Delta_\sigma$ is right exact and preserves coproducts (so, in
particular, it preserves direct limits).
 Secondly, the action of $\Delta_\sigma$ on free modules is defined by
the rule $\Delta_\sigma([X]S)=[[X]]\mathfrak S$ (so $\Delta_\sigma$
takes the free $S$-module spanned by a set $X$ to the free
$\mathfrak S$-contramodule spanned by $X$).
 The action of $\Delta_\sigma$ on morphisms of free modules is easily
recovered from the adjunction property.
 Finally, in order to compute the image of an arbitrary right
$S$-module $E$ under $\Delta_\sigma$, one can choose a right exact
sequence $P_1\overset f\to P_0\to E\to 0$ with free $S$-modules
$P_0$, $P_1$; then the $\mathfrak S$-contramodule $\Delta_\sigma(E)$
is obtained from the right exact sequence $\Delta_\sigma(P_1)
\xrightarrow{\Delta_\sigma(f)}\Delta_\sigma(P_0)\to
\Delta_\sigma(E)\to0$.

 For any right $S$-module $E$, there is a natural adjunction morphism
$\delta_{\sigma,E}\colon E\to\Delta_\sigma(E)$ in the category of
right $S$-modules.
 Here the $S$-module structure on $\Delta_\sigma(E)$ is obtained by
applying the forgetful functor to the $\mathfrak S$-contramodule
structure.

\begin{lemma} \label{delta-contratensor}
 For any right $S$-module $E$ and any discrete left $\mathfrak S$-module
$M$, there is a natural isomorphism of abelian groups
\[
 \Delta_\sigma(E)\odot_{\mathfrak S}M\cong
 E\otimes_SM.
\]
\end{lemma}

\begin{proof}
 A natural map of abelian groups $E\otimes_SM\to
\Delta_\sigma(E)\odot_{\mathfrak S}M$ is constructed as the composition
$E\otimes_SM\to\Delta_\sigma(E)\otimes_SM\to\Delta_\sigma(E)
\odot_{\mathfrak S}M$, where the map $E\otimes_SM\to\Delta_\sigma(E)
\otimes_SM$ is induced by the adjunction morphism $\delta_{\sigma,E}$,
while $\Delta_\sigma(E)\otimes_SM\to\Delta_\sigma(E)
\odot_{\mathfrak S}M$ is the natural
surjection~\eqref{tensor-contratensor-surjection}.
 The resulting map $E\otimes_SM\to\Delta_\sigma(E)\odot_{\mathfrak S}M$
is an isomorphism for \emph{free} $S$-modules $E$ in view of
the natural isomorphism~\eqref{contratensor-with-free-contramodule}.
 Since both $-\otimes_SM$ and $\Delta({-})\odot_{\mathfrak S}M$ are
right exact functors, it follows that the natural morphism between
them is an isomorphism for all $S$-modules.
\end{proof}

 Given a topological ring $\mathfrak S$, one can always take the ring
$S=\mathfrak S$ and the identity morphism $\sigma=\mathrm{id}$.
 This special case of the constructions above appears in part~(ii) of
the next proposition.

\begin{proposition} \label{Delta-varinjlim-smallprojcontra}
 Let\/ $\mathfrak S$ be a complete, separated left linear topological
ring, $S$ be a ring, and $\sigma\colon S\to\mathfrak S$ be a ring
homomorphism.
\begin{enumerate}
\item For any flat right $S$-module $F$, the right\/
$\mathfrak S$-contramodule\/ $\Delta_\sigma(F)$ belongs to the class\/
$\varinjlim^{\rcontra{\mathfrak S}}
(\rfcontra{\mathfrak S})_{\mathrm{proj}}$.
\item A right\/ $\mathfrak S$-contramodule\/ $\mathfrak F$ belongs to\/
$\varinjlim^{\rcontra{\mathfrak S}}
(\rfcontra{\mathfrak S})_{\mathrm{proj}}$ \emph{if and only if}
there exists a flat right\/ $\mathfrak S$-module $F$ for which\/
$\Delta_{\mathrm{id}}(F)\cong\mathfrak F$.
\end{enumerate}
\end{proposition}

\begin{proof}
 (i) It is clear from the construction that the functor $\Delta_\sigma$
takes finitely generated free $S$-modules to finitely generated free
$\mathfrak S$-contramodules, and more generally finitely presented
$S$-modules to finitely presented $\mathfrak S$-contramodules.
 In particular, $\Delta_\sigma$ takes finitely generated projective
$S$-modules to finitely generated projective $\mathfrak S$-contramodules.
 Since the functor $\Delta_\sigma$ also preserves direct limits,
and $F$ is a direct limit of finitely generated projective $S$-modules,
the assertion follows.

 (ii) One observes that the restriction of $\Delta_{\mathrm{id}}$ to
$\rfmod{\mathfrak S}$ is the inverse functor to the forgetful
functor $\rfcontra{\mathfrak S}\to\rfmod{\mathfrak S}$.
 So the forgetful functor and the functor $\Delta_{\mathrm{id}}$, restricted
to $\rfmod{\mathfrak S}$ and $\rfcontra{\mathfrak S}$, provide
the equivalence between these two categories that was discussed in
the beginning of this section.
 In particular, these two functors restrict to mutually inverse equivalences
between the categories of finitely generated projective (contra)modules
$(\rfcontra{\mathfrak S})_{\mathrm{proj}}$ and
$(\rfmod{\mathfrak S})_{\mathrm{proj}}$.

 The ``if'' implication in~(ii) is a particular case of~(i); so we only
have to prove the ``only if''.
 Let $(\mathfrak P_i,f_{ji}\mid i\le j\in I)$ be a direct system in
$\rcontra{\mathfrak S}$ such that $\mathfrak P_i\in
(\rfcontra{\mathfrak S})_{\mathrm{proj}}$ for every $i\in I$
and $\mathfrak F=\varinjlim^{\rcontra{\mathfrak S}}\mathfrak P_i$.
 Denote by $P_i$ the underlying right $\mathfrak S$-module of
the right $\mathfrak S$-contramodule~$\mathfrak P_i$; then we have
$P_i\in(\rfmod{\mathfrak S})_{\mathrm{proj}}$ and
$\mathfrak P_i=\Delta_{\mathrm{id}}(P_i)$.
 Now $F=\varinjlim P_i$ is a flat right $\mathfrak S$-module
and $\Delta_{\mathrm{id}}(F)=\varinjlim^{\rcontra{\mathfrak S}}
\Delta_{\mathrm{id}}(P_i)=\mathfrak F$.
\end{proof}

\begin{corollary}
 Let $R$ be a ring, $M$ be a module and\/ $\mathfrak S=\End{M_R}$ be
its endomorphism ring, endowed with the finite topology.
 Let $S$ be a ring and\/ $\sigma\colon S\to\mathfrak S$ be a ring
homomorphism.
 Assume that for each\/ $\mathfrak F\in \varinjlim^{\rcontra{\mathfrak S}}
(\rcontra{\mathfrak S})_{\mathrm{proj}}$ there exists a flat right
$S$-module $F$ such that\/ $\mathfrak F\cong\Delta_\sigma(F)$.
 Then\/ $\varinjlim\add M_R=\varinjlim\Add M_R$.
\end{corollary}

\begin{proof}
 Let $N_R\in\varinjlim\Add M$.
 By Theorem~\ref{lim-Add}, there exists an $\mathfrak S$-contramodule
$\mathfrak F\in\varinjlim^{\rcontra{\mathfrak S}}
(\rcontra{\mathfrak S})_{\mathrm{proj}}$ such that $N\cong
\mathfrak F\odot_{\mathfrak S}M$.
 By assumption, there is a flat $S$-module $F$ such that
$\mathfrak F\cong\Delta_\sigma(F)$.
 Applying Lemma~\ref{delta-contratensor}, we obtain an isomorphism
$\mathfrak F\odot_{\mathfrak S}M\cong F\otimes_SM$.
 Since $F$ is a direct limit of finitely generated free $S$-modules,
it follows that $N\in\varinjlim\ssum M$ (cf.\
the proof of Theorem~\ref{limadd}).

 Alternatively, by Proposition~\ref{Delta-varinjlim-smallprojcontra}(i),
the assumption of the corollary implies property (C-GL2), so
it remains to invoke Corollary~\ref{C-GL2-implies-cor}.
\end{proof}

 Suppose that we are given a contramodule $\mathfrak F\in
\varinjlim^{\rcontra{\mathfrak S}}
(\rcontra{\mathfrak S})_{\mathrm{proj}}$.
 Where does one get a flat $S$-module $F$ such that $\Delta_\sigma(F)
\cong\mathfrak F$\,?
 In the rest of this section, as well as in
Sections~\ref{genprufer}--\ref{gabrieltopol}, we use (essentially)
one obvious approach: take $F=\mathfrak F$.
 This means that $F$ is the underlying $S$-module of the
$\mathfrak S$-contramodule~$\mathfrak F$.

 Generally speaking, it is far from obvious that this approach works
at all.
 Most importantly, there is \emph{no reason} for the underlying $S$-module
of $\mathfrak F$ to be flat.
 It is also not necessarily true that $\Delta_\sigma(\mathfrak F)=
\mathfrak F$.
 So we restrict ourselves to several special cases in which this
particular approach provides a solution.

 When does one have $\Delta_\sigma(\mathfrak F)=\mathfrak F$\,?
 Notice that the adjunction morphism $\Delta_\sigma(\mathfrak C)\to
\mathfrak C$ in $\rcontra{\mathfrak S}$ is an isomorphism for all
$\mathfrak S$-contramodules $\mathfrak C$ if and only if this morphism
is an isomorphism for all \emph{free} $\mathfrak S$-contramodules
(because both the forgetful functor and the functor $\Delta_\sigma$
are right exact).
 Furthermore, the composition of two adjoint functors
$\rcontra{\mathfrak S}\to\rmod S\to\rcontra{\mathfrak S}$ is the identity
functor if and only if the forgetful functor $\rcontra{\mathfrak S}\to
\rmod S$ is fully faithful.

 The forgetful functor $\rcontra{\mathfrak S}\to\rmod S$ need not be
fully faithful, of course (generally speaking), but surprisingly often
it is; see, e.~g., \cite[Theorem~6.10 and Example~7.10]{PS}
or~\cite[Section~6]{Pcoun}.

\begin{remark}
 The following results from the papers~\cite{PS,Pcoun} clarify
the situation a bit.
 For any right $\mathfrak S$-contramodule $\mathfrak C$ and any
discrete left $\mathfrak S$-module $K$, consider the composition of
natural surjective maps of abelian groups
\[
 \mathfrak C\otimes_S K
 \twoheadrightarrow\mathfrak C\otimes_{\mathfrak S}K
 \twoheadrightarrow  \mathfrak C\odot_{\mathfrak S}K,
\]
where the leftmost map is the obvious one and the rightmost map
is~\eqref{tensor-contratensor-surjection}.

 By~\cite[Lemma~7.11]{PS} or~\cite[proof of
Theorem~6.2(iii)$\Rightarrow$(ii)]{Pcoun}, if the forgetful functor
$\rcontra{\mathfrak S}\to\rmod S$ is fully faithful, then the map
$\mathfrak C\otimes_S K\to\mathfrak C\odot_{\mathfrak S}K$
is an isomorphism for all right $\mathfrak S$-contramodules $\mathfrak C$
and discrete left $\mathfrak S$-modules $K$.
 Conversely, if the map $\mathfrak C\otimes_S K\to\mathfrak C
\odot_{\mathfrak S}K$ is an isomorphism for all $\mathfrak C\in
\rcontra{\mathfrak S}$ and $K\in\ldiscr{\mathfrak S}$ \emph{and
the topological ring $\mathfrak S$ has a countable base of neighborhoods
of zero}, then the forgetful functor $\rcontra{\mathfrak S}\to\rmod S$
is fully faithful~\cite[Theorem~6.2]{Pcoun}.

 Hence we see that the condition about the map
$\mathfrak C\otimes_S K\to\mathfrak C\odot_{\mathfrak S}K$ being
an isomorphism is important for our purposes.
 One could arrive to the same conclusion much more directly, as
the formulation of the next corollary illustrates.
\end{remark}

\begin{corollary} \label{contramodules-flat-as-modules-cor}
 Let $R$ be an associative ring and $M$ be a right $R$-module.
 Let $\mathfrak S=\End(M_R)$ be the endomorphism ring of $M$, endowed
with the finite topology.
 Let $S$ be an associative ring and\/ $\sigma\colon S\to\mathfrak S$
be a ring homomorphism such that
\begin{enumerate}
\item for any right\/ $\mathfrak S$-contramodule\/ $\mathfrak C$ and
any discrete left\/ $\mathfrak S$-module $K$, the natural map from
the tensor product to the contratensor product
\[
 \mathfrak C\otimes_SK\to\mathfrak C\odot_{\mathfrak S}K
\]
is an isomorphism; and
\item for any right\/ $\mathfrak S$-contramodule\/ $\mathfrak F$ which
can be obtained as a direct limit of projective\/
$\mathfrak S$-contramodules in the category\/ $\rcontra{\mathfrak S}$,
the underlying right $S$-module of\/ $\mathfrak F$ is flat.
\end{enumerate}
 Then\/ $\varinjlim\add M_R=\varinjlim\Add M_R$.
\end{corollary}

\begin{proof}
 Take $K=M$, $\mathfrak C=\mathfrak F$, and compare
Theorem~\ref{lim-Add} with Theorem~\ref{limadd}.
\end{proof}

\medskip

\section{Generalized Pr\"ufer modules}
\label{genprufer}

 In this section we prove Propositions~\ref{prufer-inverted-element}
and~\ref{prufer-inverted-subset}.
 The arguments are based on the theory of contramodules over topological
rings, and more specifically on
Corollary~\ref{contramodules-flat-as-modules-cor}.

 We start with formulating the Artin--Rees lemma for centrally generated
ideals in left noetherian rings in the form suitable for our purposes.

\begin{lemma} \label{artin-rees}
 Let $R$ be a left noetherian ring and $I\subset R$ be an ideal generated
by central elements.
 Let $M$ be a finitely generated left $R$-module with a submodule
$N\subseteq M$.
 Then there exists an integer $m\ge0$ such that for all $n\ge0$
the following two submodules in $N$ coincide:
\[
 I^{n+m}M\cap N=I^n(I^mM\cap N).
\]
 Hence the inclusion $I^{n+m}M\cap N\subseteq I^nN$ holds for all
$n\ge0$.
\end{lemma}

\begin{proof}
 See~\cite[Exercise~1ZA(c) and Theorem~13.3]{GW}.
\end{proof}

\begin{lemma} \label{noetherian-flatness-lemma}
 Let $R$ be a left noetherian ring and $I\subset R$ be an ideal
generated by central elements.
 Let $(F_n)_{n\ge1}$ be a projective system of flat right
$R/I^n$-modules indexed by the integers $n\ge1$.
 Suppose that the transition map $F_n\to F_m$ is surjective for
all $n\ge m\ge 1$.
 Then $F=\varprojlim_{n\ge1}F_n$ is a flat right $R$-module.
\end{lemma}

\begin{proof}
 This is a noncommutative version of~\cite[Theorem~6.11]{Y}.
 We follow the argument spelled out in~\cite[Lemma~B.9.2]{Pweak},
where commutativity is (unnecessarily) assumed.
 It suffices to show that the tensor product functor $F\otimes_R{-}$ is
exact on the abelian category of finitely generated left $R$-modules
$\lfmod R$ (i.~e., as a functor from $\lfmod R$ to abelian groups).

 Consider the functor $G$, also acting from $\lfmod R$ to abelian groups
and defined by the rule $G(N)=\varprojlim_{n\ge1}(F_n\otimes_R N)$.
 Let us show that this functor is exact.
 Indeed, for any short exact sequence $0\to K\to L\to M
\to 0$ in $\lfmod R$ there are short exact sequences of left
$R/I^n$-modules $0\to K/(I^nL\cap K)\to L/I^nL\to M/I^nM\to0$.
 Taking the tensor products with $F_n$ over $R/I^n$ preserves exactness
of these short exact sequences, since $F_n$ is a flat $R/I^n$-module.
 The passage to the projective limits over~$n$ preserves exactness of
the resulting sequences of tensor products, because these are
countable directed projective systems of surjective maps.

 On the other hand, by Lemma~\ref{artin-rees}, the projective system of 
abelian groups $F_n\otimes_RK/(I^nL\cap K)$ is mutually cofinal with
the projective system $F_n\otimes_RK/I^nK$.
 This means that there are natural maps
\[
 F_{n+m}\otimes_R K/(I^{n+m}L\cap K)\to
 F_n\otimes_R K/I^nK \to
 F_n\otimes_R K/(I^nL\cap K),
\]
which form commutative diagrams with the transition maps in
the projective systems.
 After the passage to the projective limits over $n\ge1$, these two
maps become mutually inverse isomorphisms.
 So the natural morphism between the two projective limits
$\varprojlim_{n\ge1} F_n\otimes_R K/I^nK\to\varprojlim_{n\ge1}
F_n\otimes_R K/(I^nL\cap K)$ is an isomorphism, and we have shown that
the functor $G$ is exact.

 Now we have a natural morphism $F\otimes_RN\to
\varprojlim_{n\ge1}(F_n\otimes_R N)$ for all $N\in\lfmod R$.
 For finitely generated \emph{free} left $R$-modules $N$, this morphism
is obviously an isomorphism.
 Any morphism of right exact functors on $\lfmod R$ which is
an isomorphism for finitely generated free modules is an isomorphism
for all finitely generated modules.
 So the two functors $F\otimes_R{-}$ and $G$ are isomorphic.
 Since we have shown that the functor $G$ is exact, it follows that
the functor $F\otimes_R{-}$ is exact on $\lfmod R$;
so $F$ is a flat right $R$-module.
\end{proof}

\begin{proof}[Proof of Proposition~\ref{prufer-inverted-element}]
 Let $J\subset R$ denote the two-sided ideal of all elements $r\in R$
for which there exists $n\ge1$ such that $rt^n=0$.
 Then $R[t^{-1}]=(R/J)[t^{-1}]$.
 Passing from $R$ to $R/J$, we can assume without loss of generality
that $t$~is a nonzero-divisor in $R$.

 Then the endomorphism ring
$\mathfrak S=\End{(R[t^{-1}]/R)_R}$ can be computed as $\mathfrak S=
\varprojlim_{n\ge1}R/Rt^n$, and the finite topology on $\mathfrak S$ is
the topology of projective limit of the discrete rings $R/Rt^n$.
 So $\mathfrak S$ is simply the $t$-adic completion of $R$, with
the $t$-adic topology.
 Put $S=R$, and let $\sigma\colon S\to\mathfrak S$ be the completion
morphism.
 It suffices to check conditions~(i) and~(ii) from
Corollary~\ref{contramodules-flat-as-modules-cor}.

 Condition~(i) is almost obvious.
 By the definition $\mathfrak C\odot_{\mathfrak S}K$ is the quotient
group of $\mathfrak C\otimes_{\mathbb Z}K$ by the subgroup generated by
all elements of the form $\pi_{\mathfrak C}(\sum_{i=0}^\infty c_is_i)
\otimes\nobreak k-\sum_{i=0}^\infty c_i\otimes s_ik$, where
$s_i\in\mathfrak S$, $i<\omega$ is a sequence of elements converging
to zero in the topology of $\mathfrak S$, \ $c_i\in\mathfrak C$ is
an arbitrary sequence of elements, and $k\in K$.
 The definition of the tensor product $\mathfrak C\otimes_SK$ is similar
except that only finite sequences of elements $s_i\in S$ are allowed.
 We have to show that every element of the former form is, in fact,
equal to a certain element of the latter form in
$\mathfrak C\otimes_{\mathbb Z}K$.

 Now, since $K$ is a discrete $\mathfrak S$-module, for any $k\in K$
there exists $n\ge1$ such that $t^nk=0$.
 Any sequence of elements $s_i\in\mathfrak S$ converging to zero in
$\mathfrak S$ is the sum of a finite sequence of elements coming from
$S$ and an infinite sequence of elements from $\mathfrak S t^n$.
 Hence without loss of generality we can assume that
$s_i\in \mathfrak S t^n$ for all $i<\omega$.

 The point is that, for any sequence of elements $s_i\in \mathfrak S t^n$,
\,$i<\omega$, converging to zero in the topology of $\mathfrak S$, there
exists a sequence of elements $r_i\in\mathfrak S$ such that $s_i=r_it^n$
and the sequence $r_i$ also converges to zero in the topology of
$\mathfrak S$.
 In fact, $t^n$ is a non-zerodivisor in $\mathfrak S$ and the sequence $r_i=s_i/t^n$ converges to zero in $\mathfrak S$ whenever the sequence
$s_i$ does.
 Since $t^nk=0$, it follows that
\[
 \pi_{\mathfrak C}\left(\sum_{i=0}^\infty c_is_i\right)\otimes k-
 \sum_{i=0}^\infty c_i\otimes s_ik
 = \pi_{\mathfrak C}\left(\sum_{i=0}^\infty c_ir_i\right)t^n\otimes k-0=
 ct^n\otimes k-c\otimes t^nk,
\]
where $c=\pi_{\mathfrak C}(\sum_{i=0}^\infty c_ir_i)\in\mathfrak C$.
 The right-hand side is an element of the desired form (i.~e., an element
of the kernel of the map $\mathfrak C\otimes_{\mathbb Z}K\to
\mathfrak C\otimes_SK$).

 A more general approach, not relying on any non-zerodivisor arguments
or assumptions, can be found in~\cite[Corollary~6.7]{Pcoun}.

 We have checked condition~(i).
 Now Lemma~\ref{noetherian-flatness-lemma} together with the standard
theory of flat contramodules over topological rings with a countable
base of neighborhoods of zero (\cite[Section~D.1]{Pcosh}
or~\cite[Sections~5--6]{PR}; see the discussion in
Remark~\ref{flat-contramodules-rem}) yields condition~(ii).
 In particular, the standard theory tells that the direct limits of
projective contramodules are flat (as contramodules); that all
contramodules are complete (though not necessarily separated)
\cite[Lemma~D.1.1]{Pcosh} or~\cite[Lemma~6.3(b)]{PR},
while flat contramodules are complete and
separated~\cite[Section~D.1]{Pcosh} or~\cite[Corollary~6.15]{PR}.

 Essentially by definition, a right $\mathfrak S$-contramodule
$\mathfrak F$ is flat if and only if the right $S/St^n$-module
$\mathfrak F\odot_\mathfrak S S/St^n=\mathfrak F/\mathfrak F t^n$
is flat for every $n\ge1$.
 Since $\mathfrak F$ is complete and separated, we have $\mathfrak F=
\varprojlim_{n\ge1}\mathfrak F/\mathfrak F t^n$ (this is also explained
in the final paragraphs of the proof of~\cite[Lemma~B.9.2]{Pweak}).
 By Lemma~\ref{noetherian-flatness-lemma} (applied to the principal ideal
$I=St\subset S$), we can conclude that $\mathfrak F$ is a flat right
$S$-module.
\end{proof}

 In order to prove Proposition~\ref{prufer-inverted-subset}, we will use
the following version of Artin--Rees lemma for multiplicative subsets.
 It is obtained from Lemma~\ref{artin-rees} by specializing
from arbitrary (finitely centrally generated) ideals to principal ideals
generated by central elements, and then generalizing from multiplicative
subsets generated by a single element to arbitrary countable
multiplicative subsets.

 The argument in the proof of
Lemma~\ref{multiplicative-subset-Artin-Rees}, as well as the discussion
of a $T$-indexed projective system in the subsequent
Lemma~\ref{multiplicative-subset-noetherian-flatness}, will presume
the \emph{partial preorder of divisibility} on a central
multiplicative subset $T\subset R$: given two elements $s\in T$ and
$t\in T$, we say that $t\preceq s$ if $t$ divides $s$ in $R$, that is
$Rs\subseteq Rt$, or equivalently, there exists $r\in R$ such that $rt=s$.
 It is possible that $t\preceq s$ and $s\preceq t$; in this case,
the elements $s$ and $t$ are considered to be equivalent.

\begin{lemma} \label{multiplicative-subset-Artin-Rees}
 Let $R$ be a left noetherian ring and $T\subset R$ be a central
multiplicative subset.
 Let $M$ be a finitely generated left $R$-module with a submodule
$N\subseteq M$.
 Then there exists an element $t\in T$ such that for all $s\in T$
the following two submodules in $N$ coincide:
\[
 stM\cap N=s(tM\cap N).
\]
 Hence the inclusion $stM\cap N\subseteq sN$ holds for all $s\in T$.
\end{lemma}

\begin{proof}
 This proof is taken from~\cite{MO-Mohan}.
 For every $s\in T$, denote by $P_s\subseteq M$ the submodule consisting
of all elements $m\in M$ such that $sm\in N$.
 Clearly, $N\subseteq P_s\subseteq M$ and $P_{s'}\subseteq P_{s''}$
whenever $s'$ divides~$s''$ in~$R$.
 So $P_s$, \,$s\in T$ form an upwards directed family
of submodules in~$M$.
 Since the $R$-module $M$ is noetherian, there exists $t\in T$
such that $P_s\subseteq P_t$ for all $s\in T$.

 Let us show that $stM\cap N=s(tM\cap N)$.
 Indeed, the inclusion $s(tM\cap N)\subseteq stM\cap N$ is obvious.
 Now let $x\in stM\cap N$.
 Then $x=stm$ for some $m\in M$, and it follows that $m\in P_{st}$.
 Hence $m\in P_t$ and $tm\in N$ by the choice of~$t$.
 Therefore, $tm\in tM\cap N$ and thus $x=stm\in s(tM\cap N)$.
\end{proof}

 The next lemma is likewise obtained from
Lemma~\ref{noetherian-flatness-lemma} by specializing to principal ideals and then generalizing to countable multiplicative subsets.

\begin{lemma} \label{multiplicative-subset-noetherian-flatness}
 Let $R$ be a left noetherian ring and $T\subset R$ be a countable
multiplicative subset consisting of central elements.
 Let $(F_t)_{t\in T}$ be a projective system of flat right $R/Rt$-modules
indexed by~$T$.
 Suppose that the transition map $F_s\to F_t$ is surjective for all $t$
and $s\in T$ such that $t$ divides $s$ in $R$.
 Then $F=\varprojlim_{t\in T}F_t$ is a flat right $R$-module.
\end{lemma}

\begin{proof}
 As in the proof of Lemma~\ref{noetherian-flatness-lemma}, it suffices
to show that the tensor product functor $F\otimes_R{-}$ is exact on
the category of finitely generated left $R$-modules.

 Consider the functor $G$, also acting from $\lfmod R$ to abelian groups and defined by the formula
$N\longmapsto\varprojlim_{s\in T}(F_s\otimes_RN)$.
 Let us show that this functor is exact.

 Indeed, for any short exact sequence $0\to K\to L\to M\to0$
in $\lfmod R$ there are short exact sequences of $R/Rs$-modules
$0\to K/(sL\cap K)\to L/sL\to M/sM\to0$.
 Taking the tensor products with $F_s$ over $R/Rs$ preserves exactness of
these short exact sequences, since $F_s$ is a flat $R/Rs$-module.
 The passage to the projective limits over $s\in T$ preserves exactness of
the resulting sequences of tensor products, because there are countable
directed projective systems of surjective maps.

 On the other hand, Lemma~\ref{multiplicative-subset-Artin-Rees} implies
that the projective system of abelian groups $F_s\otimes_R K/(sL\cap K)$
is mutually cofinal with the projective system $F_s\otimes_R K/sK$.
 Let us explain how this follows.
 We have a natural surjective map of abelian groups $F_s\otimes_R K/sK
\to F_s\otimes_R K/(sL\cap K)$ induced by the natural epimorphism of
left $R$-modules $K/sK\to K/(sL\cap K)$ for all $s\in T$.
 Choosing an element $t\in T$ as
in Lemma~\ref{multiplicative-subset-Artin-Rees} for the submodule
$K\subseteq L$, we have $stL\cap K\subseteq sK$ for all $s\in T$.
 Hence there is also a surjective map of abelian groups
$F_{st}\otimes_R K/(stL\cap K)\to F_s\otimes_R K/sK$
induced by the epimorphisms $F_{st}\to F_s$ and $K/(stL\cap K)\to K/sK$
for all $s\in T$.
 All these maps of tensor products form commutative diagrams with
the transition maps in the two projective systems, in the obvious sense.
 Therefore, the related projective limits coincide,
$\varprojlim_{s\in T}(F_s\otimes_R K/(sL\cap K))\cong
\varprojlim_{s\in T}(F_s\otimes_R K/sK)$, and we have shown that
the functor $G$ is exact.

 Now we have a natural morphism $F\otimes_RN\to
\varprojlim_{s\in T}(F_s\otimes_RN)$ for all $N\in\lfmod R$.
 For finitely generated \emph{free} $R$-modules $N$, this morphism
is obviously an isomorphism.
 The argument finishes similarly to the proof of
Lemma~\ref{noetherian-flatness-lemma}.
\end{proof}

\begin{proof}[Proof of Proposition~\ref{prufer-inverted-subset}]
 Let $J\subset R$ denote the two-sided ideal of all elements $r\in R$
for which there exists $t\in T$ such that $rt=0$.
 Then $T^{-1}R=T^{-1}(R/J)$.
 Passing from $R$ to $R/J$, we can assume that all the elements of $T$
are nonzero-divisors in~$R$.

 Then the endomorphism ring $\mathfrak S=\End{(T^{-1}R/R)_R}$ can be
computed as $\mathfrak S=\varprojlim_{t\in T}R/Rt$, and the finite
topology on $\mathfrak S$ is the topology of projective limit of
the discrete rings $R/Rt$.
 So $\mathfrak S$ is simply the $T$-completion of~$R$
(see e.g.~\cite[Chapter~1]{GT}).
 Put $S=R$, and let $\sigma\colon S\to\mathfrak S$ be the completion
morphism.
 Let us check conditions~(i) and~(ii) from
Corollary~\ref{contramodules-flat-as-modules-cor}.

 Once again, condition~(i) is almost obvious.
 For any $k\in K$ there exists $t\in T$ such that $tk=0$.
 So it suffices to check that the map $\mathfrak C/\mathfrak C t\to
\mathfrak C\odot_{\mathfrak S} R/Rt$ is an isomorphism.
 The point is that, for any sequence of elements $s_i\in\mathfrak S t$ 
converging to zero in the topology of $\mathfrak S$, there exists
a sequence of elements $r_i\in\mathfrak S$ such that $s_i=r_it$ and
the sequence $r_i$ also converges to zero in the topology of~$\mathfrak S$.
 In fact, $t$ is a non-zerodivisor in $\mathfrak S$ and the sequence
$r_i=s_i/t$ converges to zero in $\mathfrak S$ whenever the sequence
$s_i$ does.
 A more general approach is provided by~\cite[Corollary~6.7]{Pcoun}.

 Similarly to the final paragraphs of the proof of
Proposition~\ref{prufer-inverted-element} above, the standard theory of
flat contramodules over topological rings with a countable base of
neighborhoods of zero (\cite[Section~D.1]{Pcosh}
or~\cite[Sections~5--6]{PR}) together with
Lemma~\ref{multiplicative-subset-noetherian-flatness}
yields condition~(ii).
 Essentially by definition, a right $\mathfrak S$-contramodule
$\mathfrak F$ is flat if and only if the right $S/St$-module
$\mathfrak F\odot_{\mathfrak S} S/St=\mathfrak F/\mathfrak F t$ is flat
for all $t\in T$.
 Any direct limit of projective contramodules is flat.
 Any flat contramodule $\mathfrak F$ is complete and separated
by~\cite[Section~D.1]{Pcosh} or~\cite[Lemma~6.3(b) and
Corollary~6.15]{PR}, so we have
$\mathfrak F=\varprojlim_{t\in T}\mathfrak F/\mathfrak F t$.
 By Lemma~\ref{multiplicative-subset-noetherian-flatness}, we can
conclude that $\mathfrak F$ is a flat right $S$-module.
\end{proof}

\medskip

\section{Gabriel topologies}
\label{gabrieltopol}

 The aim of this section is to formulate and prove a generalization of
Propositions~\ref{prufer-inverted-element}
and~\ref{prufer-inverted-subset} to modules $M_R$ whose endomorphism
ring $\mathfrak S$ is left noetherian and the finite topology on
$\mathfrak S$ satisfies a certain list of conditions.
 In fact, we will consider a more general setting in which the ring
$\mathfrak S$ itself is not necessarily noetherian, but it has a dense
noetherian subring on which the additional conditions are imposed.
 In addition to the application to the $\varinjlim\add M$ versus
$\varinjlim\Add M$ problem, we will show that properties (C-GL1)
and (C-GL2) from Remark~\ref{govorov-lazard-for-contra} hold for
some topological rings.

 Let $\mathfrak S$ be a left linear topological ring, $S$ be a ring,
and $\sigma\colon S\to\mathfrak S$ be a ring homomorphism.
 Then the ring $S$ can be endowed with the induced topology: the open
subsets (or open left ideals) in $S$ are the full preimages
under~$\sigma$ of the open subsets (respectively, open left ideals)
in $\mathfrak S$.
 This makes $S$ a left linear topological ring.
 When $\mathfrak S$ is separated and complete, and the image of~$\sigma$
is dense in $\mathfrak S$, the original topological ring $\mathfrak S$
can be recovered as the completion of the topological ring $S$, and
$\sigma$~is the completion map.

 Let $S$ be a ring.
 A class of modules $\mathcal T\subseteq \lmod S$ is called
a \emph{pretorsion class} if $\mathcal T$ is closed under direct sums
and epimorphic images in $\lmod S$.
 A pretorsion class $\mathcal T$ is said to be \emph{hereditary} if it
is closed under submodules.
 A pretorsion class is called a \emph{torsion class}~\cite{D} if it is
closed under extensions.

 Let $S$ be a left linear topological ring.
 Then the class of all discrete left $S$-modules $\ldiscr S$ is
a hereditary pretorsion class in $\lmod S$.
 One says that the topology on $S$ is a \emph{Gabriel topology}
if $\ldiscr S$ is a (hereditary) torsion class. 
 A left linear topology is Gabriel if and only if it satisfies
(the left version of) the axiom~T4 from~\cite[Section~VI.5]{S}.

 Let $S$ be a ring and $J\subset S$ be a two-sided ideal.
 One says that the ideal $J$ has the (left) \emph{Artin--Rees property}
if for any finitely generated left $S$-module $M$ with a submodule
$N\subseteq M$ there exists an integer $m>0$ such that $J^mM\cap N
\subseteq JN$.
 For other equivalent characterizations of ideals with the Artin--Rees
property, see~\cite[Theorem~2.1]{Sm}.
 Any ideal generated by central elements in a left noetherian ring $R$
has the left Artin--Rees property by Lemma~\ref{artin-rees}, and
moreover the same applies to so-called polycentral ideals and sums of
polycentral ideals~\cite[Corollary~2.8 and Theorem~6.3]{Sm}, but
generally speaking a two-sided ideal in a left noetherian ring need not
have the Artin--Rees property~\cite{Bor}.

 The following theorem is the main result of this section.
 
\begin{theorem} \label{gabriel-main-theorem}
 Let\/ $\mathfrak S$ be a complete, separated left linear topological ring.
 Let $S$ be a ring and $\sigma\colon S\to\mathfrak S$ be a ring
homomorphism with dense image; consider the induced topology on $S$.
 Assume that the ring $S$ is left noetherian, the induced topology on it
is a (left) Gabriel topology, and this topology has a countable base
consisting of two-sided ideals having the left Artin--Rees property.
 Then both the conditions (C-GL1) and (C-GL2) hold for the topological
ring\/~$\mathfrak S$.
 In other words, any flat right\/ $\mathfrak S$-contramodule belongs to
the class\/ $\varinjlim^{\rcontra{\mathfrak S}}
(\rfcontra{\mathfrak S})_{\mathrm{proj}}$.
\end{theorem}

 The proof of the theorem is based on the following generalization of
Lemmas~\ref{noetherian-flatness-lemma}
and~\ref{multiplicative-subset-noetherian-flatness}.

\begin{proposition} \label{with-AR-property-proposition}
 Let $S$ be a left noetherian ring and $S\supset J_1\supseteq J_2
\supseteq J_3\supseteq\dotsb$ be a descending chain of two-sided ideals, indexed by the integers $k\ge1$, such that all the ideals $J_k\subset S$
have the left Artin--Rees property and for each $k$, $m\ge1$ there exists
$l\ge k$ such that $J_l\subseteq J_k^m$.
 Let $(F_k)_{k\ge1}$ be a projective system of flat right
$S/J_k$-modules such that the transition map $F_l\to F_k$ is surjective
for all $l\ge k\ge1$.
 Then $F=\varprojlim_{k\ge1}F_k$ is a flat right $S$-module.
\end{proposition}

\begin{proof}
 Similarly to the proofs of Lemmas~\ref{noetherian-flatness-lemma}
and~\ref{multiplicative-subset-noetherian-flatness}, it suffices to
show that the tensor product functor $F\otimes_S{-}$ is exact on
the abelian category $\lfmod S$ of finitely generated left $S$-modules.
 For this purpose, we consider the functor $G$ defined by the rule
$G(N)=\varprojlim_{k\ge1}(F_k\otimes_S N)$, and check that this functor
is exact on $\lfmod S$.
 The question reduces to showing that, for any short exact sequence
$0\to K\to L\to N\to 0$ in $\lfmod S$ the map of projective limits
\begin{equation} \label{cofinality-one-direction}
 \varprojlim_{k\ge1} F_k\otimes_S K/J_kK \to
 \varprojlim_{k\ge1} F_k\otimes_S K/(J_kL\cap K)
\end{equation}
induced by the natural epimorphisms $K/J_kK\to K/(J_kL\cap K)$
is an isomorphism.

 Given an integer $k\ge1$, there exists $m=m(k)\ge1$ such that
$J_k^mL\cap K\subseteq J_kK$ (since the ideal $J_k$ has the Artin--Rees
property).
 Then, by assumption, there exists $l=l(k)\ge k$ such that
$J_l\subseteq J_k^m$.
 Hence $J_lL\cap K\subseteq J_kK$.
 We can choose the integers $l(k)$ in such a way that
$l(k+1)>l(k)$ for all $k\ge1$.
 Then the epimorphisms $F_{l(k)}\to F_k$ and the natural epimorphisms
$K/(J_{l(k)}L\cap K)\to K/J_kK$ induce a map of projective limits
\begin{equation} \label{cofinality-other-direction}
 \varprojlim_{k\ge1} F_{l(k)}\otimes_S K/(J_{l(k)}L\cap K) \to
 \varprojlim_{k\ge1} F_k\otimes_S K/J_kK.
\end{equation}
 It is clear that the maps~\eqref{cofinality-one-direction}
and~\eqref{cofinality-other-direction} are mutually inverse
isomorphisms.
 It follows that the functor $G$ is exact, and the argument finishes
similarly to the proofs of Lemmas~\ref{noetherian-flatness-lemma}
and~\ref{multiplicative-subset-noetherian-flatness}.
\end{proof}

 The next lemma provides a useful characterization of flat contramodules
over a topological ring with a base of open two-sided ideals.

\begin{lemma} \label{flat-contra-over-base-of-two-sided}
 Let $S$ be a left linear topological ring and\/ $\mathfrak S$ be
the completion of~$S$.
\begin{enumerate}
\item Let\/ $\mathfrak F$ be a flat right\/ $\mathfrak S$-contramodule
and $J\subset S$ be an open two-sided ideal.
 Then the right $S/J$-module\/ $\mathfrak F\odot_{\mathfrak S}S/J$ is flat.
\item Suppose that $S$ has a base of neighborhoods of zero consisting of
open two-sided ideals.
 Let\/ $\mathfrak F$ be a right\/ $\mathfrak S$-contramodule.
 Then\/ $\mathfrak F$ is a flat\/ $\mathfrak S$-contramodule if
and only if, for every open two-sided ideal $J\subset S$, the right
$S/J$-module\/ $\mathfrak F\odot_{\mathfrak S}S/J$ is flat.
\end{enumerate}
\end{lemma}

\begin{proof}
 (i) One observes that, for any left $S/J$-module $K$ and any right
$\mathfrak S$-contramodule $\mathfrak F$, there is a natural isomorphism
of abelian groups
\[
 (\mathfrak F\odot_{\mathfrak S} S/J)
 \otimes_{S/J}K \cong \mathfrak F\odot_{\mathfrak S}K.
\]
 It follows immediately that $\mathfrak F\odot_{\mathfrak S}S/J$ is
flat as an $S/J$-module whenever $\mathfrak F$ is flat as
an $\mathfrak S$-contramodule (in the sense
of Definition~\ref{flat-contramodules-defn}).

 (ii) Use the same natural isomorphism as in~(i), together with
the observation that, whenever open two-sided ideals $J$ form a base of neighborhoods of zero in $S$, all discrete left $S$-modules (or, which is
the same, discrete left $\mathfrak S$-modules) are direct unions of
$S/J$-modules.
\end{proof}

\begin{proof}[Proof of Theorem~\ref{gabriel-main-theorem}]
 Let $\mathfrak F$ be a flat right $\mathfrak S$-contramodule; we have
to show that $\mathfrak F\in\varinjlim^{\rcontra{\mathfrak S}}
(\rfcontra{\mathfrak S})_{\mathrm{proj}}$.
 By Proposition~\ref{Delta-varinjlim-smallprojcontra}(i), it suffices
to find a flat right $S$-module $F$ for which
$\mathfrak F\simeq\Delta_\sigma(F)$.
 Following the approach outlined in Section~\ref{contra-versus}, we
take $F$ to be underlying $S$-module of the $\mathfrak S$-contramodule
$\mathfrak F$.

 According to~\cite[Corollary~6.7]{Pcoun}, the forgetful functor
$\rcontra{\mathfrak S}\to\rmod S$ is fully faithful for the completion
map $\sigma\colon S\to\mathfrak S$ of any left linear topological
ring $S$ whose topology is Gabriel and has a countable base of
neighborhoods of zero consisting of finitely generated left ideals.
 It follows that the adjunction morphism $\Delta_\sigma(C)\to\mathfrak C$
is an isomorphism for any right $\mathfrak S$-contramodule $\mathfrak C$
and its underlying right $S$-module~$C$ (see~\cite[Proposition~I.1.3]{GZ}).
 Another assertion from~\cite[Corollary~6.7]{Pcoun} tells that
the natural map $\mathfrak C\otimes_SK\to\mathfrak C
\odot_{\mathfrak S}K$ is an isomorphism for all right
$\mathfrak S$-contramodules $\mathfrak C$ and discrete left
$\mathfrak S$-modules~$K$.

 It remains to show that, under the assumptions of the theorem,
the $S$-module $F$ is flat.
 By Lemma~\ref{flat-contra-over-base-of-two-sided}(i), the right $S/J$-module
$\mathfrak F\odot_{\mathfrak S}S/J$ is flat for every open two-sided
ideal $J\subset S$.
 As we have seen in the previous paragraph, in the situation at hand we
actually have $\mathfrak F\odot_{\mathfrak S}S/J=\mathfrak F\otimes_S S/J$;
so one can simply say that the right $S/J$-module
$\mathfrak F\otimes_S S/J$ is flat.

 Now any base of neighborhoods of zero is a directed poset by inverse
inclusion; and any countable directed poset has a cofinal chain inside it.
 Hence there exists a descending chain of open two-sided ideals
$S\supset J_1\supseteq J_2\supseteq J_3\supseteq\dotsb$ such that all
the ideals $J_k\subset S$ have the left Artin--Rees property and
the collection of all the ideals $(J_k\mid k\ge1)$ is a topology base
in $S$.

 Put $F_k=\mathfrak F\otimes_S S/J_k$.
 Then $F_k$ is a flat right $S/J_k$-module and the natural maps $F_l\to F_k$
are surjective for $l\ge k\ge 1$.
 By~\cite[Lemma~VI.5.3]{S}, the ideal $J_k^m$ is open in $S$ for each
$m\ge1$ (since it is a Gabriel topology); so there exists $l\ge k$
for which $J_l\subseteq J_k^m$.
 Thus Proposition~\ref{with-AR-property-proposition} tells that
$\varprojlim_{k\ge1}F_k$ is a flat right $S$-module.
 Finally, the natural map $\mathfrak F\to\varprojlim_{k\ge1}F_k$ is
an isomorphism, as all flat right contramodules over a left linear
topological ring with a countable base of neighborhoods of zero are
complete and separated by~\cite[Lemma~6.3(b) and Corollary~6.15]{PR}
(the particular case when the topological ring has a base of two-sided
ideals is also covered by the preceding exposition
in~\cite[Section~D.1]{Pcosh}).
\end{proof}

\begin{corollary} \label{gabriel-cor}
 Let $R$ be a ring, $M$ be a right $R$-module, and\/ $\mathfrak S=
\End{M_R}$ be its endomorphism ring, endowed with the finite topology.
 Let $S$ be a ring and $\sigma\colon S\to\mathfrak S$ be a ring
homomorphism with dense image.
 Assume that the ring $S$ is left noetherian, the induced topology on it
is a (left) Gabriel topology, and this topology has a countable base
consisting of two-sided ideals having the left Artin--Rees property.
 Then\/ $\varinjlim\add M_R=\varinjlim\Add M_R$.
\end{corollary}

\begin{proof}
 Follows immediately from Theorem~\ref{gabriel-main-theorem}
and Corollary~\ref{C-GL2-implies-cor}.
\end{proof}

\begin{remark}
 The following example illustrates the utility of considering a ring
homomorphism with dense image $\sigma\colon S\to \mathfrak S$ in
Theorem~\ref{gabriel-main-theorem} and Corollary~\ref{gabriel-cor},
rather than just always taking $S=\mathfrak S$.
 Let $R=\mathbb Z$ be the ring of integers and $T\subset R$ be
the multiplicative subset of all nonzero elements.
 Put $M=T^{-1}R/R=\mathbb Q/\mathbb Z$ (so this is even a particular
case of Proposition~\ref{prufer-inverted-subset}, as well as of
Lemma~\ref{injectlim}).
 Then $\mathfrak S=\End(M_R)$ is the product $\prod_p\mathbb J_p$
of the (topological) rings of $p$-adic integers, taken over all
the prime numbers~$p$.
 The ring $\mathfrak S$ is \emph{not} noetherian (indeed,
$\bigoplus_p\mathbb J_p$ is an infinitely generated ideal in
$\prod_p\mathbb J_p$).
 However, the topological ring $\mathfrak S$ has a dense noetherian
subring $S=R=\mathbb Z$, making Corollary~\ref{gabriel-cor} applicable.

 In fact, both Proposition~\ref{prufer-inverted-element} and
Proposition~\ref{prufer-inverted-subset} are special cases of
Corollary~\ref{gabriel-cor} (with $S=R$).
 The following corollary illustrates the applicability of the results
of this section in the context of Section~\ref{genprufer}.
\end{remark}

\begin{corollary}
 Let $R$ be a left noetherian ring.
\begin{enumerate}
\item Let $I\subset R$ be an ideal generated by central elements.
 Then the conditions (C-GL1) and (C-GL2) hold for right contramodules
over the topological ring\/ $\mathfrak S=\varprojlim_{n\ge1}R/I^n$ with
the projective limit (equivalently, $I$-adic) topology.
\item Let $T\subset R$ be a countable multiplicative subset
consisting of central elements.
 Then the conditions (C-GL1) and (C-GL2) hold for right contramodules
over the topological ring\/ $\mathfrak S=\varprojlim_{t\in T}R/Rt$ with
the projective limit topology (equivalently,
the $T$-topology~\cite[Section~1]{GT}).
\end{enumerate}
\end{corollary}

\begin{proof}
 In both cases, we only have to check applicability of
Theorem~\ref{gabriel-main-theorem}.
 In both cases, we put $S=R$, and let $\sigma\colon S\to\mathfrak S$
be the completion map.
 In case (i), the induced topology on $R$ is the $I$-adic topology, which
is a left Gabriel topology with a countable base of centrally
generated ideals $I^n\subset R$.
 In case (ii), the induced topology on $R$ is the $T$-topology, which is
a left Gabriel topology with a countable base of centrally
generated ideals $Rt\subset R$.
 In both cases, \cite[Lemma~2.3]{H} or~\cite[Lemma~3.1]{Pcoun} can be used
to show that the topology is Gabriel.
\end{proof}

\medskip

\section{Quasi-deconstructibility of flat contramodules}
\label{deconstructibility}

 The aim of this section is to improve the cardinality estimate for
deconstructibility of the class $\varinjlim\Add M$ in
Corollary~\ref{AddM-deconstr} under an additional assumption of
condition (C-GL1) from Remark~\ref{govorov-lazard-for-contra}.
 In fact, the cardinality estimate for deconstructibility of
$\varinjlim\Add M$ which we obtain under the assumption of
(C-GL1) is even better that the one for the class $\varinjlim\add M$
in Corollary~\ref{deconstr}.
 However, there is a \emph{caveat} that the estimates in this section
are only for the cardinalities of the sets of generators of the modules
involved, while the estimates in Corollaries~\ref{deconstr}
and~\ref{AddM-deconstr} are for the cardinalities of both the sets
of generators and relations.

 Let $\mathfrak S$ be a complete, separated, left linear topological ring.
 The derived functor of contratensor product $\Ctrtor*{\mathfrak S}--$
is constructed as follows.
 Given a right $\mathfrak S$-contramodule $\mathfrak C$, choose a resolution
of $\mathfrak C$ by projective $\mathfrak S$-contramodules $\mathfrak P_n$,
\[
 \dotsb\to\mathfrak P_3\to \mathfrak P_2\to\mathfrak P_1\to\mathfrak P_0
 \to\mathfrak C\to 0.
\]
 For any discrete left $\mathfrak S$-module $N$, set $\Ctrtor n{\mathfrak S}
{\mathfrak C}N$ to be the degree~$n$ homology group of the complex of
abelian groups $(\dotsb\to\mathfrak P_n\odot_{\mathfrak S}N\to
\mathfrak P_{n-1}\odot_{\mathfrak S}N\to\dotsb\mid n\ge0)$.
 Since the functor $-\odot_{\mathfrak S}N$ is right exact on
the abelian category $\rcontra{\mathfrak S}$, there is a natural isomorphism
of abelian groups $\Ctrtor0{\mathfrak S}{\mathfrak C}N\cong
\mathfrak C\odot_{\mathfrak S}N$.

 As any left derived functor, the functor $\Ctrtor*{\mathfrak S}--$ takes
short exact sequences in the resolved argument to long exact sequences of
the homology.
 So, for any short exact sequence of right $\mathfrak S$-contramodules
$0\to\mathfrak C\to\mathfrak D\to\mathfrak E\to0$ and any discrete left
$\mathfrak S$-module $N$, there is a long exact sequence of abelian groups
\begin{multline} \label{contramodule-argument-exact-seq}
 \dotsb\to\Ctrtor2{\mathfrak S}{\mathfrak E}N
 \to\Ctrtor1{\mathfrak S}{\mathfrak C}N
 \to\Ctrtor1{\mathfrak S}{\mathfrak D}N \\
 \to\Ctrtor1{\mathfrak S}{\mathfrak E}N
 \to\mathfrak C\odot_{\mathfrak S}N
 \to\mathfrak D\odot_{\mathfrak S}N
 \to\mathfrak E\odot_{\mathfrak S}N\to0. 
\end{multline}

 Furthermore, for any projective right $\mathfrak S$-contramodule
$\mathfrak P$, the functor $\mathfrak P\odot_{\mathfrak S}-$ is
exact on the abelian category $\ldiscr{\mathfrak S}$.
 Therefore, for any short exact sequence of discrete left
$\mathfrak S$-modules $0\to K\to L\to N\to 0$ and any right
$\mathfrak S$-contramodule $\mathfrak C$, there is a long exact sequence
of abelian groups
\begin{multline} \label{discrete-module-argument-exact-seq}
 \dotsb\to\Ctrtor2{\mathfrak S}{\mathfrak C}N
 \to\Ctrtor1{\mathfrak S}{\mathfrak C}K
 \to\Ctrtor1{\mathfrak S}{\mathfrak C}L \\
 \to\Ctrtor1{\mathfrak S}{\mathfrak C}N
 \to\mathfrak C\odot_{\mathfrak S}K
 \to\mathfrak C\odot_{\mathfrak S}L
 \to\mathfrak C\odot_{\mathfrak S}N\to0. 
\end{multline}

 An $\mathfrak S$-contramodule $\mathfrak F$ is said to be
\emph{$1$-strictly flat} if $\Ctrtor1{\mathfrak S}{\mathfrak F}N=0$
for all discrete $\mathfrak S$-modules~$N$.
 Equivalently, $\mathfrak F$ is $1$-strictly flat if and only if, for
any short exact sequence of $\mathfrak S$-contramodules $0\to\mathfrak C
\to\mathfrak D\to\mathfrak F\to0$ and any discrete $\mathfrak S$-module
$N$, the induced map of abelian groups $\mathfrak C\odot_{\mathfrak S}N
\to\mathfrak D\odot_{\mathfrak S}N$ is injective.
 Moreover, $\mathfrak F$ is called \emph{$\infty$-strictly flat} if
$\Ctrtor n{\mathfrak S}{\mathfrak F}N=0$ for all discrete
$\mathfrak S$-modules~$N$ and all integers $n>0$.
 It is clear from the exact
sequence~\eqref{discrete-module-argument-exact-seq} that any
$1$-strictly flat $\mathfrak S$-contramodule is flat.

\begin{remark}
 It is a basic fact of the classical theory of flat modules over a ring
$S$ that, for a given right $S$-module $F$, the functor of tensor product
$F\otimes_S-$ is exact on the category of left $S$-modules if and only if,
for every short exact sequence of right $S$-modules $0\to C\to D\to F\to0$
and every left $S$-module $N$, the induced map of abelian groups
$C\otimes_SN\to D\otimes_SN$ is injective.
 This is provable because, besides the flat right $S$-modules which
this assertion describes, there also exist enough flat left $S$-modules
(so the left $S$-module $N$ has a flat resolution).
 The point is that, in the theory of contratensor products over
a topological ring $\mathfrak S$, flat objects exist only in
the contramodule argument: \emph{nonzero flat discrete modules usually
do not exist}.
 This is one reason why the theory of flat contramodules is complicated.
\end{remark}

 Over a topological ring $\mathfrak S$ with a countable base of
neighborhoods of zero, all flat contramodules are $\infty$-strictly
flat~\cite[Remark~6.11 and Corollary~6.15]{PR}.
 The following lemma lists some properties of the class of all
$1$-strictly flat contramodules in the general case.

\begin{lemma} \label{1-strictly-flat-lemma}
\begin{enumerate}
\item The class of all\/ $1$-strictly flat\/ $\mathfrak S$-contramodules
is closed under extensions in\/ $\rcontra{\mathfrak S}$.
\item The class of all\/ $1$-strictly flat\/ $\mathfrak S$-contramodules
is closed under direct limits in\/ $\rcontra{\mathfrak S}$.
\item The kernel of any epimorphism from a flat contramodule to
a\/ $1$-strictly flat contramodule is flat.
\end{enumerate}
\end{lemma}

\begin{proof}
 (i) Follows immediately from the exact
sequence~\eqref{contramodule-argument-exact-seq}.

 (ii) This is~\cite[Corollary~7.1]{BPS} (based
on~\cite[Lemma~3.1]{Pproperf}).
 
 (iii) Let $0\to\mathfrak F\to\mathfrak G\to\mathfrak H\to0$ be a short
exact sequence of $\mathfrak S$-contramodules.
 Assume that the $\mathfrak S$-contramodule $\mathfrak H$ is $1$-strictly
flat; then it is clear from the exact
sequence~\eqref{contramodule-argument-exact-seq} that the short
sequence $0\to\mathfrak F\odot_{\mathfrak S}N
\to\mathfrak G\odot_{\mathfrak S}N\to\mathfrak H\odot_{\mathfrak S}N\to0$
is exact for any discrete $\mathfrak S$-module $N$.
 Let $0\to K\to L\to M\to0$ be a short exact sequence of discrete
$\mathfrak S$-modules; then the sequence $0\to\mathfrak F\odot_{\mathfrak S}K
\to\mathfrak F\odot_{\mathfrak S}L\to\mathfrak F\odot_{\mathfrak S}M\to0$ is
the kernel of the natural termwise surjective morphism from the sequence
$0\to\mathfrak G\odot_{\mathfrak S}K\to\mathfrak G\odot_{\mathfrak S}L\to
\mathfrak G\odot_{\mathfrak S}M\to0$ to the sequence
$0\to\mathfrak H\odot_{\mathfrak S}K\to\mathfrak H\odot_{\mathfrak S}L\to
\mathfrak H\odot_{\mathfrak S}M\to0$.
 The latter sequence is exact, since $\mathfrak H$ is a flat
$\mathfrak S$-contramodule.
 Assuming that $\mathfrak G$ is a flat $\mathfrak S$-contramodule,
the sequence $0\to\mathfrak G\odot_{\mathfrak S}K\to\mathfrak G
\odot_{\mathfrak S}L\to\mathfrak G\odot_{\mathfrak S}M\to0$ is exact
as well.
 Now the kernel of any termwise surjective morphism from a short exact
sequence to a short exact sequence is a short exact sequence.
 Hence the sequence $0\to\mathfrak F\odot_{\mathfrak S}K\to
\mathfrak F\odot_{\mathfrak S}L\to\mathfrak F\odot_{\mathfrak S}M\to0$
is exact, too, and the $\mathfrak S$-contramodule $\mathfrak F$ is flat.
\end{proof}

\begin{corollary} \label{flat-contra-implications-cor}
\begin{enumerate}
\item If all flat right\/ $\mathfrak S$-contramodules are\/ $1$-strictly
flat, then the class of all flat right\/ $\mathfrak S$-contramodules is
closed under extensions and the kernels of epimorphisms in\/
$\rcontra{\mathfrak S}$.
\item If all flat right\/ $\mathfrak S$-contramodules are\/ $1$-strictly
flat, then all of them are\/ $\infty$-strictly flat.
\item If condition (C-GL1) holds for right\/ $\mathfrak S$-contramodules,
then all flat right\/ $\mathfrak S$-contramodules are\/ $1$-strictly flat.
\end{enumerate}
\end{corollary}

\begin{proof}
 (i) Follows immediately from Lemma~\ref{1-strictly-flat-lemma}(i)
and~(iii).

 (ii) It is clear from (i) that all the syzygy contramodules in
a projective resolution of a flat contramodule are flat.
 So the projective resolution is obtained by splicing short exact
sequences of $1$-strictly flat contramodules.
 Considering the long exact sequence~\eqref{contramodule-argument-exact-seq}
for each of these short exact sequences, we conclude that the projective
resolution stays exact after applying ${-}\odot_{\mathfrak S}N$.

 (iii) For any complete, separated, left linear topological ring
$\mathfrak S$, all projective right $\mathfrak S$-contramodules
are $1$-strictly flat by the definition of $\Ctrtor*{\mathfrak S}--$;
hence all direct limits of projective contramodules are
$1$-strictly flat by Lemma~\ref{1-strictly-flat-lemma}(ii).
\end{proof}

 We will say that a right $\mathfrak S$-contramodule $\mathfrak Z$ is
\emph{contratensor-negligible} if $\mathfrak Z\odot_{\mathfrak S}N=0$
for all discrete left $\mathfrak S$-modules $N$.
 By the definition, any contratensor-negligible contramodule is flat.
 One of the versions of \emph{contramodule Nakayama lemma} tells that,
over a complete, separated left linear topological ring with a countable
base of neighborhoods of zero, any contratensor-negligible contramodule
vanishes~\cite[Lemma~6.14]{PR}.
 Without the assumption of a countable topology base, this form of
contramodule Nakayama lemma does \emph{not} hold, and an example of
a nonvanishing contratensor-negligible contramodule can be found
in~\cite[Remark~6.3]{Pcoun} (see also
Example~\ref{bad-flat-continued} below).
 Still, we are not aware of any example of a nonzero
contratensor-negligible contramodule over a topological ring $\mathfrak S$
over which all flat contramodules are $1$-strictly flat.

 It was shown in~\cite[Corollary~7.6]{PR} that, whenever a topological ring
$\mathfrak S$ has a countable base of neighborhoods of zero, the class of
all flat $\mathfrak S$-contramodules is deconstructible in
$\rcontra{\mathfrak S}$.
 Without assuming a countable topology base, our next aim in this section
is to show that if all flat $\mathfrak S$-contramodules are
$1$-strictly flat, then the class of all flat $\mathfrak S$-contramodules
is \emph{quasi-deconstructible modulo the class of contratensor-negligible
contramodules}, in the sense of the definition below.

 We need to use this notion of quasi-deconstructibility for flat
contramodules rather than the usual deconstructibility in our argument,
because the direct limit functors in the category of contramodules
are not exact.
 In particular, the direct limit of a well-ordered chain of subobjects
need not be a subobject, generally speaking.
 Because of this nonexactness issue, the usual construction
of filtrations runs into a problem which is resolved by introducing
\emph{quasi-filtrations}.

 The following lemma lists the properties of the class of all
contratensor-negligible contramodules.

\begin{lemma} \label{contratensor-negligible-lemma}
 Let\/ $\mathfrak S$ be a complete, separated, left linear topological ring.
\begin{enumerate}
\item The class of all contratensor-negligible contramodules is closed
under extensions, coproducts, and epimorphic images in\/
$\rcontra{\mathfrak S}$.
\item If\/ $\mathfrak D\twoheadrightarrow\mathfrak E$ is an epimorphism of
right\/ $\mathfrak S$-contramodules with a contratensor-negligible kernel,
and $N$ is a discrete left\/ $\mathfrak S$-module, then the induced map
of abelian groups $\mathfrak D\odot_{\mathfrak S}N\to
\mathfrak E\odot_{\mathfrak S}N$ is an isomorphism.
\end{enumerate}
 Now assume that all flat right\/ $\mathfrak S$-contramodules are\/
$1$-strictly flat.
\begin{enumerate}
\setcounter{enumi}{2}
\item If\/ $\mathfrak C\rightarrowtail\mathfrak D$ is an monomorphism of
right\/ $\mathfrak S$-contramodules with a con\-tra\-ten\-sor-negligible
cokernel, and $N$ is a discrete left\/ $\mathfrak S$-module, then
the induced map of abelian groups $\mathfrak C\odot_{\mathfrak S}N\to
\mathfrak D\odot_{\mathfrak S}N$ is an isomorphism.
\item The class of all contratensor-negligible contramodules is closed
under subobjects in\/ $\rcontra{\mathfrak S}$.
\item Let $(0\to\mathfrak F_i\to\mathfrak G_i\to\mathfrak H_i\to0
\mid i\in I)$ be a direct system of short exact sequences of right
$\mathfrak S$-contramodules, indexed by a direct poset~$I$, such
that all the contramodules\/ $\mathfrak F_i$, $\mathfrak G_i$,
$\mathfrak H_i$ are flat.
 Then the kernel of the induced morphism of direct limits\/
$\varinjlim^{\rcontra{\mathfrak S}}\mathfrak F_i\to
\varinjlim^{\rcontra{\mathfrak S}}\mathfrak G_i$ is
a contratensor-negligible $\mathfrak S$-contramodule.
\end{enumerate}
\end{lemma}

\begin{proof}
 (i), (ii) All the assertions hold because the functor
$-\odot_{\mathfrak S}N$ is right exact and preserves coproducts for
every $N\in\ldiscr{\mathfrak S}$.

 (iii) Since the cokernel $\mathfrak E=\mathfrak D/\mathfrak C$ is
contratensor-negligible, it is flat, hence by assumption $\mathfrak E$
is $1$-strictly flat.
 Now the assertion follows from the long exact
sequence~\eqref{contramodule-argument-exact-seq}.

 (iv) Let $\mathfrak Z$ be a contratensor-negligible contramodule
and $\mathfrak K\subseteq\mathfrak Z$ be a subcontramodule.
 By~(i), the quotient contramodule $\mathfrak Z/\mathfrak K$ is
contratensor-negligible; hence it is flat, and by assumption it
follows that $\mathfrak Z/\mathfrak K$ is $1$-strictly flat.
 Now one can see from the long exact
sequence~\eqref{contramodule-argument-exact-seq} associated with
the short exact sequence $0\to\mathfrak K\to\mathfrak Z\to
\mathfrak Z/\mathfrak K\to0$ that the contramodule $\mathfrak K$
is contratensor-negligible.

 (v) The direct limits are right exact in any cocomplete abelian
category; so we have a right exact sequence
$\varinjlim^{\rcontra{\mathfrak S}}\mathfrak F_i\to
\varinjlim^{\rcontra{\mathfrak S}}\mathfrak G_i\to
\varinjlim^{\rcontra{\mathfrak S}}\mathfrak H_i\to0$ in
$\rcontra{\mathfrak S}$.
 The direct limits of flat contramodules are flat, so both
the contramodules $\mathfrak G=
\varinjlim^{\rcontra{\mathfrak S}}\mathfrak G_i$ and
$\mathfrak H=\varinjlim^{\rcontra{\mathfrak S}}\mathfrak H_i$ are flat.
 By Lemma~\ref{1-strictly-flat-lemma}(iii) and by assumption, it follows
that the kernel $\mathfrak L$ of the epimorphism $\mathfrak G
\twoheadrightarrow\mathfrak H$ is also a ($1$-strictly) flat contramodule.
 Put $\mathfrak F=\varinjlim^{\rcontra{\mathfrak S}}\mathfrak F_i$;
we have to show that the kernel of the epimorphism $\mathfrak F
\twoheadrightarrow\mathfrak L$ is contratensor-negligible.

 Let $N$ be a discrete left $\mathfrak S$-module.
 Then the sequence of abelian groups $0\to\mathfrak F\odot_{\mathfrak S}N
\to\mathfrak G\odot_{\mathfrak S}N\to\mathfrak H\odot_{\mathfrak S}N\to0$
is exact, since it is the direct limit of the sequences of abelian
groups $0\to\mathfrak F_i\odot_{\mathfrak S}N\to
\mathfrak G_i\odot_{\mathfrak S}N\to\mathfrak H_i\odot_{\mathfrak S}N\to0$,
which are exact in view of the long exact
sequence~\eqref{contramodule-argument-exact-seq}.
 The sequence of abelian groups $0\to\mathfrak L\odot_{\mathfrak S}N
\to\mathfrak G\odot_{\mathfrak S}N\to\mathfrak H\odot_{\mathfrak S}N\to0$
is also exact by~\eqref{contramodule-argument-exact-seq}.

 Thus the epimorphism $\mathfrak F\twoheadrightarrow\mathfrak L$
induces an isomorphism $\mathfrak F\odot_{\mathfrak S}N
\cong\mathfrak L\odot_{\mathfrak S}N$.
 Finally, the long exact sequence~\eqref{contramodule-argument-exact-seq}
associated with the short exact sequence of contramodules
$0\to\mathfrak K\to\mathfrak F\to\mathfrak L\to0$ implies that
the contramodule $\mathfrak K$ is contratensor-negligible (because
the contramodule $\mathfrak L$ is $1$-strictly flat).
\end{proof}

\begin{example} \label{bad-flat-continued}
 The topological ring $\mathfrak S$ from Example~\ref{bad-flat}
exhibits all kinds of bad behavior.
 Over this topological ring, one has
$\Ctrtor*{\mathfrak S}{\mathfrak C}N=\Tor*{\mathfrak S}
{\mathfrak C}N$ for all $\mathfrak S$-contramodules $\mathfrak C$
and discrete $\mathfrak S$-modules $N$.
 Hence, for example, the flat (and contratensor-negligible)
$\mathfrak S$-contramodule $\mathfrak F=T^{-1}\mathfrak S/\mathfrak S$
is \emph{not} $1$-strictly flat, as $\Ctrtor1{\mathfrak S}{\mathfrak F}{N}
\cong\mathfrak S/\mathfrak S t\ne0$ for the discrete $\mathfrak S$-module
$N=\mathfrak S/\mathfrak S t$, where $t$ is any element from
$T\setminus\{1\}$.
 Furthermore, the $\mathfrak S$-contramodule $T^{-1}\mathfrak S$
is contratensor-negligible, but it contains all kinds of
subcontramodules which are not contratensor-negligible (e.g.,
$\mathfrak S\subset T^{-1}\mathfrak S$) and even not flat
as contramodules.

 Moreover, let $\mathfrak H$ be any non-flat module over the ring
$T^{-1}\mathfrak S$ (such as, e.g., $\mathfrak H=
T^{-1}\mathfrak S/T^{-1}\mathfrak S(x_1+x_2)$, where $x_1$ and $x_2$
are two different variables).
 Let us view $\mathfrak H$ as an $\mathfrak S$-module, and consequently
as an $\mathfrak S$-contramodule.
 Then the contramodule $\mathfrak H$ is not only contratensor-negligible,
but also $\infty$-strictly flat.
 Still, $\mathfrak H$ is not a flat $\mathfrak S$-module, hence it is
not a direct limit of projective $\mathfrak S$-contramodules.
\end{example}

 An injective morphism of right $\mathfrak S$-contramodules $f\colon
\mathfrak C\rightarrowtail\mathfrak D$ is said to be a \emph{contratensor
pure} monomorphism (\emph{c-pure} monomorphism for brevity) \cite{BP,BPS}
if the induced map of abelian groups $f\odot_{\mathfrak S}N\colon
\mathfrak C\odot_{\mathfrak S}N\to\mathfrak D\odot_{\mathfrak S}N$
is injective for all discrete left $\mathfrak S$-modules $N$.
 In this case, $\mathfrak C$ is said to be a \emph{c-pure subcontramodule}
of $\mathfrak D$, the short exact sequence $0\to\mathfrak C\to\mathfrak D
\to\mathfrak E\to0$ is called \emph{c-pure}, and the surjective morphism
$\mathfrak D\twoheadrightarrow\mathfrak E$ is said to be
a \emph{c-pure epimorphism}.

\begin{lemma} \label{purity-in-contramodules}
 Let\/ $\mathfrak S$ be a complete, separated, left linear topological
ring, and let\/ $0\to\mathfrak F\to\mathfrak G\to\mathfrak H\to0$ be
a c-pure short exact sequence of right\/ $\mathfrak S$-contramodules.
\begin{enumerate}
\item If the contramodule\/ $\mathfrak G$ is\/ $1$-strictly flat,
then so is the contramodule\/ $\mathfrak H$.
\item If the contramodule\/ $\mathfrak G$ is\/ flat, then so are both
the contramodules\/ $\mathfrak F$ and\/ $\mathfrak H$.
\end{enumerate}
\end{lemma}

\begin{proof}
 (i) Follows immediately from the long exact
sequence~\eqref{contramodule-argument-exact-seq}.

 (ii) Let $0\to K\to L\to M\to0$ be a short exact sequence of discrete
left $\mathfrak S$-modules.
 By the c-purity assumption, we have a short exact sequence of abelian
groups $0\to\mathfrak F\odot_{\mathfrak S}N\to
\mathfrak G\odot_{\mathfrak S}N\to\mathfrak H\odot_{\mathfrak S}N\to0$
for any discrete left $\mathfrak S$-module $N$, and in particular for
the discrete modules $K$, $L$, and $M$.
 On the other hand, we know that the short sequence of abelian groups
$0\to\mathfrak G\odot_{\mathfrak S}K\to\mathfrak G\odot_{\mathfrak S}L
\to\mathfrak G\odot_{\mathfrak S}N\to0$ is exact, while the short
sequences $\mathfrak F\odot_{\mathfrak S}K\to
\mathfrak F\odot_{\mathfrak S}L\to\mathfrak F\odot_{\mathfrak S}N\to0$
and $\mathfrak H\odot_{\mathfrak S}K\to\mathfrak H\odot_{\mathfrak S}L
\to\mathfrak H\odot_{\mathfrak S}N\to0$ are right exact.
 Now the map $\mathfrak F\odot_{\mathfrak S}K\to
\mathfrak F\odot_{\mathfrak S}L$ is injective, since so are the maps
$\mathfrak F\odot_{\mathfrak S}K\to
\mathfrak G\odot_{\mathfrak S}K\to\mathfrak G\odot_{\mathfrak S}L$;
hence the short sequence $0\to\mathfrak F\odot_{\mathfrak S}K\to
\mathfrak F\odot_{\mathfrak S}L\to\mathfrak F\odot_{\mathfrak S}N\to0$
is also exact.
 It remains to observe that the cokernel of a termwise injective
morphism of short exact sequences is a short exact sequence.
\end{proof}

\begin{lemma} \label{test-set-of-discrete-modules}
 Let\/ $\mathfrak S$ be a complete, separated left linear topological
ring and $S\subseteq\mathfrak S$ be a dense subring.
 Put\/ $\nu=\card S+\aleph_0$, and let\/ $\mu$ be the minimal cardinality
of a base of neighborhoods of zero in\/~$\mathfrak S$.
 Then there exists a set $\mathcal N$ of finitely generated discrete
left\/ $\mathfrak S$-modules such that\/ $\card\mathcal N\le\mu.\nu$
and\/ $\varinjlim\mathcal N=\ldiscr{\mathfrak S}$.
\end{lemma}

\begin{proof}
 Notice first of all that for every finitely generated discrete
left $\mathfrak S$-module $N$ one has $\card N\le\nu$.
 Indeed, if $x_1$,~\dots, $x_m\in N$ is a set of generators of $N$ and
$\mathfrak I_j$ is the annihilator of $x_j$ in $\mathfrak S$, then
$S+\mathfrak I_j=\mathfrak S$ for every $1\le j\le m$,
hence $N=Sx_1+\dotsb+Sx_m$.

 Let $\{\mathfrak I_b\mid b\in B\}$ be a set of open left ideals forming
a base of neighborhoods of zero in $\mathfrak S$ with $\card B=\mu$.
 Denote by $\mathcal N_0$ the set of all left $\mathfrak S$-modules
of the form $\mathfrak S/\mathfrak I_{b_1}\oplus\dotsb\oplus
\mathfrak S/\mathfrak I_{b_m}$, where $m<\omega$ and $b_1$,~\dots,
$b_m\in B$.
 Then $\card\mathcal N_0\le\mu+\aleph_0$.
 Furthermore, let $\mathcal N$ be the set of all quotient modules $N/K$,
where $N\in\mathcal N_0$ and $K\subseteq N$ is a finitely
generated submodule.
 Then $\mathcal N$ is a set of finitely generated discrete left
$\mathfrak S$-modules and $\card\mathcal N\le\mu.\nu$.

 Let $L$ be a discrete left $\mathfrak S$-module.
 For every element $x\in L$, choose an index $b_x\in B$ such that
$\mathfrak I_{b_x}x=0$.
 For every finite subset $X\subseteq L$, put
$N_{0,X}=\bigoplus_{x\in X}\mathfrak S/\mathfrak I_{b_x}$.
 Then there is a natural $\mathfrak S$-module map $h_X\colon
N_{0,X}\to L$ taking any element $(s_x+\mathfrak I_{b_x}\mid
x\in X)$ to the element $\sum_{x\in X}s_xx\in L$.
 Put $K_{0,X}=\ker h_X$.
 Note that we have $N_{0,X}\subseteq N_{0,Y}$ and $K_{0,X}\subseteq
K_{0,Y}$ for any two finite subsets $X\subseteq Y$ in~$L$.

 For any finite subset $U\subseteq K_{0,X}$, consider the submodule
$K_U=\sum_{u\in U}\mathfrak S u$ generated by $U$ in $K_{0,X}$.
 Denote by $P$ the set of all pairs $(X,U)$, where $X\subseteq L$
and $U\subseteq K_{0,X}$ are finite subsets.
 For each $p=(X,U)\in P$, put $N_p=N_{0,X}/K_U$.
 Denote the composition of the natural $\mathfrak S$-module maps
$N_p\twoheadrightarrow N_{0,X}/K_{0,X}\rightarrowtail L$ by~$f_p$.

 Define a partial order on the set $P$ by the rule that
$p=(X,U)\preceq q=(Y,V)$ if and only if $X\subseteq Y$
and $U\subseteq V$.
 Clearly, for any $p\preceq q\in P$ there exists a unique
$\mathfrak S$-module map $f_{qp}\colon N_p\to N_q$ forming
a commutative square diagram with the inclusion $N_{0,X}\hookrightarrow
N_{0,Y}$ and the natural surjections $N_{0,X}\twoheadrightarrow
N_p$ and $N_{0,Y}\twoheadrightarrow N_q$.
 Then $(N_p,f_{qp} \mid p\preceq q\in P)$ is a direct system of
$\mathfrak S$-modules from the class $\mathcal N$, and $\varinjlim N_p=
(L, f_p (p\in P))$, hence $L\in\varinjlim\mathcal N$.
\end{proof}

 Let $\gamma$~be a cardinal.
 An $\mathfrak S$-contramodule $\mathfrak C$ is said to be
\emph{$\le\gamma$-generated} if it is a quotient contramodule of
a free contramodule $[[X]]\mathfrak S$ with $\card X\le\gamma$.
 Notice that, for any $\le\gamma$-generated right
$\mathfrak S$-contramodule $\mathfrak C$ and any discrete left
$\mathfrak S$-module $N$ the group $\mathfrak C\odot_{\mathfrak S}N$
is an epimorphic image of the group $N[X]$
(in view of~\eqref{contratensor-with-free-contramodule}), so
$\card(\mathfrak C\odot_{\mathfrak S}N)\le \gamma\cdot\card N+\aleph_0$.

 Given an $\mathfrak S$-contramodule $\mathfrak D$ and a subset
$Y\subseteq\mathfrak D$, the \emph{subcontramodule of\/ $\mathfrak D$
generated by $Y$} can be constructed as the image of
the contramodule morphism $[[Y]]\mathfrak S\to\mathfrak D$ induced
by the inclusion map $Y\to\mathfrak D$.
 This is the unique minimal $\mathfrak S$-subcontramodule of
$\mathfrak D$ containing~$Y$.

\begin{proposition} \label{enochs-for-contra}
 Let\/ $\mathfrak S$ be a complete, separated, left linear topological ring
and $S\subseteq\mathfrak S$ be a dense subring.
 Put\/ $\nu=\card S+\aleph_0$, and let\/ $\mu$~be the minimal cardinality
of a base of neighborhoods of zero in\/~$\mathfrak S$.
 Then any nonzero right\/ $\mathfrak S$-contramodule has a nonzero\/
$\le\mu.\nu$-generated c-pure subcontramodule.
 Moreover, for any cardinal $\gamma\ge\mu.\nu$, any right
$\mathfrak S$-contramodule\/ $\mathfrak D$, and any subset $G\subseteq
\mathfrak D$ with $\card G\le\gamma$ there exists a c-pure\/
$\le\gamma$-generated subcontramodule $\mathfrak C\subseteq\mathfrak D$
such that $G\subseteq\mathfrak C$.
\end{proposition}

\begin{proof}
 Let $\mathfrak D$ be a nonzero $\mathfrak S$-contramodule.
 To reduce the first assertion of the proposition to the second one,
choose a nonzero element $c_0\in\mathfrak D$ and put $G=\{c_0\}$
and $\gamma=\mu.\nu$.
 To prove the second assertion, denote by $\mathfrak C_0$
the subcontramodule generated by $G$ in $\mathfrak D$.
 By the definition, $\mathfrak C_0$ is $\le\gamma$-generated.

 Proceeding by induction on the integers $i<\omega$, we will construct
a chain of subcontramodules $\mathfrak C_0\subseteq\mathfrak C_1
\subseteq\mathfrak C_2\subseteq\dotsb\subseteq\mathfrak D$.
 For every $i<\omega$, the contramodule $\mathfrak C_i$ will be
generated by at most $\gamma$ elements.

 Let $\mathcal N$ be the set of finitely generated discrete left
$\mathfrak S$-modules from Lemma~\ref{test-set-of-discrete-modules}.
 We observe that an injective morphism of $\mathfrak S$-contramodules
$\mathfrak C\rightarrowtail\mathfrak D$ is c-pure if and only if
the induced map of abelian groups $\mathfrak  C\odot_{\mathfrak S}N
\to\mathfrak D\odot_{\mathfrak S}N$ is injective \emph{for all
the discrete\/ $\mathfrak S$-modules $N$ from the set~$\mathcal N$}
(since both the functors $\mathfrak C\odot_{\mathfrak S}-$ and
$\mathfrak D\odot_{\mathfrak S}-$ preserve direct limits).

 To construct the contramodule $\mathfrak C_{i+1}$ for $i<\omega$,
we consider, for every $N\in\mathcal N$, the kernel of the map
$g_{i,N}\colon\mathfrak C_i\odot_{\mathfrak S}N\to
\mathfrak D\odot_{\mathfrak S}N$ induced by the inclusion
$\mathfrak C_i\rightarrowtail\mathfrak D$.
 We have $\card N\le\nu\le\gamma$ (see the proof of
Lemma~\ref{test-set-of-discrete-modules}) and $\mathfrak C_i$ is
$\le\gamma$-generated, so the cardinality of the kernel of $g_{i,N}$
cannot exceed~$\gamma$.
 Denote by $K_i$ the disjoint union of the kernels of $g_{i,N}$
taken over all $N\in\mathcal N$.
 Then $\card K_i\le\gamma$.

 Every element $k\in K_i$ is an element of the contratensor
product $\mathfrak C_i\odot_{\mathfrak S}N_k$ for some
$\mathfrak S$-module $N_k\in\mathcal N$, so it comes from
an element $\tilde k\in\mathfrak C_i\otimes_{\mathbb Z}N_k$.
 The fact that the image of $\tilde k$ vanishes in
$\mathfrak D\odot_{\mathfrak S}N_k$ is witnessed by an element $r_k$ of
the tensor product $[[\mathfrak D]]\mathfrak S\otimes_{\mathbb Z}N_k$
(see the definition of the contratensor product in
Section~\ref{contramodule-methods}).
 Write $r_k=\sum_{v=1}^{u_k} s_{k,v}\otimes_{\mathbb Z} b_{k,v}$,
where $u_k\ge0$ is an integer, $s_{k,v}\in[[\mathfrak D]]\mathfrak S$,
and $b_{k,v}\in N_k$.
 Every element $s_{k,v}$ is an infinite formal linear combination of
elements of $\mathfrak D$ with a zero-convergent family of coefficients
in $\mathfrak S$; denote by $D_{k,v}$ the set of all elements of
$\mathfrak D$ which enter into this formal linear combination with
a nonzero coefficient.
 Any zero-convergent family of nonzero elements in $\mathfrak S$
has cardinality~$\le\mu+\aleph_0$, so $\card D_{k,v}\le\mu+\aleph_0
\le\gamma$.

 Denote by $D_i\subseteq\mathfrak D$ the union of the sets $D_{k,v}$
taken over all $k\in K_i$ and $1\le v\le u_k$.
 Then $\card D_i\le\gamma$.
 Let $\mathfrak C_{i+1}\subseteq\mathfrak D$ be the subcontramodule
generated by $\mathfrak C_i$ and $D_i$.
 Since $\mathfrak C_i$ is $\le\gamma$-generated, so is $\mathfrak C_{i+1}$.
 We observe that, by construction, for any $\mathfrak S$-module
$N\in\mathcal N$, the kernel of the map
$\mathfrak C_i\odot_{\mathfrak S}N\to\mathfrak D\odot_{\mathfrak S}N$
induced by the inclusion $\mathfrak C_i\rightarrowtail\mathfrak D$
is equal to the kernel of the map $\mathfrak C_i\odot_{\mathfrak S}N
\to\mathfrak C_{i+1}\odot_{\mathfrak S}N$ induced by the inclusion
$\mathfrak C_i\rightarrowtail\mathfrak C_{i+1}$.
 Passing to a direct limit of a direct system of $\mathfrak S$-modules
from $\mathcal N$ and using Lemma~\ref{test-set-of-discrete-modules},
it follows that the same property holds for all discrete left
$\mathfrak S$-modules~$N$.

 Let $\mathfrak C$ be the $\mathfrak S$-subcontramodule of $\mathfrak D$
generated by $\bigcup_{i<\omega}\mathfrak C_i$.
 Since $\mathfrak C_i$ is $\le\gamma$-generated for all $i<\omega$,
the $\mathfrak S$-contramodule $\mathfrak C$ is also $\le\gamma$-generated.

 Finally, in order to show that the map $\mathfrak C\odot_{\mathfrak S}N
\to\mathfrak D\odot_{\mathfrak S}N$ is injective for any discrete
left $\mathfrak S$-module $N$, we will consider the direct limit of
the chain of contramodule morphisms $\mathfrak C_i\to\mathfrak C_{i+1}$
and use the fact that the contratensor product functor preserves
direct limits.
 We have a surjective (but possibly non-injective)
$\mathfrak S$-contramodule map $\varinjlim^{\rcontra{\mathfrak S}}_i
\mathfrak C_i\twoheadrightarrow\mathfrak C$.
 The composition $\varinjlim^{\rcontra{\mathfrak S}}_i
\mathfrak C_i\twoheadrightarrow\mathfrak C\rightarrowtail\mathfrak D$
is the direct limit of the inclusion maps $\mathfrak C_i\rightarrowtail
\mathfrak D$.
 For any discrete left $\mathfrak S$-module $N$, the induced map of
abelian groups $\varphi\colon \bigl(\varinjlim^{\rcontra{\mathfrak S}}_i
\mathfrak C_i\bigr)\odot_{\mathfrak S}N\to\mathfrak D
\odot_{\mathfrak S}N$ is the direct limit of the maps of abelian groups
$\mathfrak C_i\odot_{\mathfrak S}N\to\mathfrak D\odot_{\mathfrak S}N$
induced by the inclusions $\mathfrak C_i\rightarrowtail\mathfrak D$.
 Since the kernel of the map
$\mathfrak C_i\odot_{\mathfrak S}N\to\mathfrak D\odot_{\mathfrak S}N$
is equal to the kernel of the map $\mathfrak C_i\odot_{\mathfrak S}N
\to\mathfrak C_{i+1}\odot_{\mathfrak S}N$, the map~$\varphi$ is injective.

 On the other hand, the map~$\varphi$ is the composition of the induced
maps $\bigl(\varinjlim^{\rcontra{\mathfrak S}}_i
\mathfrak C_i\bigr)\odot_{\mathfrak S}N\to\mathfrak C
\odot_{\mathfrak S}N\to\mathfrak D\odot_{\mathfrak S}N$.
 Since the contramodule morphism $\varinjlim^{\rcontra{\mathfrak S}}_i
\mathfrak C_i\twoheadrightarrow\mathfrak C$ is surjective, so is
the induced map of abelian groups
$\bigl(\varinjlim^{\rcontra{\mathfrak S}}_i
\mathfrak C_i\bigr)\odot_{\mathfrak S}N\to
\mathfrak C\odot_{\mathfrak S}N$.
 It follows that the latter map is an isomorphism, and the map
$\mathfrak C\odot_{\mathfrak S}N\to\mathfrak D\odot_{\mathfrak S}N$
is injective, as desired.
\end{proof}

 Let $\mathsf B$ be a cocomplete abelian category.
 Let $\mathcal T\subseteq\mathsf B$ and $\mathcal N\subseteq \mathsf B$
be two classes of objects; we will call the elements of $\mathcal N$
\emph{negligible}.

 Let $(f_{ji}\colon F_i\to F_j\mid i\le j\le\alpha)$ be a direct system of
objects in $\mathsf B$ indexed by an ordinal~$\alpha$.
 The direct system $(F_i\mid i\le\alpha)$ is called a \emph{continuous chain}
if the natural morphism $\varinjlim_{i<j}^{\mathsf B}F_i\to F_j$
is an isomorphism for all limit ordinals $j\le\alpha$.
 A continuous chain $(F_i\mid i\le\alpha)$ is said to be
a \emph{$\mathcal T$-quasi-filtration modulo $\mathcal N$} if $F_0=0$ and,
for every $i<\alpha$, the kernel of the morphism $f_{i+1,i}\colon F_i\to
F_{i+1}$ is isomorphic to an element of $\mathcal N$, while the cokernel
of $f_{i+1,i}$ is isomorphic to an element of~$\mathcal T$.

 An object $F\in\mathsf B$ is said to be \emph{$\mathcal T$-quasi-filtered
modulo $\mathcal N$} if there exists an ordinal~$\alpha$ and
a $\mathcal T$-quasi-filtration modulo $\mathcal N$, \
$(F_i\mid i\le\alpha)$, such that $F\cong F_\alpha$.
 A class of objects $\mathcal F\subseteq\mathsf B$ is said to be
\emph{quasi-deconstructible modulo $\mathcal N$} if there exists
a set of objects $\mathcal T\subseteq\mathcal F$ such that all
the objects of $\mathcal F$ are $\mathcal T$-quasi-filtered
modulo~$\mathcal N$.

 A category $\mathsf B$ is said to be \emph{well-powered} if
(representatives of equivalence classes of) subobjects of any given object
form a set (rather than a proper class).

\begin{lemma} \label{quasi-deconstructible}
 Let\/ $\mathsf B$ be cocomplete well-powered abelian category,
and let $\mathcal F\subseteq\mathsf B$ be a class of objects closed
under direct limits.
 Let $\mathcal T\subseteq\mathcal F$ be a subclass such that every nonzero
object from $\mathcal F$ has a nonzero subobject belonging to $\mathcal T$
for which the corresponding quotient object belongs to $\mathcal F$.
 Let $\mathcal N\subseteq\mathsf B$ be a class of objects.
 Assume that, for any direct system of short exact sequences
$(0\to F_i\to G_i\to H_i\to0 \mid i\in I)$ in\/ $\mathsf B$ with
$G_i\in\mathcal F$ and $H_i\in\mathcal F$ for all $i\in I$, the kernel of
the induced morphism\/ $\varinjlim^{\mathsf B}F_i\to
\varinjlim^{\mathsf B}G_i$ belongs to $\mathcal N$.
 Then all the objects of $\mathcal F$ are $\mathcal T$-quasi-filtered
modulo $\mathcal N$.
\end{lemma}

\begin{proof}
 This is our version of~\cite[Lemma~4.14]{PR}.
 Let $F\in\mathcal F$ be an object.
 Choose a limit ordinal~$\alpha$ such that $F$ does not have a strictly
increasing chain of subobjects of length~$\alpha$.
 Proceeding by transfinite induction, we will construct
a $\mathcal T$-quasi-filtration $(f_{ji}\colon F_i\to F_j\mid
i\le j\le\alpha)$ modulo $\mathcal N$ and a cocone of morphisms
$(f_i\colon F_i\to F\mid i\le\alpha)$ such that $f_jf_{j,i}=f_i$
for all $i\le j$ and $f_\alpha\colon F_\alpha\to F$ is an isomorphism.
 The cokernel of the morphism~$f_i$ will belong to $\mathcal F$ and
the kernel of~$f_i$ will belong to $\mathcal N$ for all $i\le\alpha$.
 In fact, for successor ordinals~$i$, the morphism~$f_i$ will be
a monomorphism.

 Put $F_0=0$.
 On a successor step $i+1$, if the morphism~$f_i$ is an isomorphism,
put $F_{i+1}=F_i$, \,$f_{i+1}=f_i$, and $f_{i+1,i}=\mathrm{id}$.
 If $i=0$ or $i$~is a successor ordinal, then the morphism~$f_i$ is
a monomorphism.
 If $i$~is a limit ordinal, then the kernel of~$f_i$ belongs to~$\mathcal N$.
 In this case, if the morphism~$f_i$ is an epimorphism with a nonzero kernel,
we put $F_{i+1}=F$, \ $f_{i+1}=\mathrm{id}$, and $f_{i+1,i}=f_i$.

 Otherwise, the cokernel $\operatorname{coker}(f_i)$ is nonzero, so
there exists a nonzero subobject $t_i\colon T_i\rightarrowtail
\operatorname{coker}(f_i)$ with $T_i\in\mathcal T$ and 
$\operatorname{coker}(t_i)\in\mathcal F$.
 Let $f_{i+1}\colon F_{i+1}\to F$ be the pullback of the monomorphism~$t_i$
with respect to the epimorphism $F\twoheadrightarrow
\operatorname{coker}(f_i)$.
 Then $f_{i+1}$ is a monomorphism with $\operatorname{coker}(f_{i+1})
=\operatorname{coker}(t_i)$, and there exists a unique morphism
$f_{i+1,i}\colon F_i\to F_{i+1}$ with $f_{i+1}f_{i+1,i}=f_i$.
 The cokernel of $f_{i+1,i}$ is isomorphic to $T_i$, and the kernel of
$f_{i+1,i}$ is isomorphic to the kernel of~$f_i$ (so it belongs
to~$\mathcal N$).

 On a limit step~$j$, we put $F_j=\varinjlim_{i<j}^{\mathsf B} F_i$,
and let $f_j\colon F_j\to F$ be the unique morphism such that
$f_jf_{j,i}=f_i$ for all $i<j$.
 The cokernel of~$f_j$ is the direct limit of the cokernels of~$f_i$
taken over all $i<j$, so $\operatorname{coker}(f_j)\in\mathcal F$
as $\mathcal F$ is closed under direct limits.
 In the direct system of short exact sequences $(0\to F_i\to F\to
\operatorname{coker}(f_i)\to0\mid i<j)$, we have $F\in\mathcal F$ and
$\operatorname{coker}(f_i)\in\mathcal F$ for all $i<j$, hence
the kernel of the morphism $f_j\colon\varinjlim_{i<j}^{\mathsf B}F_i\to F$
belongs to~$\mathcal N$.

 Finally, since $F$ does not have an increasing chain of subobjects
of length~$\alpha$, there exists $j<\alpha$ such that $f_j\colon
F_j\to F$ is an epimorphism.
 Then $f_{j+1}\colon F_{j+1}\to F$ is an isomorphism, and so is
$f_\alpha\colon F_\alpha\to F$.
\end{proof}

\begin{corollary} \label{flat-contramodules-quasi-deconstructible}
 Let\/ $\mathfrak S$ be a complete, separated, left linear topological
ring such that all flat right\/ $\mathfrak S$-contramodules are\/
$1$-strictly flat.
 Then the class $\mathcal F$ of all flat\/ $\mathfrak S$-contramodules
is quasi-deconstructible modulo the class $\mathcal N$ of all
contratensor-negligible contramodules in the abelian category\/
$\rcontra{\mathfrak S}$.
\end{corollary}

\begin{proof}
 The class of flat contramodules is always closed under direct limits.
 By Corollary~\ref{flat-contra-implications-cor}(i), under our present
assumptions the class $\mathcal F$ is also closed under extensions
and kernels of epimorphisms in $\rcontra{\mathfrak S}$.

 Let $\mathfrak F$ be a flat $\mathfrak S$-contramodule.
 As above, we assume that $S\subset\mathfrak S$ is a dense subring and
put $\nu=\card S+\aleph_0$ (one can always take $S=\mathfrak S$).
 We also denote by $\mu$ the minimal cardinality of a base of neighborhoods
of zero in~$\mathfrak S$.
 By Proposition~\ref{enochs-for-contra}, there is a nonzero
$\le\nobreak\mu.\nu$-generated c-pure subcontramodule $\mathfrak C
\subseteq\mathfrak F$.
 By Lemma~\ref{purity-in-contramodules}, both
the $\mathfrak S$-contramodules $\mathfrak C$ and
$\mathfrak F/\mathfrak C$ are flat.

 Let $(0\to\mathfrak F_i\to\mathfrak G_i\to\mathfrak H_i\to0
\mid i\in I)$ be a direct system of short exact sequences in
$\rcontra{\mathfrak S}$ such that $\mathfrak G_i\in\mathcal F$ and
$\mathfrak H_i\in\mathcal F$ for all $i\in I$.
 Then we also have $\mathfrak F_i\in\mathcal F$.
 By Lemma~\ref{contratensor-negligible-lemma}(v), it follows that
the kernel of the morphism
$\varinjlim^{\rcontra{\mathfrak S}}\mathfrak F_i\to
\varinjlim^{\rcontra{\mathfrak S}}\mathfrak G_i$ is
contratensor-negligible.

 Thus Lemma~\ref{quasi-deconstructible} is applicable, and we can
conclude that all flat right $\mathfrak S$-contramodules are
$\mathcal T$-quasi-filtered modulo $\mathcal N$, where $\mathcal T$
is the set of (representatives of isomorphism classes) of
$\mu.\nu$-generated flat $\mathfrak S$-contramodules.
\end{proof}

 Finally we come to the main result of this section.

\begin{theorem} \label{C-GL1-implies-theorem}
 Let $R$ be a ring, $M$ be a module and\/ $\mathfrak S=\End{M_R}$ be
its endomorphism ring, endowed with the finite topology.
 Assume that condition (C-GL1) from Remark~\ref{govorov-lazard-for-contra}
holds for right\/ $\mathfrak S$-contramodules.
 Let $S\subseteq\mathfrak S$ be a dense subring;
put\/ $\nu=\card S+\aleph_0$. 
 Let\/ $\tau$ be the minimal cardinality of a set of generators of
the right $R$-module~$M$.
 Then all right $R$-modules from the class\/ $\varinjlim\Add M$
are filtered by\/ $\le\lambda$-generated modules from the same class,
where\/ $\lambda=\nu.\tau$.
\end{theorem}

\begin{proof}
 By Corollary~\ref{flat-contra-implications-cor}(iii), all flat right
$\mathfrak S$-contramodules are $1$-strictly flat under the assumptions
of the theorem.
 Hence Corollary~\ref{flat-contramodules-quasi-deconstructible} is
applicable, and all flat $\mathfrak S$-contramodules are
$\mathcal T$-quasi-filtered modulo contratensor-negligible contramodules,
where $\mathcal T$ is the set of all $\le\mu.\nu$-generated flat
$\mathfrak S$-contramodules (where $\mu$~is the minimal cardinality
of a base of neighborhoods of zero in~$\mathfrak S$).

 Let $G=\{x_i\in M\mid i<\tau\}$ be a set of generators of the right
$R$-module~$M$.
 Then the annihilators of finite subsets of $G$ form a base of
neighborhoods of zero in $\mathfrak S$ (since any finitely generated
$R$-submodule in $M$ is contained in the submodule generated by some
finite subset of~$G$).
 Hence we have $\mu\le\tau$ (in the trivial case when $\tau$~is
finite, the ring $\mathfrak S$ is discrete and $\mu=1$).

 Under (C-GL1), all contramodules from $\mathcal T$ belong to
$\varinjlim^{\rcontra{\mathfrak S}}
(\rcontra{\mathfrak S})_{\mathrm{proj}}$, i.e., they are direct limits
of projective $\mathfrak S$-contramodules.
 By Theorem~\ref{lim-Add}, it follows that the right $R$-module
$\mathfrak T\odot_{\mathfrak S}M$ belongs to $\varinjlim\Add M$
for all $\mathfrak T\in\mathcal T$.
 Denote by $\mathcal S$ the set of all right $R$-modules
$\mathfrak T\odot_{\mathfrak S}M$ with $\mathfrak T\in\mathcal T$;
so $\mathcal S\subseteq\varinjlim\Add M$.
 It is clear from~\eqref{contratensor-with-free-contramodule}
that all the $R$-modules from $\mathcal S$ are $\le\nu.\tau$-generated.
 We will show that all the $R$-modules from $\varinjlim\Add M$ are
$\mathcal S$-filtered.

 Let $N_R$ be a module from $\varinjlim\Add M$.
 By the other implication in Theorem~\ref{lim-Add}, there exists
a contramodule $\mathfrak F\in\varinjlim^{\rcontra{\mathfrak S}}
(\rcontra{\mathfrak S})_{\mathrm{proj}}$ such that
$N\cong\mathfrak F\odot_{\mathfrak S}M$.
 By construction, there exists a $\mathcal T$-quasi-filtration
modulo contratensor-negligible contramodules
$(f_{ji}\colon\mathfrak F_i\to\mathfrak F_j\mid i\le j\le\alpha)$
in $\rcontra{\mathfrak S}$ such that the contramodule $\mathfrak F_\alpha$
is isomorphic to~$\mathfrak F$.
 We will prove by transfinite induction that the induced map of right
$R$-modules $\mathfrak F_i\odot_{\mathfrak S}M\to\mathfrak F
\odot_{\mathfrak S}M\cong N$ is injective for all $i\le\alpha$.
 Furthermore, denoting the image of this map by $N_i\subseteq N$,
we will have $N_j=\bigcup_{i<j}N_i$ for all limit ordinals $j\le\alpha$,
and the quotient module $N_{i+1}/N_i$ will be isomorphic to a module
from $\mathcal S$ for all $i<\alpha$.

 Indeed, for a successor ordinal $i+1$, the kernel of the map
$f_{i+1,i}\colon\mathfrak F_i\to\mathfrak F_{i+1}$ is
contratensor-negligible, while the cokernel is isomorphic to
a contramodule $\mathfrak T_i\in\mathcal T$.
 Denote by $\mathfrak L_i$ the image of~$f_{i+1,i}$.
 Then, by Lemma~\ref{contratensor-negligible-lemma}(ii),
the map $\mathfrak F_i\odot_{\mathfrak S}M\to
\mathfrak L_i\odot_{\mathfrak S}M$ induced by the epimorphism
$\mathfrak F_i\twoheadrightarrow\mathfrak L_i$ is an isomorphism
of $R$-modules.
 Since $\mathfrak T_i$ is a $1$-strictly flat contramodule, it is clear
from the long exact sequence~\eqref{contramodule-argument-exact-seq}
that the map $\mathfrak L_i\odot_{\mathfrak S}M\to
\mathfrak F_{i+1}\odot_{\mathfrak S}M$ induced by the monomorphism
$\mathfrak L_i\rightarrowtail\mathfrak F_{i+1}$ is a monomorphism
of $R$-modules with the cokernel isomorphic to $\mathfrak T_i
\odot_{\mathfrak S}M$.
 Thus $f_{i+1,i}\odot_{\mathfrak S}M\colon\mathfrak F_i
\odot_{\mathfrak S}M\to\mathfrak F_{i+1}\odot_{\mathfrak S}M$
is an injective $R$-module map with the cokernel isomorphic
to a module from~$\mathcal S$.
 Finally, for a limit ordinal~$j\le\alpha$ we have $\varinjlim_{i<j}
(\mathfrak F_i\odot_{\mathfrak S}M)=\mathfrak F_j\odot_{\mathfrak S}M$,
since the functor $-\odot_{\mathfrak S}M$ preserves direct limits.
\end{proof}

\begin{remark}
 How much of a difference is there between the cardinality estimates
in Corollary~\ref{AddM-deconstr} or~\ref{deconstr}, on the one hand,
and in Theorem~\ref{C-GL1-implies-theorem}, on the other hand?
 Any complete, separated left linear topological ring can be obtained
as the endomorphism ring of a module, with the finite topology on
the endomorphism ring~\cite[Corollary~4.4]{PS4}.
 Let $\mathfrak S$ be a complete, separated left linear topological ring
and $S\subseteq\mathfrak S$ be a dense subring.
 Put $\kappa=\card\mathfrak S$ and $\nu=\card S$.
 How much bigger can be $\kappa$ as compared to~$\nu$\,?

 More generally, let $\mathfrak X$ be a Hausdorff topological space and
$X\subset\mathfrak X$ be a dense subset.
 Then the following map $\mathfrak X\to 2^{2^X}$ is injective.
 To every point $x\in\mathfrak X$, the set of all subsets in $X$ of
the form $U\cap X$, where $U$ is an open neighborhood of $x$ in
$\mathfrak X$, is assigned.
 Put $\kappa=\card\mathfrak X$ and $\nu=\card X$; we have shown that
$\kappa\le 2^{2^\nu}$.

 More precisely, suppose that every point of $\mathfrak X$ has a base
of open neighborhoods of the cardinality~$\le\mu$.
 Then essentially the same construction produces an injective map
from $\mathfrak X$ to the set of all subsets of the cardinality~$\le\mu$
in~$2^X$.
 Thus $\kappa\le 2^{\mu.\nu}$.
 Notice that, whenever $\mathfrak X$ is a topological abelian group with
a base of neighborhoods of zero formed by open subgroups, and $X\subset
\mathfrak X$ is a dense subgroup, an open subgroup in $\mathfrak X$ is
determined by its intersection with~$X$.
 Hence one has $\mu\le 2^\nu$.

 The following example shows that the $2^{2^\nu}$ boundary is sharp.
 Let $k$ be a finite or countable field.
 Consider the $\nu$-dimensional vector space $V=k^{(\nu)}$ over~$k$;
then $\card V=\nu$ (assuming $\nu$~is infinite).
 Let $\mathfrak V$ be the pro-finite-dimensional completion of
the vector space $V$; so $\mathfrak V=\varprojlim_{W\subset V}V/W$,
where $W$ ranges over all the vector subspaces of finite codimension
in~$V$.
 Endow $\mathfrak V$ with the projective limit (i.~e., completion)
topology; then $V$ is a dense vector subspace in~$\mathfrak V$.
 On the other hand, as an abstract vector space, $\mathfrak V$ is
naturally isomorphic to the double dual vector space to $V$, i.~e.,
$\mathfrak V\simeq (V^*)^*$.
 Thus $\dim_k\mathfrak V=\card\mathfrak V=2^{2^\nu}$.

 One can endow $\mathfrak V$ with the zero multiplication and adjoin
a unit formally, to make it a ring (or $k$-algebra) with unit.
 This produces a complete, separated left linear (in fact, commutative)
topological ring $\mathfrak S$ with a dense subring $S\subset\mathfrak S$
such that $\card S=\nu$ and $\card\mathfrak S=2^{2^\nu}$.
\end{remark}

\begin{remark}
 Throughout Sections~\ref{contramodule-methods}--\ref{deconstructibility}
we have only considered the \emph{finite topology} on the endomorphism
ring $\End{M_R}$, but in fact there is some flexibility about
the choice of an endomorphism ring topology in these results.
 A complete, separated left linear topology $\theta$ on the ring
$\mathfrak S=\End{M_R}$ is said to be \emph{suitable}
\cite[Section~8.2]{BPS} if the left $\mathfrak S$-module $M$ is
discrete with respect to $\theta$ and, for any set $X$, a family of
elements $(s_x\in\mathfrak S\mid x\in X)$ converges to zero in
the topology $\theta$ if and only if it does in the finite topology.
 Then it follows that the sum $\sum_{x\in X}s_x\in\mathfrak S$
(understood as the limit of finite partial sums) is the same in $\theta$
and in the finite topology.
 The finite topology is suitable; but generally speaking, a suitable
topology is finer (has more open left ideals) than the finite topology.
 For any suitable topology on $\mathfrak S$, the related monad structure
on the functor $X\longmapsto [[X]]\mathfrak S$ coincides with the one
for the finite topology; so the related categories of contramodules
are the same.
 All the results of Sections~\ref{contramodule-methods}--\ref{contra-versus}
and~\ref{gabrieltopol}--\ref{deconstructibility} remain valid with
the finite topology on $\End{M_R}$ replaced with any suitable topology.

 There are several constructions of suitable topologies known for
the endomorphism ring of an arbitrary module $M$.
 In addition to the finite one, there is also the \emph{weakly finite}
topology~\cite[Theorem~9.9]{PS}, \cite[Example~3.10(2)]{PS4},
\cite[Example~2.2(2)]{BP}, \cite[Section~8.2]{BPS} and
the \emph{$M$-small topology}~\cite[Example~2.2(3)]{BP}.
 A base of neighborhoods of zero in the weakly finite (or ``small'')
topology is formed by the annihilators of those submodules of $M$
which are small as abstract $R$-modules (in the sense of
Section~\ref{versus}), while in the $M$-small topology these are
the annihilators of so-called $M$-small submodules.
 The weakly finite and $M$-small topologies are also suitable.
 For a self-small module, the discrete topology on the endomorphism
ring is suitable~\cite[Example~3.10(5)]{PS4}.
 One of the potential advantages of these alternative topologies
is that they may have a countable base of neighborhoods of zero when
the finite topology has not.
 For example, let $M=\sum_{i<\omega}M_i$ be a sum of a countable family
its submodules $M_i$ such that the $R$-module $M_i$ is small for
every $i<\omega$.
 Then the weakly finite topology on $\End{M_R}$ has a countable base
of neighborhoods of zero~\cite[Lemma~8.5]{BPS}, while the finite
topology on the endomorphism ring of a small module may be
uncountably based (use the example from~\cite[Example~3.10(3)]{PS4}).
\end{remark}

\medskip

\section{Open problems}\label{problems}

\medskip
{\it Problem 1}: \medskip  
Does the equality $\varinjlim \add \mathcal D = \varinjlim \Add \mathcal D$ hold for any class of modules $\mathcal D$? In particular, does $\varinjlim \add M = \varinjlim \Add M$ for any module $M$? 

See Corollary \ref{selfsmall}, Lemma \ref{injectlim}, Propositions \ref{eta-selfsmall}, \ref{prufer-inverted-element}, \ref{prufer-inverted-subset}, and Corollaries \ref{project} and \ref{gabriel-cor} for some positive answers.

Does $\varinjlim \add M = \varinjlim \Add M$ hold for any tilting module $M$? See Corollary \ref{artinfg} and Theorem \ref{dedekind} for partial positive answers. 

\medskip
{\it Problem 2}: Assume that $\mathcal C$ is a deconstructible class of modules. Is $\mathcal L = \varinjlim \mathcal C$ also deconstructible? 

Lemma \ref{fact} gives a positive answer in the particular case when $\mathcal C$ is closed under homomorphic images (there, the $\kappa$-deconstructibility of $\mathcal C$ even implies the $\kappa$-deconstructibility of $\mathcal L$). For other positive cases, see Corollaries \ref{deconstr} and~\ref{AddM-deconstr}, and Theorem~\ref{C-GL1-implies-theorem}. 

As another case, consider $\mathcal C = \mbox{Filt}(\mathcal S)$ where $\mathcal S$ is a set closed under direct summands, extensions, $R \in \mathcal S$, and $\mathcal S$ consists of FP$_2$-modules -- see Lemma \ref{lenzing}(ii). Then $\mathcal C$ is clearly $\aleph_0$-deconstructible, and $\mathcal L$ is $\kappa^+$-deconstructible for $\kappa = \card R + \aleph_0$. In this case $\kappa$ cannot be taken smaller in general, as seen on the particular case when $R$ is a PID and $\mathcal S$ is the set of all free modules of finite rank: then $\mathcal C$ is the class of all free modules, $\mathcal L$ the class of all torsion-free modules, and $\{ 0, Q \}$ is the only $\mathcal L$-filtration of the quotient field $Q$ of $R$.

\medskip
{\it Problem 3}: Is the trivial necessary condition of being closed under direct summands also sufficient for the class $\mathcal L = \varinjlim \mathcal C$ to be closed under direct limits? Cf.\ Examples \ref{indec} and \ref{Bergman}.

\medskip
{\it Problem 4}: Does $\widetilde{\Add T}$ equal $\overline{\Add T}$ for any (infinitely generated) tilting module? See Corollary \ref{artinfg} and Theorem \ref{dedekind} for some positive cases. 

\begin{acknowledgment}
We are grateful to Michal Hrbek and Jan \v S\v tov\'\i\v cek for very helpful consultations, and to Gene Abrams for providing us with a copy of the paper \cite{ON}. We also thank the referee for careful reading of the text and many helpful comments.  
\end{acknowledgment}


\begin{thebibliography}{99} 

\bibitem{AR}
J. Ad\'amek, J. Rosicky, \textit{Locally Presentable and Accessible
Categories}, Cambridge University Press, 1994. 
\bibitem{AF}
F.W. Anderson, K.R. Fuller, \textit{Rings and Categories of Modules}, 2nd ed., GTM 13, Springer-Verlag, N.Y. 1992.
\bibitem{AT}
L. Angeleri H\" ugel, J. Trlifaj, \textit{Direct limits of modules of finite projective dimension}, Rings, Modules, Algebras, and Abelian Groups, LNPAM 236, M.Dekker, New York 2004, 27-44. 
\bibitem{AST}
L. Angeleri H\" ugel, J. \v Saroch, J. Trlifaj, \textit{Approximations and Mittag-Leffler conditions - the applications}, Israel J. Math. 226(2018), 757-780.
\bibitem{Az}
G. Azumaya, \textit{Locally split submodules and modules with perfect endomorphism rings}, Noncommutative ring theory, LNM 1448, Springer, Berlin 1990, 1-6.
\bibitem{BP}
S. Bazzoni, L. Positselski, \textit{Covers and direct limits:
a contramodule-based approach}, Math.\ Z. 299(2021), 1-52.
\bibitem{BPS}
S. Bazzoni, L. Positselski, J. \v S\v tov\'\i\v cek, \textit{Projective covers of flat contramodules}, Int.\ Math.\ Res.\ Not.\ IMRN (2021),
DOI:10.1093/imrn/rnab202.
\bibitem{B} 
G.M. Bergman, \textit{Every module is an inverse limit of injectives}, Proc. Amer. Math. Soc. 141(2013), 1177-1183.
\bibitem{Bor}
M. Boraty\'nski, \textit{A change of rings theorem and the Artin-Rees
property}, Proc.\ Amer.\ Math.\ Soc.\ 53(1975), 307-310.
\bibitem{C} 
L. Claborn, \textit{Dedekind domains: overrings and semiprime elements}, Pacific J.\ Math. 15(1965), 289-304.
\bibitem{deM} 
G. De Marco, \textit{Projectivity of pure ideals}, Rend.\ Sem.\ Mat.\ Univ.\ Padova 68(1983), 423-430.
\bibitem{El} 
R. El Bashir, \textit{Covers and directed colimits}, Algebras Repres.\ Theory 9(2006), 423-430.
\bibitem{D}
S.E. Dickson, \textit{A torsion theory for Abelian categories},
Trans.\ Amer.\ Math.\ Soc.\ 121(1966), 223-235.
\bibitem{Dr}
A. Dress, \textit{On the decomposition of modules},
Bull.\ Amer.\ Math.\ Soc.\ 75(1969), 984-986.
\bibitem{FS}
L. Fuchs, L. Salce, \textit{Modules over Non-Noetherian Domains}, MSM 84, AMS, Providence 2001.
\bibitem{GZ}
P. Gabriel, M. Zisman, \textit{Calculus of Fractions and Homotopy Theory},
Springer, Berlin, 1967.
\bibitem{GS} 
R. G\"{o}bel, S. Shelah, \textit{Indecomposable almost free modules -- the local case}, Can.\ J.\ Math. 50(1998), 719-738.
\bibitem{GT}
R. G\"{o}bel, J. Trlifaj, \textit{Approximations and Endomorphism Algebras of Modules}, 2nd rev.\ ext.\ ed., GEM 41, W.\ de Gruyter, Berlin 2012.
\bibitem{GG} 
J.L. G\'{o}mez Pardo, P.A. Guil Asensio, \textit{Big direct sums of copies of a module have well behaved indecomposable decompositions}, J.\ Algebra 232(2000), 86-93.
\bibitem{G}
K.R. Goodearl, \textit{Von Neumann Regular Rings}, 2nd ed., Krieger Publ. Co., Malabar 1991.
\bibitem{GW}
K.R. Goodearl, R.B. Warfield, Jr. \textit{An Introduction to Noncommutative Noetherian Rings}, 2nd ed., LMSST 61, LMS,  Cambridge 2004.
\bibitem{Gov}
V.E. Govorov, \textit{On flat modules} (Russian), Sibir.\ Mat. Zh.\
6(1965), 300-304.
\bibitem{HP1}
D. Herbera, P. P\v r\'\i hoda, \textit{Reconstructing projective modules from its trace ideal}, J. Algebra 416(2014), 25-57.
\bibitem{HJ} 
H. Holm, P. J{\o}rgensen, \textit{Covers, precovers, and purity}, Illinois J.\ Math. 52(2008), 691-703.
\bibitem{H}
M. Hrbek, \textit{One-tilting classes and modules over commutative rings},
J. Algebra 462(2016), 1-22.
\bibitem{L}
J. Lambek, \textit{Lectures on Rings and Modules}, 3rd ed., AMS Chelsea Publ., Providence 2009.
\bibitem{La0}
D. Lazard, \textit{Autour de la platitude}, Bull.\ Soc.\ Math.\ France 97(1969),
81-128.
\bibitem{La}
D. Lazard, \textit{Libert\' e de gros modules projectifs}, J. Algebra 31(1974), 437-451.
\bibitem{Len}
{H. Lenzing}, \textit{Homological transfer from finitely presented to infinite modules}, in Lect.\ Notes Math.\ \textbf{1006}, Springer, New
York 1983, 734-761.
\bibitem{GPR}
W. Wm. McGovern, G. Puninski, P. Rothmaler, \textit{When every projective module is a direct sum of finitely generated modules}, J. Algebra 315(2007), 454-481.
\bibitem{MO-Mohan}
 MathOverflow User Mohan
 (\texttt{https://mathoverflow.net/users/9502/mohan}).
 Artin--Rees lemma for multiplicative subsets?
 An answer to the MathOverlow question
\texttt{https://mathoverflow.net/q/397373}.
\bibitem{ON}
J. D. O'Neill, \textit{Examples of non-finitely generated projective modules}, Methods in Module Theory, LNPAM 140, M.Dekker, New York 1993, 271-278.
\bibitem{O}
B. Osofsky, \textit{Rings all of whose finitely generated modules are injective}, Pacific J.\ Math. 15(1964), 645-650.
\bibitem{Pweak}
L. Positselski, \textit{Weakly curved A${}_\infty$-algebras over
a topological local ring}, M\'em.\ Soc.\ Math.\ France \textbf{159}, 2018.
\bibitem{Pcosh}
L. Positselski, \textit{Contraherent cosheaves}, arXiv:1209.2995v6.
\bibitem{Pcoun}
L. Positselski, \textit{Flat ring epimorphisms of countable type},
Glasgow Math.\ J. 62(2020), 383-439.
\bibitem{Pproperf}
L. Positselski, \textit{Contramodules over pro-perfect topological rings},
Forum Math. 34(2022), 1-39.
\bibitem{PR}
L. Positselski, J. Rosick\'y, \textit{Covers, envelopes, and cotorsion 
theories in locally presentable abelian categories and contramodule
categories}, J. Algebra 483(2017), 83-128.
\bibitem{PS}
L. Positselski, J. \v S\v tov\'\i\v cek, \textit{The tilting-cotilting
correspondence}, Int.\ Math.\ Res.\ Not.\ IMRN 2021, 189-274.
\bibitem{PS4}
L. Positselski, J. \v S\v tov\'\i\v cek, \textit{Topologically
semisimple and topologically perfect topological rings},
arXiv:1909.12203v4, to appear in Publ.\ Mat.
\bibitem{PT}
D. Posp\'\i\v sil, J. Trlifaj, \textit{Tilting for regular rings of Krull dimension $2$}, J. Algebra 336(2011), 184-199.
\bibitem{Sa}
J. \v Saroch, \textit{Approximations and Mittag-Leffler conditions - the tools}, Israel J. Math. 226(2018), 737-756.
\bibitem{Sm}
P.F. Smith, \textit{The Artin-Rees property}, in S\'eminaire d'Alg\`ebre
Paul Dubreil et Marie-Paule Malliavin, 34eme Annee, Lect.\ Notes Math.\
\textbf{924}, Springer, Berlin 1982, 197-240.
\bibitem{S}
B. Stenstr\" om, \textit{Rings of Quotients}, GMW 217, Springer-Verlag, Berlin 1975.
\bibitem{TPad} 
J. Trlifaj, \textit{Steady rings may contain large sets of orthogonal idempotents}, in Abelian Groups and Modules, MIA 343, Kluwer, Dordrecht 1995, 467-473.
\bibitem{Tams} 
J. Trlifaj, \textit{Whitehead test modules}, Trans. Amer. Math. Soc. 348(1996), 1521-1554.
\bibitem{T} 
J. Trlifaj, \textit{The Dual Baer Criterion for non-perfect rings}, Forum Math. 32(2020), 663-672.
\bibitem{V}
W. V. Vasconcelos, \textit{Finiteness in projective ideals}, J. Algebra 25(1973), 269-278.
\bibitem{Y}
A. Yekutieli, \textit{Flatness and completion revisited},
Algebras Represent. Theory 21(2018), 717-736.
\bibitem{Z}
W. Zimmermann, \textit{On locally pure-injective modules}, J. Pure Appl. Algebra 166(2002), 337–357.
\end{thebibliography}
\end{document}